\documentclass[12pt,twoside]{amsart}
\usepackage[margin=3cm]{geometry}
\usepackage[colorlinks=false]{hyperref}
\usepackage[english]{babel}
\usepackage{graphicx,titling}
\usepackage{float}
\usepackage{amsmath,amsfonts,amssymb,amsthm}
\usepackage{lipsum}
\usepackage[T1]{fontenc}
\usepackage{fourier}
\usepackage{color}
\usepackage{esint}
\usepackage{caption}
\usepackage{ccicons}

\makeatletter
\def\blfootnote{\xdef\@thefnmark{}\@footnotetext}
\makeatother

\newcommand\ccnote{
    \blfootnote{\copyright\,\, Pierre Germain, Simon L. Rydin Myerson, and Daniel Pezzi}
    \blfootnote{\ccLogo\, \ccAttribution\,\, Licensed under a \href{https://creativecommons.org/licenses/by/4.0/}{Creative Commons Attribution License (CC-BY)}.}
}

\usepackage[export]{adjustbox}
\numberwithin{equation}{section}
\usepackage{setspace}

\renewcommand{\leq}{\leqslant}

\renewcommand{\geq}{\geqslant}
\renewcommand{\mathbb}{\varmathbb}
\usepackage{fancyhdr}
\usepackage{amsmath,amsthm,amssymb,amscd,color, xcolor,mathtools,hyperref,fontawesome}
\usepackage[capitalize]{cleveref}
\usepackage{StyleSheet}

\newtheorem{thm}{Theorem}[section]
 \newtheorem{cor}[thm]{Corollary}
 \newtheorem{lem}[thm]{Lemma}
 \newtheorem{prop}[thm]{Proposition}
 \newtheorem{conj}{Conjecture}

 \theoremstyle{definition}
 \newtheorem{df}[thm]{Definition}
 \theoremstyle{remark}
 \newtheorem{rem}[thm]{Remark}
 
 \numberwithin{equation}{section}
\usepackage{bbm,tikz,pgfplots,tikz-3dplot}
\usepackage[untagged, highstructure]{accessibility}

\usetikzlibrary{arrows,shapes,positioning}
\usetikzlibrary{intersections}
\usetikzlibrary{decorations.pathreplacing,calligraphy}
\usepgfplotslibrary{fillbetween}

\newcommand{\dd}{\, {\rm d}}
\newcommand{\rtA}{L}
\newcommand{\ad}{\mathrm{ad}}
\newcommand{\GL}{\mathrm{GL}}

\newcommand{\say}[1]{``#1"}

\usepackage[multiuser]{fixme}

\FXRegisterAuthor{PG}{PGlong}{PG}
\FXRegisterAuthor{SM}{SMlong}{SM}
\FXRegisterAuthor{DP}{DPlong}{DP}

\address{Pierre Germain, Department of Mathematics, Universitaet Wien, Universit
\"atsring 1, Vienna, AT 1010}
\email{pierre.germain@univie.ac.at}
\medskip
\address{Simon L. Rydin Myerson, Department of Mathematical Sciences, Chalmers University of Technology and the University of Gothenburg, 412 96 Gothenburg, Sweden}
\email{myerson@chalmers.se}
\medskip
\address{Daniel Pezzi, Department of Mathematics, Krieger Hall, Homewood Campus, Johns Hopkins University, 3400 N. Charles Street Baltimore, MD 21218, United States}
\email{dpezzi1@jhu.edu}

\begin{document}

\thispagestyle{empty}

\begin{minipage}{0.28\textwidth}
\begin{figure}[H]
\includegraphics[width=2.5cm,height=2.5cm,left]{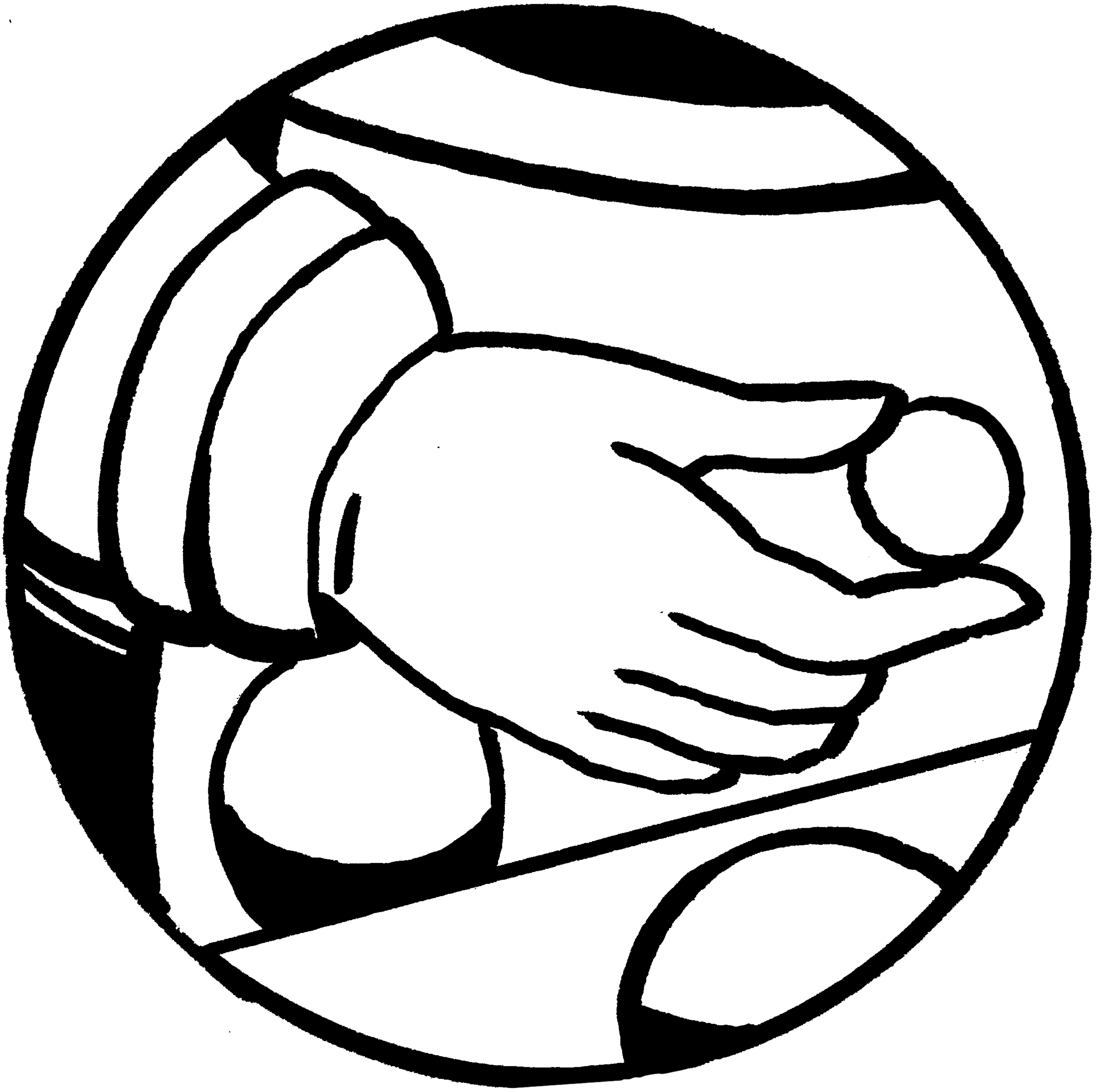}
\end{figure}
\end{minipage}
\begin{minipage}{0.7\textwidth} 
\begin{flushright}
Ars Inveniendi Analytica (2026), Paper No. 6, 49 pp.
\\
DOI 10.15781/fwja8790
\\
ISSN: 2769-8505
\end{flushright}
\end{minipage}

\ccnote

\vspace{1cm}


\begin{center}
\begin{huge}
\textit{Bounds for spectral projectors}

\textit{on the three-dimensional torus}

\end{huge}
\end{center}

\vspace{1cm}


\begin{minipage}[t]{.28\textwidth}
\begin{center}
{\large{\bf{Pierre Germain}}} \\
\vskip0.15cm
\footnotesize{Universitaet Wien}
\end{center}
\end{minipage}
\hfill
\noindent
\begin{minipage}[t]{.28\textwidth}
\begin{center}
{\large{\bf{Simon L. Rydin Myerson}}} \\
\vskip0.15cm
\footnotesize{Chalmers University of Technology \& University of Gothenburg}
\end{center}
\end{minipage}
\hfill
\noindent
\begin{minipage}[t]{.28\textwidth}
\begin{center}
{\large{\bf{Daniel Pezzi}}} \\
\vskip0.15cm
\footnotesize{Johns Hopkins University} 
\end{center}
\end{minipage}

\vspace{1cm}


\begin{center}
\noindent \em{Communicated by Xavier Tolsa}
\end{center}
\vspace{1cm}


\noindent \textbf{Abstract.}
\textit{We study $L^2$ to $L^p$ operator norms of spectral projectors for the Euclidean Laplacian on the torus in the case where the spectral window is narrow. With a window of constant size this is a classical result of Sogge; in the small-window limit we are left with $L^p$ norms of eigenfunctions of the Laplacian, as considered for instance by Bourgain. For the three-dimensional torus we prove new cases of a previous conjecture of the first two authors concerning the size of these norms; we also refine certain prior results to remove $\epsilon$-losses in all dimensions. We use methods from number theory: the geometry of numbers, the circle method and exponential sum bounds due to Guo. We complement these techniques with height splitting and a bilinear argument to prove sharp results. We exposit on the various techniques used and their limitations.}
\vskip0.3cm

\vspace{0.5cm}

\section{Introduction}

Our starting point is the Laplacian $\Delta = \frac{\partial^2}{\partial x_1^2}+\frac{\partial^2}{\partial x_2^2}+\frac{\partial^2}{\partial x_3^2}$ on the torus $\mathbb{T}^3 = \mathbb{R}^3 / \mathbb{Z}^3$. 

For a positive definite quadratic form
$$
Q(x) = x^T A x, \qquad \mbox{with $A$ symmetric},
$$
this example can be generalized to operators of the form $\Delta_Q = Q(\partial_1,\partial_2,\partial_3)$, still defined on $\mathbb{T}^3 = \mathbb{R}^3 / \mathbb{Z}^3$ (equivalently, we could have kept the Laplacian fixed and considered general tori of the form $\mathbb{R}^3 / \Lambda$, where $\Lambda$ is a lattice of $\mathbb{R}^3$ of full rank).

{We refer to the case $A = \operatorname{Id}$, $Q({x}) = \|{x}\|^2$, or $\Lambda = \mathbb{Z}^3$, which are all equivalent, as the \textit{square torus}.
\subsection{The conjectures in dimension 3}
We consider \(\Delta_Q = Q(\partial_1,\partial_2,\partial_3)\) on \(\mathbb{T}^3 = \mathbb{R}^3 / \mathbb{Z}^3\), and define by functional calculus the spectral projectors
$$
P_{\lambda,\delta} = \mathbf{1}_{(-1,1)} \left( \frac{\sqrt{-\Delta_Q} - 2\pi \lambda}{{2\pi} \delta} \right),
$$
where \(\delta \in (0,1)\) is small and \(\lambda>1\) is large (the factors ${2\pi}$ are added in the above definition to make subsequent notations simpler). In other words,
$$
P_{\lambda,\delta} = \Phi_{\lambda,\delta} * \cdot \quad \mbox{with} \quad \Phi_{\lambda,\delta} = \sum_{\substack{k \in \mathbb{Z}^3 \\ | \sqrt{Q(k)} -\lambda| < \delta}} e^{2\pi i k \cdot x}.
$$
By duality and $TT^*$,
\begin{equation}
\label{mallard}
\|  P_{\lambda,\delta} \|_{L^{2} \to L^p}^2 = \|P_{\lambda,\delta} \|_{L^{p'} \to L^p} = \| P_{\lambda,\delta} \|_{L^{p'} \to L^2}^2.
\end{equation}
The following conjecture appeared in~\cite{GermainMyerson1}.

\begin{conj} \label{conjA} If $\lambda^{-1} \leq \delta \leq 1$ and $p \geq 2$, then we have
\begin{equation*}
\| P_{\lambda,\delta} \|_{L^{2} \to L^p} \lesssim_{Q,p} (\lambda \delta)^{ \frac{1}{2} - \frac{1}{p}} + \lambda^{1 - \frac{3}{p}} \delta^{\frac 12},
\end{equation*}
except in certain boundary cases (see below). In other words
$$
\| P_{\lambda,\delta} \|_{L^{2} \to L^p} \lesssim_{Q,p}
\begin{cases}
\lambda^{1-\frac 3p} \delta^{\frac 12} &
\begin{aligned}
   & \text{if }4\leq p \leq 6\text{ and }\delta \geq \lambda^{2-\frac p2};
\\&
\text{or }6 \leq p\leq \infty\text{ and }\delta \geq \lambda^{-1},
\end{aligned}
\\
(\lambda \delta)^{ \frac{1}{2} - \frac{1}{p}} &
\begin{aligned}
   & 
\text{if }2\leq p \leq 4\text{ and }\delta \geq \lambda^{-1};
\\&\text{or }4\leq p \leq 6\text{ and }\lambda^{-1} \leq \delta \leq \lambda^{2-\frac p2}.\end{aligned}
\end{cases}
$$
For some values of $Q,p,\lambda,\delta$ a subpolynomial loss (``$\epsilon$- loss") should be added to the above conjecture. This is at least the case on the square torus for $p=\infty$, $\delta = \lambda^{-1}$. The statement becomes then
$$
\| P_{\lambda,\delta} \|_{L^{2} \to L^p} \lesssim_\eps \lambda^\epsilon \left[ (\lambda \delta)^{ \frac{1}{2} - \frac{1}{p}} + \lambda^{1 - \frac{3}{p}} \delta^{\frac 12} \right].
$$
If this holds, we say that the conjecture holds with $\epsilon$-loss.
\end{conj}

The conjecture bounds the operator norm by a sum of two terms, namely $(\lambda \delta)^{ \frac{1}{2} - \frac{1}{p}}$ and $\lambda^{1 - \frac{3}{p}} \delta^{\frac 12}$. This has the following geometric interpretation, see \cite{GermainMyerson1} for more details. The term $(\lambda \delta)^{ \frac{1}{2} - \frac{1}{p}}$ corresponds to the $L^p$ norm of a quasimode\footnote{By definition, a quasimode $f$ is such that $P_{\lambda,\delta} f = f$, it is thus nearly an eigenfunction.} normalized in $L^2$ which concentrates on a closed geodesic; and the term $\lambda^{1 - \frac{3}{p}} \delta^{\frac 12}$ corresponds to the $L^p$ norm of a quasimode normalized in $L^2$ which concentrates at a point. Therefore, we refer to the regions $(\lambda \delta)^{ \frac{1}{2} - \frac{1}{p}} < \lambda^{1 - \frac{3}{p}} \delta^{\frac 12}$ and $(\lambda \delta)^{ \frac{1}{2} - \frac{1}{p}} > \lambda^{1 - \frac{3}{p}} \delta^{\frac 12}$ as the \textit{point focusing regime} and \textit{geodesic-focusing regime} respectively.

Let $\lambda_0>1$. Those eigenvalues $\lambda^2$ of the Laplacian $-\Delta$ onto the square torus which have  size $\sim \lambda_0$ have spacing $\sim \frac{1}{\lambda_0}$. Therefore, if $\lambda^2$ is an eigenvalue of the Laplacian and $\delta = \frac{1}{10} \lambda$, the spectral projector $P_{\lambda,\delta}$ is simply the orthogonal projection on the eigenspace associated to the eigenvalue ${4 \pi^2} \lambda^2$. Then the above conjecture reduces to the earlier conjecture of Bourgain \cite{Bourgain1,BourgainDemeter3}.

\begin{conj} \label{conjB} If $u$ is an eigenfunction of the Laplacian $-\Delta$ on $\mathbb{T}^3$ with eigenvalue $\lambda^2$, then
$$
\| u \|_{L^p} \lesssim_\eps
\begin{cases}
\lambda^{\frac{1}{2} - \frac{3}{p}+\epsilon} \| u \|_{L^2}  & \mbox{if $p \geq 6$} \\
\lambda^\epsilon  \| u \|_{L^2} & \mbox{if $p \leq 6$}
\end{cases}
$$
(The subpolynomial loss $\lambda^\epsilon$ might be superfluous if $p <\infty$).
\end{conj}

\subsection{Known results in dimension $3$}

Earlier results on Conjecture \ref{conjA} are as follows:
\begin{itemize}
\item For $\delta =1$, the conjecture corresponds to the fundamental result of Sogge~\cite{Sogge}, which holds on any compact Riemannian manifold without boundary.
\item For $p=\infty$ and on the square torus, the conjecture with $\epsilon$-loss is a consequence of the elementary estimate on the number of lattice points on the sphere $\# ( \mathbb{Z}^3 \cap \lambda \mathbb{S}^2 )\lesssim_\epsilon \lambda^{1+\epsilon}$. If $\delta > \lambda^{- \frac{39}{48}}$, the conjecture is satisfied without loss by Chamizo-Iwaniec \cite{ChamizoIwaniec} (Theorem 1.1).
\item For $p=\infty$ and a general torus, the conjecture is satisfied if $\delta > \lambda^{-\frac{85}{158}}$ by the main theorem of Guo \cite{Guo}.
\item For $p = 6$ and general tori, the conjecture was proved by Bourgain-Shao-Sogge-Yao \cite{BSSY} in the range $\delta > \lambda^{-\frac{85}{252}}$, which was then improved by Hickman \cite{Hickman} to $\delta > \lambda^{-\frac{55}{162}}$.
\item If $2 \leq p \leq 4$ or $p \geq 4$ and $\delta > \max ( \min ( \lambda^{-\frac{3p-8}{5p-8}},\lambda^{-\frac{8-p}{5p-16}} ), \lambda^{-\frac{1}{2}})$ and on a general torus, it is shown in a work by the first two authors \cite{GermainMyerson1} that the conjecture holds with $\epsilon$-loss.
\item For $2 \leq p \leq 4$ and $\delta >  \lambda^{-1+\kappa}$, $\kappa>0$, and on a general torus, the conjecture is proved by the third author \cite{Pezzi} following partial results by Blair-Huang-Sogge \cite{BlairHuangSogge} for larger $\delta$.
\end{itemize}

Regarding Conjecture \ref{conjB}, the cases $p=\infty$ and $p=4$ can be easily settled: the former case is equivalent to counting the number of lattice points on the sphere, while the latter can be treated via a basic arithmetic argument, as was noted in  Bourgain \cite{Bourgain97} (an $\epsilon$-loss is incurred in both cases). It is difficult to do better than interpolating between these two bounds, but this was achieved recently by Mudgal \cite{mudgal} using methods from additive number theory. Indeed, it follows from Corollary 2.5 in \cite{mudgal} and an application of H\"older's inequality that an eigenfunction of the Laplacian with eigenvalue $4 \pi^2 \lambda^2$ on the square torus satisfies
$$
\| u \|_{L^p} \lesssim_\epsilon \lambda^{\frac{1}{2}-\frac 5 {2p} + \frac{3^{(2-\frac{p}{2})}}{2p} + \epsilon} \| u \|_{L^2} \qquad \mbox{if $p\geq 4$, $p \in 2\mathbb{N}$}.
$$

Finally, we refer to \cite{Germain} for background on estimating the operator norm of spectral projectors from $L^2$ to $L^p$ on the torus $\mathbb{T}^d$ for $d \neq 3$ or on general Riemannian manifolds.

\subsection{New results}

Our main results are as follows. We must distinguish between the square and non-square cases because the sharp global lattice point count is not known for the entire range of $\delta$ in the general case. Recall that in Conjecture \ref{conjA} we have $\lambda^{-1}\leq\delta\leq 1$.

\begin{theorem}[Square torus] \label{mainresult} 
We consider the square torus, and for some $\kappa > 0$, the range
\begin{equation}\label{eq:region}
        \delta > 
        \begin{cases}
           \la^{-1+\kappa} & 2\leq p \leq 4,\\
           \la^{-\frac{316-27p}{104(p-2)}+\kappa}&4<p\leq \frac{235}{52},\\
           \la^{\frac{p}{2}-3+\kappa} & \frac{235}{52}\leq p < \frac{389}{79},\\
           \la^{-\frac{9p-224}{444-158p}+\kappa} &\frac{389}{79}<p<5,\\
           \la^{-\frac{1}{2}+\kappa}&5\leq p\leq 49,\\
           \la^{-\frac{85p-358}{158p-128}}&49\leq p\leq \infty.
    \end{cases}
\end{equation}
Under these conditions, Conjecture \ref{conjA} is satisfied, with an implicit constant that depends on $\kappa$. In the case \(\kappa=0\), it is satisfied with \(\epsilon\)-loss.
\end{theorem}

The proof of this theorem is reached in Section \ref{sec:proofmaintheorem} by combining results obtained in previous sections. The only result that uses the fact the torus is square is the global lattice point bound for small $\delta$. All other results hold for general tori which allows us to prove the analogue of \cref{mainresult} whenever the global count holds.

\begin{theorem}[General torus]
\label{generaltorithm}
The conclusion of \cref{mainresult} also holds for a general torus provided that
\begin{equation}
\label{eq:expected_number_of_points} \#\{x\in\mathbb Z^3,\, |\sqrt{Q(x)}-\lambda|<\delta\} \lesssim_\epsilon \lambda^{2+\epsilon}\delta
\end{equation}
for every $\epsilon>0$ (in other words, the number of lattice points in a thin shell is bounded by the volume of the shell up to a subpolynomial factor). For general tori, this bound is known to hold if $\delta >\la^{-\frac{85}{158}}$, as follows from the main result of Guo~\cite{Guo} referred to above; this implies \cref{conjA} in that range.
\end{theorem}
\begin{remark}
For the square torus the hypothesis \eqref{eq:expected_number_of_points}
holds for $\delta>\lambda^{-1}$, by counting solutions to $x_1^2+x_2^2+x_3^2=n$ for each integer $n=\la^2+O(\lambda\delta)$, and this range for $\delta$ is optimal. The bound \eqref{eq:expected_number_of_points} is true on average over $\lambda$, and it seems likely to be true whenever $\delta>\lambda^{-1}$.  It seems possible that a sufficiently generalised version of the cap decomposition in \cref{sec:caps} could be used in an induction on scales argument to prove this bound.
\end{remark}
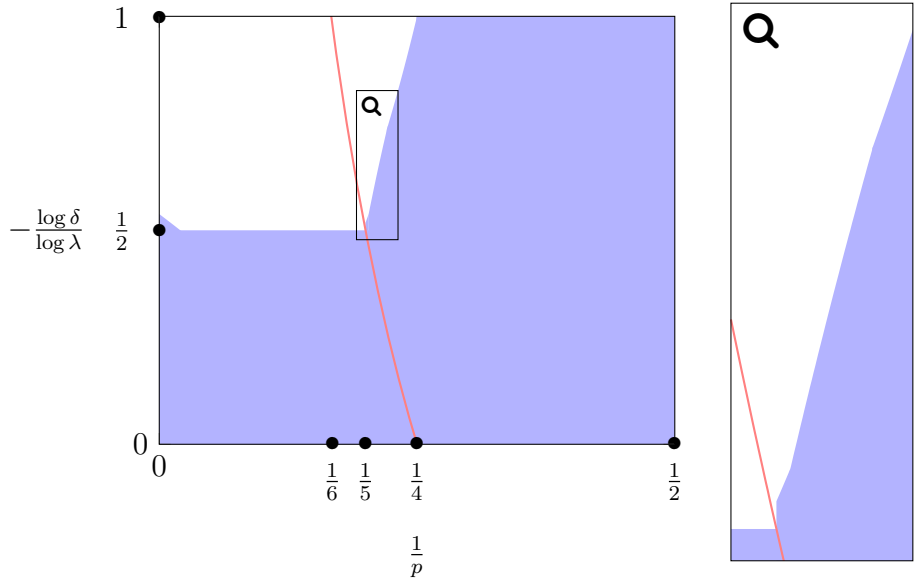
\begin{figure}
\centering
\begin{tikzpicture}
\begin{axis}[ xmin=0, xmax=.5, ymin=0, ymax=1, xtick=0, ytick=0]
\addplot[name path=A,domain=0:1,samples=100, opacity=0] {1};
\addplot[name path=B,domain=0:1,samples=100, opacity=0] {0};
\addplot[name path=C,domain=0:1,samples=100, opacity=0] {0.5};
\addplot[name path=E,domain=.15:.25,samples=100, opacity=0] {3 - 1/(2*x)};
\addplot [name path=F,domain=0:.5, samples=100, opacity=0] {(85-358*x)/(158-128*x)};
\addplot [name path=G,domain=0:.3, samples=100, opacity=0] {(316*x-27)/(104-208*x)};
\addplot [name path=H,domain=0:.3, samples=100, opacity=0] {(9-224*x)/(444*x-158)};
\addplot[blue!30] fill between[of=B and A, soft clip={domain=.2495:.5}];
\addplot[blue!30] fill between[of=B and C, soft clip={domain=0:.2505}];
\addplot[blue!30] fill between[of=G and B, soft clip={domain=.2207:.2505}];
\addplot[blue!30] fill between[of=B and E, soft clip={domain=.2029:.2217}];
\addplot[blue!30] fill between[of=H and B, soft clip={domain=.2:.2031}];
\addplot[blue!30] fill between[of=F and B, soft clip={domain=0:.2}];
\addplot [domain=0:1, samples=100, name path=f, thick, color=red!50] {1/(2*x)-2};
\end{axis}
\draw(2.622,2.722) rectangle (3.174,4.706);
\draw(2.825,4.5) node{\scriptsize\faSearch};
\draw(3.425,0) node{$\bullet$};
\draw(3.425,-.5) node{$\frac{1}{4}$}; 
\draw(2.74,0) node{$\bullet$};
\draw(2.74,-.5) node{$\frac{1}{5}$}; 
\draw(6.85,0) node{$\bullet$};
\draw(6.85,-.5) node{$\frac{1}{2}$}; 
\draw(3.425,-1.5) node{$\frac{1}{p}$};
\draw(2.3,0) node{$\bullet$};
\draw(2.3,-.5) node{$\frac{1}{6}$}; 
\draw(0,5.67) node{$\bullet$};
\draw(-.5,5.67) node{$1$};
\draw(-1.5,2.835) node{$-\frac{\log \delta}{\log \lambda}$};
\draw(0,2.835) node{$\bullet$};
\draw(-.5,2.835) node{$\frac{1}{2}$};
\begin{axis}[ xmin=.19, xmax=.23, ymin=.48, ymax=.83 , xtick=0, ytick=0, width=4cm, height=9cm, at={(7.6cm,-1.55cm)}, axis on top]
\addplot[name path=A,domain=.19:.23,samples=100, opacity=0] {.83};
\addplot[name path=B,domain=.19:.23,samples=100, opacity=0] {.48};
\addplot[name path=C,domain=.19:.23,samples=100, opacity=0] {0.5};
\addplot[name path=E,domain=.19:.23,samples=100, opacity=0] {3 - 1/(2*x)};
\addplot [name path=G,domain=.19:.23, samples=100, opacity=0] {(316*x-27)/(104-208*x)};
\addplot [name path=H,domain=.19:.23, samples=100, opacity=0] {(9-224*x)/(444*x-158)};
\addplot[blue!30] fill between[of=B and C, soft clip={domain=.19:.23}];
\addplot[white] fill between[of=A and C, soft clip={domain=.19:.23}];
\addplot[blue!30] fill between[of=G and B, soft clip={domain=.2211:.23}];
\addplot[blue!30] fill between[of=B and E, soft clip={domain=.2029:.2213}];
\addplot[blue!30] fill between[of=H and B, soft clip={domain=.2:.2031}];
\addplot [domain=.19:.23, samples=100, name path=f, thick, color=red!50] {1/(2*x)-2};
\end{axis}
\draw(8,5.5) node{\large\faSearch};
\end{tikzpicture}
\alt{A diagram showing where the main theorem is proven. The x-axis is given by the value of $1/p$, where $p$ ranges from 2 to $\infty$. The y-axis is given the negative of the  exponent on the shrinking window and ranges from 0 to 1. The diagram shows the theorem is proved by coloring in a point in the plane. A red curve is also drawn to distinguish between the two behaviors in the conjecture. The region where the theorem is proved is established in the statement of the theorem. There is a zoom in of a region near $p=4$ and a window $\la^{-1/2}$ to show that the figure is discontinuous at $p=5$, having a vertical jump.}
\caption{\label{loriot}
Theorem \ref{mainresult} verifies  Conjecture~\ref{conjA} in the blue region.
The vertical axis corresponds to $-\frac{\log \delta}{\log \lambda}$, and the horizontal axis to $\frac 1p$. The red line is the curve $\delta = \lambda^{2 - \frac{p}{2}}$, which separates the \textit{point-focusing regime} (to the left) from the \textit{geodesic-focusing regime} (to the right).}
\end{figure}

\medskip
The methods developed in the present paper allow improvements in dimensions $d \neq 3$. Indeed, they allow the removal of subpolynomial losses in \cite{DemeterGermain} and \cite{GermainMyerson1}.

\begin{theorem}\label{d2EpsilonRemovalIntro}
    Let $d=2$ and $\T{2}$ be the square torus. Let $\kappa,\kappa' >0$ and assume that $\delta>\la^{-\frac{1}{3}+\kappa}$ with $2\leq p \leq \infty$, or that $p<5-\kappa$ with $\delta>\la^{-1+\kappa'}$. Then

    \begin{equation*}
        \lpnorm{P_{\la,\delta} f}{p}{\T{2}}\lesssim (\la^{\frac{1}{2}-\frac{2}{p}}\delta^{1/2}+ (\la\delta)^{\frac{1}{2}(\frac{1}{2}-\frac{1}{p})})\lpnorm{f}{2}{\T{2}},
    \end{equation*}
    which is the sharp $2$-dimensional analog of \cref{conjA} here.
\end{theorem}
Our improvement of the main theorem of \cite{GermainMyerson1} does not admit a clean restatement, and so we state a version here dependent on the results of that paper.

\begin{theorem}
    \label{genDEpsRemovalIntro}
    Let $P_{\la,\delta}$ be the $d$-dimensional spectral projection operator. Suppose Theorem 6.1 in \cite{GermainMyerson1} gives that $P_{\la,\delta}$ obeys the $d$-dimensional analog of \cref{conjA} with $\eps$-loss for a pair $(p_0,\delta_0)$. Suppose that the point $(p_0,\delta_0)$ is such that the conjecture also holds with $\epsilon$-loss in the set 

    \begin{equation*}
        \{(p,\delta): |(p,\delta)-(p_0,\delta_0)|<\kappa\}
    \end{equation*}
    for some $\kappa>0$, where $|\cdot|$ is the magnitude of vectors in $\R{2}$. Then $P_{\la,\delta}$ obeys the sharp version of the conjecture at $(p_0,\delta_0)$.
\end{theorem}

\subsection{Organization of the paper}

\textit{Section \ref{sec:caps}} aims to improve bounds on the number of caps containing many lattice points. Here, caps refer to a collection of rectangular boxes with dimensions $\sim \delta \times \sqrt{\lambda \delta} \times \sqrt{\lambda \delta}$ which provide an almost disjoint covering of the annulus $\{ \xi \in \mathbb{R}^3, \lambda-\delta < |\xi| < \lambda + \delta\}$. These improved bounds are obtained by methods of the geometry of numbers.

\textit{Section \ref{sec:dec}} relies on the application of the $\ell^2$-decoupling estimate of Bourgain and Demeter \cite{BourgainDemeter3}, which gives an equivalent of orthogonality in $L^p$ for sums of functions supported on distinct caps in Fourier space. When applying this estimate, it is advantageous to bound the number of caps containing more points than average, which is possible with the help of the previous section.

\textit{Section \ref{sec:exponential}} focuses on the application of three-dimensional exponential sums \cite{Guo,Muller} to our problem, following Demeter-Germain \cite{DemeterGermain} who considered the case of dimension 2. We follow the idea of Stein-Tomas \cite{Tomas}: decomposing $P_{\lambda,\delta}$ into elementary operators, each of which is estimated by interpolating between elementary $L^2 \to L^2$ and $L^1 \to L^\infty$ bounds. This decomposition is achieved after writing out the kernel of $P_{\lambda,\delta}$, applying Poisson summation, and then splitting the corresponding sum dyadically.

\textit{Section \ref{sec:circle}} also adapts the of Stein-Tomas approach, but this time the decomposition is achieved through the circle method. This was first proposed by Bourgain \cite{Bourgain1} who focused on dimensions $\geq 4$. 

\textit{Section \ref{sec:epsilonremoval1}} applies the idea of height splitting in the point-focusing regime (to be more specific, it deals with the operator $P_{\operatorname{few}}$, which will be defined later and is dominant in the point-focusing regime). The idea of height splitting is to distinguish subsets of the torus where $P_{\lambda,\delta} f$ is bigger or smaller than some threshold, and use different approaches on each one. It follows \cite{BlairHuangSogge} and \cite{Pezzi} and allows the removal of the $\lambda^\epsilon$ loss which is automatically incurred when applying the $\ell^2$ decoupling estimate from Section \ref{sec:caps}.

\textit{Section \ref{sec:epsilonremoval2}} applies ideas from bilinear restriction theory in the geodesic-focusing regime (to be more specific, it deals with the operator $P_{\operatorname{many}}$, which will be defined later and is dominant in the geodesic-focusing regime). Here, too, this eliminates subpolynomial losses.

\textit{Section \ref{sec:proofmaintheorem}} combines the results obtained in previous sections to prove the main theorem.

\textit{Section \ref{sec:otherd}} examines the implications of the methods developed here in dimensions $d \neq 3$.

\textit{Section \ref{sec:additive}} of the appendix, finally, sketches yet another approach, through additive combinatorics; we do not pursue this line of investigation here.

\subsection{Limitations of the different methods}

\textit{Decoupling and cap counting} are only useful in our hands if $\delta > \min ( \lambda^{-\frac 12}, \lambda^{\frac p2 -3})$. Indeed, in the complement of this region, the leading contribution to the operator norm is expected to come from caps containing a single lattice point, but the decoupling estimates applied to these caps do not suffice to prove the conjecture.

\textit{Three-dimensional exponential sums estimates} are useful on general tori. Bounds on such sums allow us to prove the conjecture in a range which lies beyond the critical line $\delta = \lambda^{-1/2}$. Any progress on these bounds would lead to progress on the conjecture.

\textit{The circle method} does not seem effective on arbitrary tori; it also seems to be limited to the range $\delta > \lambda^{-\frac{1}{2}}$. Beyond this range, the minor arcs become dominant as far as pointwise estimates go, and the corresponding estimate is too weak to prove the conjecture. A better understanding of cancellations in Kloosterman sums might be key to push this method further.

\textit{The height splitting argument} allows us to remove $O_\epsilon(\la^\epsilon)$ losses in the estimates in the point-focusing regime. It also allows us to extend the validity of the conjecture slightly in the geodesic focusing regime, but it is dependent on the exponential sum estimates used in \cref{sec:exponential}.

\textit{The bilinear argument} allows us to remove $O_\epsilon(\la^\epsilon)$ losses in the estimates in the geodesic-focusing regime. This argument is sharp in the sense it proves the full conjecture for a subset of caps, but it is not enough on its own to extend the range where the conjecture is proven.

\textit{Additive combinatorics} offers another line of attack, which is restricted to even integer value for $p$. In particular, the delicate transition between the two regimes in the conjecture, which occurs between $p=4$ and $p=6$, cannot be observed through this method. Combining additive combinatorics with incidence geometry provides powerful bounds, but it is more effective (or at least, simpler to apply) if $\delta\sim \la^{-1}$.

\subsection*{Acknowledgments} Pierre Germain was supported by the  
Simons Foundation Collaboration on Wave Turbulence, a start up grant from Imperial College and a Wolfson fellowship. Daniel Pezzi is grateful to Imperial College for hosting him during this work.

\subsection*{AI use statement}
No AI was used during this project or in the preparation of the manuscript.

\section{Notations}
\noindent \textit{Inequalities.} For two quantities $A$ and $B$, we write $A \lesssim B$ if there exists a constant $C$ such that $|A| \leq CB$. If the constant $C$ is allowed to depend on parameters $a_1,\dots,a_n$, the notation $\lesssim$ is replaced by $\lesssim_{a_1,\dots,a_n}$.

We write $A \sim B$ if $A \lesssim B$ and $B \lesssim A$.

Finally, $A \ll B$ means that $A \leq c B$ for a sufficiently small (depending on the context) constant $c$.

\medskip

\noindent \textit{Fourier transform and Fourier series} We are working on the torus $\mathbb{T}^3 = \mathbb{R}^3 / \mathbb{Z}^3$, for which Fourier series are given by
$$
\widehat{f}_k = \int_{\mathbb{T}^3} f(x) e^{-2\pi i k \cdot x} \dd x, \qquad f(x) = \sum_{k \in \mathbb{Z}^3} \widehat{f}_k e^{2\pi i k \cdot x}.
$$
For functions on $\mathbb{R}^d$, we adopt the following normalization for the Fourier transform and its inverse:
$$
\widehat{f}(\xi) =  \int_{\mathbb{R}^d} f(x) e^{-2\pi i x \cdot \xi} \dd x, \qquad f(x) = \int_{\mathbb{R}^d} \widehat{f}(\xi) e^{2\pi i x \cdot \xi} \dd \xi.
$$
With this normalization, the Poisson summation formula becomes
$$
\sum_{\ell \in \Z{d}} f(\ell) = \sum_{k \in \Z{d}} \widehat{f}(k),
$$
and furthermore we have
$$
\widehat{f \cdot g} = \widehat{f} * \widehat{g} \quad \mbox{and} \quad \widehat{f * g} = \widehat{f} \cdot \widehat{g}.
$$
Finally, the Fourier multiplier $m(D)$ stands for the operator defined by
$$
m(D) e^{2\pi i k \cdot x} = m(k) e^{2\pi i k \cdot x}
$$
and extended to general $f$ by linearity.
\section{Counting caps with many points}\label{sec:caps}

\subsection{Decomposition of the shell into caps}
We split the spherical shell 
$$
S_{\lambda,\delta} = \{ x \in \mathbb{R}^3, \, \left| \sqrt{Q(x)}-\lambda \right| < \delta \}
$$
into a collection $\mathcal{C}$ of finitely overlapping caps $\theta$:
$$
S_{\lambda,\delta} = \bigcup_{\theta \in \mathcal{C}} \theta,
$$
where each cap is of the form
\[
\theta = \{ x \in \mathbb{R}^3, \, | x - x_\theta| < \sqrt{\lambda \delta} \} \cap S_{\lambda,\delta} \qquad \text{for some }x_\theta, \; Q(x_\theta) = \lambda^2.
\]

Each cap fits into a rectangular box with dimensions $\sim \delta \times \sqrt{\lambda \delta} \times \sqrt{\lambda \delta}$.

Recall that \(Q(x) = x^TAx\) for a fixed, nonsingular, real symmetric matrix $A$. Then \(Ax_\theta/|Ax_\theta|\) is a unit normal to the ellipsoid \(Q(x)=\lambda\) at the point \(x_\theta\); we define \(Ax_\theta/|Ax_\theta|\) to be the unit normal of the cap \(\theta\).

\medskip

Denote $N_\theta$ for the number of points in $\mathbb{Z}^3 \cap \theta$. On one hand, it is clear that $N_\theta \lesssim \lambda \delta$. On the other hand, one expects that the average cap will contain a number of points comparable to its volume, in other words $N_\theta \sim \lambda \delta^2$ (provided this quantity is $>1$, which occurs if $\delta > \lambda^{-\frac{1}{2}}$).

This leads naturally to defining the following sets, which gather caps containing comparable numbers of points.
\begin{align*}
& \mathcal{C}_s = \{ \theta \in \mathcal{C}, \, N_\theta \sim 2^s \}, \qquad \max\{1,\la\delta^2\} < 2^s < \lambda \delta.
\end{align*}

We consider the lattice \(\Lambda_\theta\) generated by \(\mathbb{Z}^3 \cap \theta - \mathbb{Z}^3 \cap \theta \). We call the rank of this lattice the rank of the cap \(\theta\). Furthermore, we set
\begin{align*}
 \mathcal{N}^r_s &= \#\{ \theta \;, \; N_\theta \sim 2^s \;\text{and} \; \theta \; \text{has rank $r$} \}
 &&(\lambda \delta^2 \lesssim 2^s\lesssim \lambda\delta,r\in\{1,2\}).
\end{align*}
We are interested in $\lambda \delta^2 <2^s\lesssim \lambda\delta$ and $r\in\{1,2\}$ because caps with less points than average be treated uniformly, and caps with more than the average must have rank $1$ or $2$.

Our bounds for the sizes of these collections of caps are found in Proposition \ref{prop:rank_1_bound} and Proposition \ref{prop:rank2bound}, which are  recorded together at the start of \cref{sec:capMany}. We will give improved estimates for both the number of rank 1 and the number of rank 2 caps. It it not entirely clear to us what the optimal bounds might be, but we make some suggestions in the comments after Proposition \ref{prop:rank_1_bound} and Proposition \ref{prop:rank2bound}.

Along the way we also use an `incidence' version of $\mathcal N_s^r$. Recall that a primitive lattice in $\mathbb Z^3$ is the intersection of $\mathbb Z^3$ with any fixed linear subspace of $\mathbb Q^3$, and that an affine linear space is a translation of a rational linear space. Our incidence counting function captures all translations of primitive integral lattices (that is, all intersections of $\mathbb Z^3$ with a rational affine linear space) that have a large intersection with any cap $\theta$.
This is given by\SMnote{It is not necessarily true that $2^t>\lambda\delta^2$ for us, which is something to watch out for}
$$\widetilde{\mathcal N}^r_t = \#\{ (\theta,V), \;\# (\mathbb Z^3\cap V\cap \theta )\sim 2^t, \; V \subseteq \mathbb Q^3 \; \text{ affine linear}, \; \dim V= r\}.$$
We then have at once
\begin{equation*}
    \label{eq:rank1-incidence}
    \mathcal{N}^r_s \leq \widetilde{\mathcal N}^r_s.
\end{equation*}

\subsection{Bound for rank 2 caps}

In this subsection we prove the following result.

\begin{prop}\label{prop:rank2bound}\noindent
\begin{enumerate}
    \item (Special case of Theorem~4.1 in~\cite{GermainMyerson1}.)
    We have
\begin{equation*}
    \mathcal{N}^2_s \lesssim (2^{-s}\lambda\delta)^{3}.
\end{equation*}
    \item We have for any $\epsilon>0,$
    \begin{equation*}
            \mathcal{N}^2_s \lesssim_\epsilon (2^{-s}\lambda\delta)^{2}+\lambda^\epsilon\delta(2^{-s}\lambda\delta)^{4}.
    \end{equation*}
    \item We have 
    \begin{equation*}
        \mathcal{N}^2_s \lesssim \sum_{t=1}^{s/2} 2^{t-s}
        \widetilde{\mathcal N}^1_t
        .
    \end{equation*}
\end{enumerate}
\end{prop}

When $2^s\sim {\lambda\delta}$, which is essentially the largest value possible as the integer points in $\theta$ must fit in a disc of radius $\sqrt{\lambda\delta},$
then these bounds show that $\mathcal{N}^2_s \lesssim 1$ which is optimal. They are thus quite sharp for large $s,$ while our bounds for rank 1 caps will be sharpest when $s$ is small (see \cref{rem:rank1_conjecture}). It seems possible that the term $(2^{-s}\lambda\delta)^{2}$ in part 2 of the proposition could be reduced by counting separately those caps for which any of the components $w_1,w_2,w_3$, as defined below, happen to vanish. Perhaps this would lead to a sharp bound.

Since the proof in~\cite{GermainMyerson1} is quite involved, in the remainder of this subsection we give a self-contained proof of Proposition~\ref{prop:rank2bound}.

\begin{df}
    The determinant $\det \Lambda$ of a rank 2 lattice $\Lambda$ in $\mathbb R^3$, is defined to be $|u\wedge v|$ for any basis $(u,v)$ of $\Lambda$ over $\Z{}$. The vector $w=\pm u\wedge v$ is (up to sign) independent of the choice of basis.
\end{df}
 
\begin{lem}\label{lem:rank2}
    Let $\lambda \delta^2 <2^s\lesssim \lambda\delta$.
\begin{enumerate}
    \item If \(\Lambda_\theta\) has rank 2 and \(N_\theta\sim 2^s\), then for some $s/2\leq t \leq s$, there are at least $2^{s-t}$ different affine lines $L$ in $\mathbb R^3$, such that $\theta\cap \mathbb Z^3 \cap L\gtrsim 2^t$.
    
    \item If \(\Lambda_\theta\) has rank 2, fix  \(w=w(\Lambda_\theta)\) to be one of the two vectors $w$ defined above, so that in particular $w\in\mathbb Z^3$ since $\Lambda_\theta\subset \mathbb Z^3$. Then 
\begin{equation}
    |w|N_\theta \lesssim \lambda\delta,\qquad
|(Ax_\theta) \wedge w |\lesssim (\lambda\delta)^{3/2} N_\theta^{-1} .    \label{eq:rank2_geom_of_nos}
\end{equation}
    \item Let $w$ be as in part 2. Given $w_0\in\mathbb Z^3$, there are $\lesssim(\lambda\delta 2^{-s}/|w_0|)^2$ possible $\theta$ such that  \(\Lambda_\theta\) has rank 2, \(N_\theta\sim 2^s\) and  \(w(\Lambda_\theta)=w_0\).
\end{enumerate}
\end{lem}

\begin{proof}
We first choose a suitable basis of $\Lambda_\theta$.

Given the cap $\theta$, define a symmetric matrix $M_\theta$ according to $M_\theta = I_3 + \delta^{-1}Ax_\theta Ax_\theta^T/|Ax_\theta|^2$, so that $M_\theta$ is a stretch by a factor $1+\delta^{-1}$ in the direction $Ax_\theta$ normal to $\theta$. Observe that $M_\theta\theta$ is contained in a ball of radius $\sim\sqrt{\lambda\delta}$. Thus $M_\theta$ maps a cap with dimensions $\delta\times \sqrt{\lambda\delta}\times \sqrt{\lambda\delta}$ to a body with dimension $\sqrt{\lambda\delta}$ in every dimension.

let \(u=u(\theta),v=v(\theta)\) be a basis of \(\Lambda_\theta\) such that $(U,V) = (M_\theta u, M_\theta v)$ is a \emph{reduced} basis of $M_\theta\Lambda_\theta$, that is $|U|\leq |V|$  and \(|U\cdot V| \leq \frac{1}{2}|U|^2\) say. Without loss of generality \(w=w(\Lambda_\theta)= u\wedge v\).

Without loss of generality we can assume that \(x_\theta \in\mathbb Z^3\), since this change introduces only a negligible error.

\noindent\textit{Proof of part 1.} 
Let $\theta$ be a cap with rank 2 and $N_\theta \sim 2^s$.
We have $$au+bv+x_\theta\in\theta\iff aU+bV\in M_\theta(\theta-x_\theta) .$$
Now here $U,V$ are roughly orthogonal and the set $M_\theta(\theta-x_\theta)$ is a convex body with size $\sim\sqrt{\lambda\delta}$ in every direction. In particular, for suitable constants $0<c<C$ we have
$$
\frac{|a|}{|U|},\frac{|b|}{|V|}\leq c\sqrt{\lambda\delta}
\implies
aU+bV\in M_\theta(\theta-x_\theta)
\implies\frac{|a|}{|U|},\frac{|b|}{|V|}\leq C\sqrt{\lambda\delta}.
$$
Hence
$$
2^s\sim \#\theta \cap \mathbb Z^3
\sim \#\{(a,b)\in\mathbb Z^2,\;
\frac{|a|}{|U|},\frac{|b|}{|V|}\leq \sqrt{\lambda\delta}\}
\sim \frac{\lambda\delta}{|U|\cdot|V|}.
$$
Recalling that $|U|\leq |V|$, we see that $2^s \sim \frac{\sqrt{\lambda\delta}}{|U|}\cdot\frac{\sqrt{\lambda\delta}}{|V|}$ where $\frac{\sqrt{\lambda\delta}}{|U|}\geq \frac{\sqrt{\lambda\delta}}{|V|}$ holds. We can now take $2^t\sim \frac{\sqrt{\lambda\delta}}{|V|}$ and pick the lines
$$
L_b=\{au+bv,\; a\in\mathbb R\}
\qquad
(b\in\mathbb Z, b\lesssim\sqrt{\lambda\delta}/|V|).
$$
Each of these satisfies $L_b\cap \theta \cap \mathbb Z^3 \sim \frac{\sqrt{\lambda\delta}}{|U|}\sim 2^{s-t}$, as required.

\noindent\textit{Proof of part 2, set-up.}
We first prove the bounds for \(|w|N_\theta\) and \(|(Ax_\theta) \wedge w|\); to do this we set up a convex body defined by certain inequalities.

Let \(v' = v-\frac{u\cdot v}{|u|^2}u\) and \(u' = u-\frac{u\cdot v}{|v|^2}v\), so that
\begin{align}
    u\cdot v'&=
    u'\cdot v=0.
    \label{eq:dual-basis}
\end{align} 
Let \(a_\theta,b_\theta,c_\theta\in\mathbb R\) such that \(Ax_\theta = a_\theta u'+b_\theta v'+c_\theta w \), so that in particular
\begin{align}
|c_\theta w| &= \frac{|(Ax_\theta)\cdot w|}{|w|},
&|a_\theta u'+b_\theta v'| &= \frac{|(Ax_\theta)\wedge w|}{|w|}.
    \label{eq:dual-basis-2}
\end{align}

\noindent \textit{Proof of part 2, step 1. The convex region.}
We claim that if \(a,b\in \mathbb Z\) with  \(au+bv\in \theta -\theta\), then $(a,b)$ lie in the region defined by
\begin{align}\label{rank-2}
|a| \frac{|u|}{\sqrt{\lambda\delta}}&\lesssim 1,&
|b| \frac{|v|}{\sqrt{\lambda\delta}}&\lesssim 1,&
\left| (a,b)\cdot (\frac{ a_\theta u\cdot u'}{\delta \lambda},\frac{b_\theta v\cdot v'}{\delta \lambda}) \right|&\lesssim 1,
\end{align}
where the coefficients all satisfy
\begin{equation}\label{rank-2-supplement}
\frac{|u|}{\sqrt{\lambda\delta}},
\frac{|v|}{\sqrt{\lambda\delta}},
\frac{| a_\theta u\cdot u' |}{\delta \lambda},
\frac{| b_\theta v\cdot v' |}{\delta \lambda}
\lesssim 1.
\end{equation}

To justify the claim, let \(au+bv\in \theta -\theta\), and observe first that because $(U,V)$ is a reduced basis we have \(|aU+bV|\sim|aU|+|bV|\) for all \(a,b\in \mathbb R\). (Indeed that property is essentially equivalent to the basis being reduced, in the sense that it is equivalent to $|U|\cdot|V|-|U\cdot V| \gtrsim |U|\cdot|V|$ which is a little weaker than our assumption in the lemma.)  Since $\theta-\theta$ contains two linearly independent points of $\Lambda_\theta$, we see that $M_\theta \Lambda_\theta$ contains two linearly independent points of norm $\lesssim\sqrt{\lambda\delta}$, and therefore, because \(|aU+bV|\sim|aU|+|bV|\), the basis vectors $U,V$ must have  $|U|,|V|\lesssim\sqrt{\lambda\delta}$. Hence $u,v\in O(1)\cdot (\theta-\theta)$.

Now, if \(au+bv\in O(1)\cdot(\theta-\theta) \) then \(|aU+bV|\lesssim \sqrt{\lambda\delta}\). We conclude that \(|a| \lesssim \sqrt{\lambda\delta}/|U|, |b| \lesssim \sqrt{\lambda\delta}/|V|\), which  proves the first two inequalities in \eqref{rank-2}, and on taking $(a,b)=(1,0)$ and $(0,1)$ we also obtain the first two inequalities in \eqref{rank-2-supplement}.

Next, since $Ax_\theta/|Ax_\theta|$ is a unit normal to \(\theta\), for any \(au+bv\in O(1)\cdot(\theta-\theta) \) we have \( | (au+bv)\cdot Ax_\theta | \lesssim \delta |Ax_\theta|\sim \delta \lambda\).
Here \((au+bv)\cdot Ax_\theta =(au+bv)\cdot (a_\theta u'+b_\theta v')= aa_\theta u\cdot u'+bb_\theta v\cdot v'\), so this quantity is \(O(\delta\lambda)\).
This proves the last inequality in \eqref{rank-2}.
Taking the particular cases when \((a,b)=(1,0)\) or \((0,1)\) proves the last two inequalities in \eqref{rank-2-supplement}.
\medskip

\noindent
\textit{Proof of part 2, step 2. Proof of \eqref{eq:rank2_geom_of_nos}.}
Let \(x\in \theta \cap \mathbb Z^3\), and note that the points in \(\theta \cap \mathbb Z^3\) are necessarily of the form  \(x +au+bv \) where \(au+bv\in \theta -\theta\) and \(a,b\in \mathbb Z\).
We therefore find that \(N_\theta\) is bounded by the number of integer solutions \((a,b)\in \mathbb Z^2\) to \eqref{rank-2} where the coefficients satisfy \eqref{rank-2-supplement}.

Using \eqref{rank-2}, the number of possible \(a\) is bounded by \(O(\sqrt{\lambda\delta}|u|^{-1}+1)\), which is \(O(\sqrt{\lambda\delta}|u|^{-1})\) by \eqref{rank-2-supplement}. 

For a given \(a\) the number of \(b\) is bounded by
\[
O(\min\{\sqrt{\lambda\delta}|v|^{-1}+1),\lambda\delta |b_\theta v\cdot v'|^{-1}+1)\}),
\]
using \eqref{rank-2}. Now using \eqref{rank-2-supplement},
\[
N_\theta \lesssim \min\left\{
\frac{\lambda\delta}{|w|},
\frac{(\lambda\delta)^{3/2}}{b_\theta |v\cdot v'|\cdot |u|}
\right\}.
\]
Repeating the argument with \(a\) and \(b\) exchanged gives
\begin{equation}
\label{eq:rank2_start}    
N_\theta \lesssim \min\left\{
\frac{\lambda\delta}{|w|},\frac{(\lambda\delta)^{3/2}}{a_\theta |u|\cdot|w|},
\frac{(\lambda\delta)^{3/2}}{b_\theta |v|\cdot |u\cdot u'|}
\right\}.
\end{equation}
The first part of \eqref{eq:rank2_geom_of_nos} follows using the first term in this minimum. For the remaining part, by \eqref{eq:dual-basis-2} we have
\[
|(Ax_\theta) \wedge w|
\lesssim
|w|\max\{
|a_\theta u|,|b_\theta v|\}.
\]
Together with \eqref{eq:rank2_start} this gives the remaining part of \eqref{eq:rank2_geom_of_nos}.

\medskip

\noindent \textit{Proof of part 3.}
It remains to prove the bound on the number of \(\theta\) with given \(w\) and \(s\). Here we observe that because \(|(Ax_\theta) \wedge w|\lesssim (\lambda\delta)^{3/2}2^{-s}\), the vector  \(x_\theta\) must lie in the intersection of \(S_{\lambda, \delta}\) with a ball of radius \((\lambda\delta)^{3/2}2^{-s}/|w| \). Recall that the caps $\theta$ are finitely overlapping and that they are centered on the points \(x_\theta\); it follows that there are \(O(1)\) of the points \(x_\theta\) in any ball of radius \(\sqrt{\lambda\delta}\). Hence the total count of \(x_\theta\) is \(\lesssim (\lambda\delta 2^{-s}/|w|)^2\); note that this quantity is $\gtrsim 1$ since $2^s |w| \lesssim \lambda \delta$.
 \end{proof}

 \begin{proof}[Proof of Proposition~\ref{prop:rank2bound}]
 \textit{Proof of part 1.}
Let \(1 \lesssim W \lesssim \lambda\delta 2^{-s}\).
The number of vectors $w\in\mathbb Z^3$ with  \(|w|\sim W\) is at most $\lesssim W^3$. For each of these $w$ there are $\lesssim(\lambda\delta 2^{-s}/W)^2$ possible $\theta$ such that  \(\Lambda_\theta\) has rank 2, by an application of Lemma \ref{lem:rank2}. The total count is \( (\lambda\delta 2^{-s})^2 W\) and summing over dyadic integers \(1<W<\lambda\delta 2^{-s}\) gives a total $(\lambda\delta 2^{-s})^3$.

\textit{Proof of part 2.}
Consider $\theta$ with rank 2 and $N_\theta \sim 2^s$. Every point $x\in\mathbb Z^3\cap \theta-\mathbb Z^3\cap \theta$ satisfies $x\cdot Ax_\theta \lesssim\delta\lambda$ because $Ax_\theta$ is a normal to the cap with size around $\lambda$. Therefore, by the second part of \cref{lem:rank2}, and deliberately not counting the origin $x=0$, we have at least $2^s$ integer solutions  $0<|x|\lesssim\sqrt{\lambda\delta}$ to $$x\cdot w\lesssim |w| \delta + \lambda^{-1}|x|(\lambda\delta)^{3/2}2^{-s}.$$
The right-hand side here is $\lesssim \lambda\delta^2 2^{-s}$ since $|w|\lesssim \lambda\delta2^{-s}$.

Applying this to every cap in $\mathcal N_s^r$, the number of caps is bounded by $2^{-s}$ times the number of solutions in $\mathbb Z^3$ to 
$$0\neq |w|\lesssim \lambda\delta2^{-s},\quad
0\neq|x|\lesssim\sqrt{\lambda\delta},\quad x\cdot w\lesssim \lambda\delta^2 2^{-s}.$$
We divide the solutions into two classes. The number of solutions with $x_1w_1=x_2w_2=x_3w_3=0$ is $\lesssim (\sqrt{\lambda\delta})^2(\lambda\delta 2^{-s})$, that is $\lesssim (\lambda\delta)^2 2^{-s}$. On the other hand for the solutions with some $x_iw_i\neq 0$, without loss of generality we can count the solutions with $x_1w_1\neq 0$, and by summing trivially over $x_2,x_3,w_2,w_3$ and using the divisor bound this can be shown to be $\lesssim\lambda^\epsilon (\lambda\delta 2^{-s})^2 (\lambda\delta) (\lambda\delta^2 2^{-s})=\lambda^{4+\epsilon} \delta^5 2^{-3s}$. The final bound for the caps is $(2^{-s}\lambda\delta)^{2}+\lambda^\epsilon\delta (2^{-s}\lambda\delta)^4$.

\textit{Proof of part 3.} This follows at once from part 1 of \cref{lem:rank2}.
 \end{proof}

\subsection{Preparation for rank 1 caps}

Recall the notation $\mathcal N^r_s$ from the start of the section.
In \cite{GermainMyerson1} the bound
\[
\mathcal{N}^1_s \lesssim (\lambda 2^{-s})^{3/2}
\]
is proved; as in the rank 2 case, specializing to $d=3$ is somewhat elaborate. We will not need this as we give a separate proof of an improved bound.

Our improvement comes from
Theorem~2 of Huxley~\cite[p595]{Huxley2}, which, together with his Proposition 2A and taking the regions where his (1.8) holds but his (1.11) and (1.12) do not, yields
\begin{lem}\label{thm:huxley}
Let \(0<\alpha< \beta\) and
    let $G \in C^3([0,1])$ such that $G''$ takes values in \([\alpha,\beta]\), so in particular it is nonvanishing, and $|G^{(3)}|\leq \beta.$ Suppose that, for some $\epsilon>0$,
    \[
    1\leq
    (XY)^{69\cdot(3/10)-6+\epsilon}\leq 
    X^{156\cdot (3/10)-14}\leq
    (XY)^{87\cdot (3/10)-8-\epsilon},
    \]
    in other words, with a harmless adjustment to the value of \(\epsilon,\) suppose that
    \[
    1\leq
    Y^{147/181+\epsilon}\leq X \leq Y^{181/147-\epsilon},
    \]
    and let \(\delta>0\). Then
    \begin{equation*}
        \#\{ (x,y)\in \mathbb Z^2 : x\in [0,X],\,|y- Y G(x/X)|\leq \delta\}
\\    \lesssim_{\alpha, \beta,\epsilon}
    \delta X+
    (XY)^{131/416+\epsilon},
    \end{equation*}
    the last exponent being given according to Huxley's formula \(\frac{67\cdot (3/10)-7}{212\cdot(3/10)-22}.\)
\end{lem}

The lemma above is used to count solutions $a,b\in\mathbb Z$ to the inequality \eqref{eq:ellisoid_rank1} below, which has the shape $|Q^\ad(av+bw)-C|<D$ for some parameters \(v,w \in\mathbb Z^3\) satisfying $|w||v|\lesssim \sqrt{\lambda\delta}$. One might think that in the case of the standard torus, when $Q^\ad(av+bw)=|av+bw|^2$, there should be a more direct approach; due to the large size of $v,w$ this does not seem to be the case and we always rely on Lemma~\ref{thm:huxley}.

In some ranges we fall back on the following standard lemma, due to van der Corput.
\begin{lem}\label{lem:vdC}
Let \(0<\alpha< \beta\) and
    let $G \in C^2([0,1])$ such that $G''$ takes values in \([\alpha,\beta]\), so in particular \(G''\) is nonvanishing,  let \(1\leq X\leq Y^2\)
    and let \(\delta>0\). Then
    \begin{equation*}
        \#\{ (x,y)\in \mathbb Z^2 : x\in [0,X],\,|y- Y G(x/X)|\leq \delta\}
\\    \lesssim_{\alpha, \beta}
    \delta X+
    (XY)^{1/3}.
    \end{equation*}
\end{lem}
We will use this to count integer points in an annular region which may be long and thin. 
We remark that the condition \(1\leq X\leq Y^2\) is more commonly stated as  $1\leq Y^{1/2} \leq X\leq Y^2$ (or equivalently, $1\leq X^{1/2} \leq Y\leq X^2$), since for $X\leq Y^{1/2}$ the bound in the lemma can be proved by an almost trivial argument (counting the $y$ for each fixed $x$ separately).

Since the deduction of the previous result is short, we include it here.
We first  quote from Lemma~2.2 of~\cite{GrahamKolesnik}.
\begin{lem}
    Let \(0<\alpha< \beta\) and
    let $f \in C^2([0,X])$ such that $f''$ takes values in \([\alpha\lambda,\beta\lambda]\), so in particular \(f''\) is nonvanishing. Then
    \[
    \sum_{x=0}^X e^{2\pi i f(x)}
    \lesssim_{\alpha,\beta}
    X \lambda^{1/2}
    +\lambda^{-1/2}.
    \]
\end{lem}
\begin{proof}[{Proof of Lemma~\ref{lem:vdC}}]
For each $k\in\mathbb N$, the F\'ejer kernel \[F_k(t)=\sum_{j=-k}^k (1-|j|/k) e^{2\pi i jt}=k(1-\cos 2\pi kt)/(1-\cos 2\pi t)\] is nonnegative and 1-periodic, with Fourier support $[-k,k]$ and Fourier coefficients $c_{k,j}$ of absolute value at most 1. Moreover $F_k$  is essentially supported on an $O(1/k)$-neighborhood of the integers, in the sense that $F_k(t)\geq k/3$ for $|t|<1/2\pi k$, because $t^2/3 \leq 1-\cos t \leq t^2 $ for  $|t|<1$.
    
    We remark that the F\'ejer kernel itself has no essential role here, as it is easy to construct other such functions with different constants in the inequalities $F_k(t)\gtrsim 1, c_{j,k}\lesssim 1$.
    
    It follows that for any  $1\leq k\leq (2\pi\delta)^{-1}$ we have
    \[
    \#\{ (x,y)\in \mathbb Z^2 : x\in [0,X],\,|y- Y G(x/X)|\leq \delta\}
    \leq
    \frac{3}{k}\sum_{x=0}^X F_k( YG(x/X)).
    \]
    This sum is
    \[
    \frac{3}{k}\sum_{j=-k}^k c_{k,j}  \sum_{x=0}^X e^{2\pi i j YG(x/X)}.
    \]
    Separating the term $j=0$, and using the previous lemma with \(\lambda =  Yj/X^2\), this is
    \begin{align*}
    &    \lesssim_{\alpha,\beta}
    \frac{X}{k}+
     \frac{1}{k}\sum_{j=-k}^k (\sqrt{ Yj}+\frac{X}{\sqrt{Yj}})
    \\
    &\lesssim
    \frac{X}{k}+
    \sqrt{ Yk}+\frac{X}{\sqrt{ Yk}}.
    \end{align*}
    Recall that we can choose any $1\leq k\leq (2\pi\delta)^{-1}$. Pick $k\sim \min\{\delta^{-1},X^{2/3}Y^{-1/3}\}$, in order to balance the first two terms in the bound above. The bound becomes
    \begin{equation}
        \lesssim
        \begin{cases}
            \delta X+Y^{1/2}\delta^{-1/2}+XY^{-1/2}\delta^{1/2}
            &(\delta^{-1} < X^{2/3} Y^{-1/3}),
            \\
            ( XY)^{1/3}+( XY)^{1/3}+X^{2/3}Y^{-1/3}
            &(\delta^{-1} \geq X^{2/3}Y^{-1/3}).
        \end{cases}
    \end{equation}
    Provided that $X\leq Y^2$, we have $X^{2/3}Y^{-1/3}\leq(XY)^{1/3}$, so
    the second case is satisfactory. 
    
In the first case, that is $\delta^{-1} < X^{2/3}Y^{-1/3}$, we have the bound
    \[\lesssim
    \delta X+ Y^{1/2}\delta^{-1/2}+X Y^{-1/2}\delta^{1/2}.
    \]
    The second term $ Y^{1/2}\delta^{-1/2}$ is bounded by $( XY)^{1/3}$ since   $\delta^{-1} < X^{2/3} Y^{-1/3}$, and the third term is majorized by $X\delta+X Y^{-1}$, and hence by $ X\delta+ (XY)^{1/3}$ provided that $X\leq Y^2$.
\end{proof}

The counting lemmas above are well designed for applications to integer points near dilates of a fixed curve. For more general linear transformations of a fixed curve, a little work is required. The particular case we need is as follows.

\begin{lem}\label{lem:annular-count}
    Fix \(\kappa>0\).
    Let $q(x,y) = q_1x^2+2q_2xy+q_3y^2$ be a real, positive definite quadratic form with coefficients bounded by $|q_1|,|q_2|,|q_3|\leq \kappa$ and determinant bounded by $q_1q_3-q_2^2\geq \kappa$. Let $1\leq B^{1/2}\leq A\leq B^2$, and let \(0<\eta \leq 1\) and $\eps>0$. The number of 
    integer points $(a,b)$ with
    \begin{equation}
        \label{eq:annular-count}
    |q(a/A,b/B) -1|<\eta.
    \end{equation}
    is $\lesssim_\kappa \eta AB +(AB)^{1/3}$, and if moreover
    \[
    B^{\frac{147}{181}+\epsilon}\leq A \leq B^{\frac{181}{147}-\epsilon},
    \]
    then this number of integer points is $\lesssim_{\kappa,\epsilon}\eta AB +(AB)^{131/416+\epsilon}$.
\end{lem}

In many applications, including ours, the case $A=B$ is the critical one, and so the exponent $131/416$ is more important than the factor $147/181$. 
Modeling $(a/A,b/B)$ by points chosen uniformly at random from $[-1,1]^2$, and considering separately the lines $a=0, b=0$, suggests that for all $A,B\geq 1$ the number of solutions to \eqref{eq:annular-count} should be $O(\eta AB+\eta^{1/2}A+\eta^{1/2}B)$. In particular, if $A=B$ then the count should actually be $O(\eta AB+1)$.

\begin{proof}
Throughout the proof we allow implicit constants to depend on $\kappa$.

Without loss of generality, we suppose that $B\leq A$, for otherwise we may swap the roles of $A$ and $B$. (In our application we will in any case have $B\leq A$.)

Set $E=\{(a,b)\in\mathbb R^2:q(a/A,b/B)=1
\}$, and $R=\{(a,b)\in\mathbb R^2:|q(a/A,b/B)-1|<\eta
\}$

\medskip

\noindent
\textit{Step 1. Division into horizontal and vertical arcs.}
We are going to divide $E$ into overlapping pieces according to the direction of the gradient 
\begin{equation}
    \label{eq:formula-for-derivs}
\begin{pmatrix}\partial/\partial a\\\partial/\partial b\end{pmatrix} q(a/A,b/B) = \begin{pmatrix}{2/A} & 0 \\ 0 &  {2/B} \end{pmatrix} \begin{pmatrix}q_1a/A+q_2 b/B\\q_2a/A+q_3 b/B\end{pmatrix}.\end{equation}
Because the gradient \(\nabla q(x,y)\) is nonvanishing on $q=1$, we have
\begin{equation}
    \label{eq:gradient_lower_bound}
\left|
\begin{pmatrix}q_1a/A+q_2 b/B\\q_2a/A+q_3 b/B\end{pmatrix}
\right|^2
\sim 1
\qquad((a,b)\in R).
\end{equation}

Let $C\gg 1$ be a sufficiently large parameter which will be chosen at the end of the argument. In other words, $C$ will eventually be taken to be fixed, but until then, our implicit constants are not permitted to depend on $C$.

We will assume that \(\eta \ll C^{-2}\), for if not, we may partition the annular region $|q-1|\leq \eta$ into smaller regions \(|q-\rho| \leq \eta'\) and count integer points in each each of these regions separately. This introduces a dependence on $C$ into the final bound; however, this occurs in the very last step, when adding the contributions of all the  different regions \(|q-\rho| \leq \eta'\), and by that time $C$ will have been taken to be a fixed constant.

Set
\begin{align*}
R_{\text{vert}} ={}&
\{(a,b)\in\mathbb R^2:
\\
&|q(a/A,b/B)-1|<\eta,
|\frac{\partial}{\partial a}q(a/A,b/B)|
\geq
\frac{CB}{2A}|\frac{\partial}{\partial b} q(a/A,b/B)|
\},
\\
R_{\text{horiz}}={}&
\{(a,b)\in\mathbb R^2:
\\
&|q(a/A,b/B)-1|<\eta,
|\frac{\partial}{\partial a}q(a/A,b/B)|
\leq
\frac{2CB}{A}|\frac{\partial}{\partial b} q(a/A,b/B)|
\}.
\end{align*}
This cuts the annular region $R$ into two pieces, divided by the diagonal lines \( \frac{\partial}{\partial a}q(a/A,b/B)=\pm\frac{\partial}{\partial a}q(a/A,b/B)\). Each of the pieces $R_{\text{vert}},R_{\text{horiz}}$ consists of two components. We claim that
\begin{align}
\frac{\partial}{\partial a}q(a/A,b/B)&\sim A^{-1}
&&\text{and}
&a&\gtrsim A
&&\text{if}&(a,b)&\in R_{\text{vert}},
\label{eq:vert-a-deriv}
\\
\frac{\partial}{\partial b}q(a/A,b/B)&\gtrsim C^{-1}B^{-1}
&&&&&&\text{if}
&(a,b)&\in R_{\text{horiz}}.\label{eq:horiz-b-deriv}
\end{align}

To prove the first part of \eqref{eq:vert-a-deriv} it suffices to observe that \( (\partial/\partial a)q(a/A,b/B) \lesssim A^{-1}\), and so on $R_{\text{vert}}$ we have $\frac{\partial}{\partial b}q(a/A,b/B)\lesssim 1/CB$. Since $C$ is very large, from this and \eqref{eq:gradient_lower_bound} it follows that the first part of \eqref{eq:vert-a-deriv}  holds.

We next consider, for each fixed $|\xi|< \eta$, the curve \(q(a/A,b/B)=1+\xi\). This curve contains two points with  $\frac{\partial}{\partial b} q(a/A,b/B)= 0$, one in each component of $R_{\text{vert}}$; let $(a_0,b_0)$ be one of these points. At this point,  we must have
 \begin{equation}
 a_0\gtrsim A,\qquad
     \left.|\frac{\partial}{\partial a} q(a/A,b/B)|\right|_{
    (a,b)=(a_0,b_0)}\gtrsim A^{-1},
\label{eq:partial_a_lower_bound}    
 \end{equation}
 by \eqref{eq:gradient_lower_bound}.
 As we move along the curve \(q(a/A,b/B)=1+\xi\), starting at $(a_0,b_0)$, we have 
 \begin{equation}
     \label{eq:expand-a-in-b-on-vert}
 |a-a_0|\sim (A/B^2)(b-b_0)^2+O((A/B^3)(b-b_0)^3),
 \end{equation}
 and in particular, using \eqref{eq:partial_a_lower_bound}, we have
 \begin{equation}
     \label{eq:expand-a-in-b-on-vert-lower-bound}
     a
     \gtrsim
     A
     +O((A/B^2)(b-b_0)^2).
 \end{equation}
 Taylor expanding around the point $b=b_0$ as functions of $b$, we find that
 \begin{equation}
    \frac{\partial}{\partial b} q(a/A,b/B)= 2q_3 (b-b_0)/B^2 + O(q_2(b-b_0)^2/B^3),\label{eq:partial-b-expansion}
\end{equation}
and so in particular if $(a,b)\in R_{\text{vert}}$ then
\begin{equation*}
    (b-b_0)/B^2
    \lesssim \frac{1}{CB} +O((b-b_0)^2/B^3).
\end{equation*}
Hence, on the component of $R_{\text{vert}}$ containing $(a_0,b_0)$ we have $b-b_0\lesssim B/C$, and together with \eqref{eq:expand-a-in-b-on-vert-lower-bound} this proves the second part of \eqref{eq:vert-a-deriv}.

By the same argument, if \((a,b)\in R_{\text{horiz}}\) then $b-b_0\gtrsim B/C$, and so  we find that \eqref{eq:horiz-b-deriv} holds by \eqref{eq:partial-b-expansion}.
\noindent\newline

\textit{Step 2. Division into neighbourhoods of graphs.}
We now begin to put our counting problem into the form in Lemma~\ref{lem:vdC}. Set
\begin{align*}
E_{\text{vert}}&=E\cap R_{\text{vert}},
\\
E_{\text{horiz}}&=E\cap R_{\text{horiz}},
\\
R'_{\text{vert}}
&=
\{(a+\Delta,b):(a,b)\in E_{\text{vert}},|\Delta|\lesssim \eta A\}.
\\
R'_{\text{horiz}}
&=
\{(a,b+\Delta):(a,b)\in E_{\text{horiz}},|\Delta|\lesssim \eta B\},
\end{align*}
We will now prove that
\begin{equation}
R\subseteq
R'_{\text{vert}}
\cup R'_{\text{horiz}}.\label{eq:graph-decomposition}
\end{equation}

First, if \((a,b)\) belongs to the subset of  \(R_{\text{horiz}}\) defined by
\begin{equation}
    \label{eq:small_horiz_region}
    (a,b)\in R,\qquad
|\frac{\partial}{\partial a}q(a/A,b/B)|
\leq
\frac{CB}{A}|\frac{\partial}{\partial b} q(a/A,b/B)|,
\end{equation}
then we displace $(a,b)$ parallel to the $b$-axis until \(q=1\), finding \(\Delta\) such that $(a,b+\Delta)\in E$.  Since  \((a,b)\in R_{\text{horiz}}\), we expect to be able to take \(\Delta\lesssim \eta B\). This will be possible so long as the point \((a,b+\Delta)\) remains in \(R_{\text{horiz}}\). Indeed if \(\Delta\lesssim \eta B\) and $(a,b)$ are as in \eqref{eq:small_horiz_region} then
\begin{align*}
&    |\frac{\partial}{\partial a} q(a/A,(b+\Delta)/B)| =
  | \frac{\partial}{\partial a}q(a/A,b/B)
   +2q_2\Delta/AB|
   \\
   & \qquad \leq
   \frac{CB}{A}|\frac{\partial}{\partial b} q(a/A,b/B)|
   +
   2q_2|\Delta|/AB
   \\
   & \qquad \leq 
   \frac{CB}{A}|\frac{\partial}{\partial b}q(a/A,(b+\Delta)/B)|
   +\frac{CB}{A}(2q_3|\Delta|/B^2)
   +
   2q_2|\Delta|/AB
   \\
   & \qquad \leq 
   (1+O(\eta C))   \frac{CB}{A}
   |\frac{\partial}{\partial b}q(a/A,(b+\Delta)/B)|.
\end{align*}
Here we used the assumption $\eta \leq C^{-2}$. This proves \eqref{eq:graph-decomposition} in the case \eqref{eq:small_horiz_region}.
 
On the other hand if \((a,b)\in R_{\text{vert}}\) and moreover

\begin{equation}
    \label{eq:small_vert_region}
|\frac{\partial}{\partial a}q(a/A,b/B)|
\geq
\frac{CB}{A}|\frac{\partial}{\partial b} q(a/A,b/B)|,
\end{equation}
then we displace \((a,b)\) horizontally by \(\Delta\) until $(a,b+\Delta)\in E$, and we need to check that for \(\Delta \lesssim \eta A \) the point  $(a,b+\Delta)$ remains in $R_{\text{vert}}$.

Indeed
\begin{align*}
& |\frac{\partial}{\partial b}q((a+\Delta)/A,b/B)| =
  | \frac{\partial}{\partial b}q(a/A,b/B)
   +2q_2\Delta/AB|
   \\
   & \qquad \leq
   \frac{A}{CB}|\frac{\partial}{\partial a} q(a/A,b/B)|
   +
   2q_2|\Delta|/AB
   \\
   & \qquad \leq 
   \frac{A}{CB}|\frac{\partial}{\partial a}q((a+\Delta)/A,b/B)|
   +\frac{A}{CB}(2q_1|\Delta|/A^2)
   +
   2q_2|\Delta|/AB
   \\
   & \qquad \leq 
   (1+O(\eta CB))
   \frac{A}{CB}
   |\frac{\partial}{\partial a}q((a+\Delta)/A,b/B)|.
\end{align*}
This proves \eqref{eq:graph-decomposition} in the remaining case \eqref{eq:small_vert_region}.

\noindent\newline
\textit{Step 3. Computing the second derivative.}
In order to apply Lemma~\ref{lem:vdC} we need to compute \(d^2 a/d b^2\) on \(E_{\text{vert}}\) and  \(d^2 b/d a^2\) on \(E_{\text{horiz}}\). By a computation, on the curve $E$ defined by $q(a/A,b/B)=1$ we have
\begin{align*}
    \frac{d^2 a}{d b^2}
    &=
    \left(\frac{\partial}{\partial b}+ 
    \frac{d a}{d b}\frac{\partial}{\partial a}\right)
    \frac{d a}{d b}
    \\
    &=
    \left(\frac{\partial}{\partial b}-
    \frac{\partial q /\partial b}{\partial q /\partial a}
    \frac{\partial}{\partial a}\right)
    \frac{\partial q/\partial b}{\partial q/\partial a}
    \\
    &=
    \left(\frac{\partial}{\partial b}-
    \frac{A}{B}\frac{ B q_2 a + A q_3 b }{ Bq_1 a + Aq_2 b }
    \frac{\partial}{\partial a}\right)
    \frac{A}{B}\frac{ B q_2 a + A q_3 b }{ Bq_1 a + Aq_2 b }
    \\
    &=
    \frac{A}{B}\frac{A(q_2^2-q_1q_3)(q_1B^2a+2q_2ABab+q_3A^2b^2)
    }{(Bq_1 a + Aq_2 b)^3}
    \\
    &=
    \frac{A^4 B(q_2^2-q_1q_3)q(a/A,b/B)
    }{(Bq_1 a + Aq_2 b)^3}
    \\
    &=
    \frac{q_2^2-q_1q_3}{A^2B^2 (\partial q(a/A,b/B)/\partial a)^3}
    .
\end{align*}
On \(E_{\text{vert}}\) we have \(\partial q/\partial a  \sim A^{-1}\) and hence on \(E_{\text{vert}}\) we have 
\begin{equation}
\label{eq:second-deriv-vert}
    \frac{d^2 a}{d b^2}
    \sim
      \frac{A}{B^2}
\end{equation}
By a precisely equivalent computation, on \(E_{\text{horiz}}\) we have
\begin{equation}
    \label{eq:second-deriv-horiz}
      \frac{B}{A^2}
    \lesssim
    \frac{d^2 b}{d a^2}
    \lesssim
      \frac{C^3 B}{A^2}.
\end{equation}
\vspace{.1in}\newline

\noindent\textit{Step 4. Counting integer points}
By \eqref{eq:second-deriv-vert} in Lemma~\ref{lem:vdC} with \((X,Y,\delta) = (B,A,O(\eta A))\), and defining $G$ implicitly by $a=AG(b/B)$ for $(a,b)\in E_{\textup{vert}}$, we have
\begin{align*}
    |\mathbb Z^2\cap R'_{\text{vert}}|
    &\lesssim
    \eta AB+
    (AB)^{1/3},
\intertext{and similarly, by \eqref{eq:second-deriv-horiz} and Lemma~\ref{lem:vdC} with \((X,Y,\delta) = (A,B,O(\eta B))\) and defining $G$ by $b=BG(a/A)$ on $E_{\textup{horiz}}$, we have}
    |\mathbb Z^2\cap R'_{\text{horiz}}|
    &\lesssim
    \eta AB+
    (AB)^{1/3}.
\end{align*}
Replacing Lemma~\ref{lem:vdC} with Lemma \ref{thm:huxley} gives $131/416$ in place of $1/3,$ in the narrower range of $A,B$.
By \eqref{eq:graph-decomposition}, this completes the proof.
\end{proof}

\subsection{Bound for rank 1 caps}

\SMnote{TODO: The bound $2^s\gtrsim \delta^2 \lambda$ holds, but the bound $2^t\gtrsim \delta^2\lambda$ need not hold. Must check it is never used}
We now return to our problem.

\begin{prop}
\label{prop:rank_1_bound}
Given \(1\lesssim 2^s\lesssim\sqrt{\lambda\delta}\), we have 
$$
\widetilde{\mathcal N}^1_s \lesssim (\lambda\delta)^3 2^{-4s}+
    \lambda^{\frac{1903}{832}}\delta^{\frac{1379}{832}}2^{-\frac{1379}{416}s}.
$$
and
$$
\widetilde{\mathcal N}^1_s \lesssim \lambda^{\frac{7}{3}}\delta^{\frac{5}{3}}2^{-\frac{10}{3}s}.
$$
\end{prop}

\begin{rem}\label{rem:rank1_conjecture}
    We will see that the term $(\lambda\delta)^3 2^{-4s}$ comes ultimately from the term $\eta AB$ in \cref{lem:annular-count}. That term $\eta AB$ typically represents the real order of magnitude of the counting problem in that lemma. Moreover, outside of the application of \cref{lem:annular-count}, there are no obvious losses in the proof of \cref{prop:rank_1_bound}. So it seems plausible that $(\lambda\delta)^3 2^{-4s}$ is the real order of magnitude of $
\widetilde{\mathcal N}^1_s$, at least if we assume that \(1\ll 2^s\ll\sqrt{\lambda\delta}\) to avoid the extreme cases when certain implicit constants might be chosen so that the set counted by $
\widetilde{\mathcal N}^1_s$ is empty.
\end{rem}

Before we begin the proof we will also need two elementary lemmas, the first one giving an identity on matrices, and the second one some basic facts on orthogonal lattices and their bases.

\begin{lem} \label{lem3by3} 
For any $3\times 3$ matrix $M$ we have the identity
\[
M^T((M u)\wedge(M x))
= 
\det(M) (u\wedge x).
\]
\end{lem}
\begin{proof}
Observe that for any \(z=(z_1,z_2,z_3)\) we have the algebraic identities \((Mz)\cdot((M u)\wedge(M x)) = \det(Mz|Mu|Mx) = \det (M)\det(z|u|x) = z\cdot (\det(M) (u\wedge x))) \), and consider the first and last expressions in the chain of identities as linear forms in \(z\).
\end{proof}

\begin{lem} \label{lemmabasis}
Let $u \in \mathbb{Z}^3$ be primitive and $v,w \in \mathbb{Z}^3$ be a basis of the orthogonal lattice $u^\perp = \{v\in \mathbb{Z}^3:u \cdot v=0\}$. 
Then
$$
u = \pm v \wedge w.
$$
Furthermore, this basis $(v,w)$ of $u^\perp$ can be lifted to $p,q \in \mathbb{Z}^3$ such that $(u,p,q)$ forms a basis of $\mathbb{Z}^3$,
$$
v = u \wedge p \;\; \mbox{and} \;\; w = u \wedge q.
$$
\end{lem}

\begin{proof}
Let \(u\in\mathbb{Z}^3\) be primitive, let \(u^\perp=\{v\in \mathbb{Z}^3:u^Tv=0\}\).
Because \(u\) is primitive, solving the linear equations $u_1v_1+u_2v_2+u_3v_3=0$ one finds that \((v_1,v_2)\) must be be congruent  modulo $u_3$ to a multiple of $(u_2,-u_1)$, noting that $(u_2,-u_1)$ is primitive modulo $u_3$. Repeating this argument for $u_1,u_2$ we find that
 \(u^\perp\) is generated by \((u_2,-u_1,0),(u_3,0,-u_1),(0,u_3,-u_2)\), or to say it another way, \(u^\perp = \{u\wedge x:x\in \mathbb Z^3\} \).
 
 We can extend \(u\) to a basis \(u,p,q\) of \(\mathbb{Z}^3\), and let \(u\wedge p = v\) and \(u\wedge q =w\). Since \[u^\perp = \{u\wedge x:x\in \mathbb Z^3\}= \{u\wedge (au+bq+cq):a,b,c\in \mathbb Z\} =\{av+bw: a,b\in\mathbb Z\} ,\] we conclude that \(v,w\) form a basis of \(u^\perp\). Indeed this shows more: it shows that \(x\mapsto u\wedge x\) maps the group \(\mathbb Z^3\) surjectively onto \(u^\perp\) with kernel \(\mathbb Z u=\{au:a\in\mathbb Z\}\).

We claim that if \(v,w\) is any basis of of \(u^\perp\), we can lift \(v,w\) to \(p,q\) with \(u\wedge p = v\) and \(u\wedge q =w\), and moreover that  \(u,p,q\) is then a basis of $\mathbb Z^3$. This in fact follows from the theory of Abelian groups. For \(x\mapsto u\wedge x\) is a surjection from $\mathbb Z^3$ onto \(u^\perp \) with kernel \(\mathbb Z u\). Given any surjection of finitely generated abelian groups $f:A\to B$, any generating set $\{b_i\}$ for $B$, any generating set $\{k_i\}$ for $\ker f$, and any choice of preimages $f(a_i)=b_i$, we must have that $A$ is generated by $\{a_i\}\cup \{k_i\}$. In our case we have $\{b_i\}=\{v,w\}$ and $\{k_i\}=\{u\}$. Hence \(u,p,q\) generates $\mathbb Z^3$.

Finally, to say that \(u,p,q\) is a basis of $\mathbb{Z}^3$ is precisely to say the matrix $(u|p|q)$, with columns $u,p,$ and $q$, is an element of $\GL_3(\mathbb Z)$, in other words an invertible integer matrix. Since \(u,p,q\) is indeed a basis of $\mathbb{Z}^3$, and the triple product \(u\cdot(p\wedge q)\) is precisely given by the determinant \(\det (u|p|q)\), we have \(u\cdot(p\wedge q)=\pm 1\) and hence \(u\cdot(v\wedge w) = u\cdot ((u\wedge p)\wedge(u\wedge q)) = |u|^2 u\cdot(p\wedge q)  = \pm |u|^2\). Since \(\pm v\wedge w\) is a vector in the \(u\)-direction we must have
 \(u=\pm v\wedge w\).
\end{proof}

We can now prove our improved bound.

\begin{proof}[Proof of \cref{prop:rank_1_bound}]
\textit{Step 1. A reduced basis, and four parameters.}
Suppose that \(L\) is an affine line and let \(\Lambda\) be the lattice generated by \((L\cap\mathbb Z^3\cap \theta)-
(L\cap\mathbb Z^3\cap \theta)\); assume this has rank 1. Let $\Lambda=u\mathbb{Z}$, where $u=(u_1,u_2,u_3)\in\mathbb Z^3$ is an integer vector. We claim that that upon writing $\gcd(u)= \gcd(u_1,u_2,u_3)$ we have $\gcd(u)=1$, that is, $u$ is primitive. For $u$ is a difference $v^*-u^*$ between a pair of integer lattice points $u^*,v^*$ belonging to  $\theta \cap L\cap \mathbb Z^3$. Now $v^*=u^*+u$ and so the integer point $w^*=u^*+ u/\gcd(u)$ lies on the line segment joining $u^*$ and $v^*$. Since $u^*$ and $v^*$ lie in $\theta$, it follows that $w^*$ lies in $\theta$ and hence in $\theta \cap \mathbb Z^3$. Hence $u/\gcd(u)= w^*-u^*$ belongs to $ \Lambda$, that is $u/\gcd(u)\in u\mathbb Z$, from which we have that $u$ is primitive.

Observe also that all the \(\sim 2^{s}\) points in \((\theta\cap L\cap\mathbb Z^3)-(\theta \cap L \cap \mathbb Z^3) \) are contained in $\Lambda$, and also in a ball of radius \( \sim \sqrt{\delta \lambda}\), and hence \(|u| \lesssim  2^{-s}\sqrt{\lambda\delta}\).

By Lemma \ref{lemmabasis}, we can extend $u$ to a basis \(u,p,q\) of \(\mathbb Z^3\) such that \(u\wedge p = v\) and \(u\wedge q = w\) form a \emph{reduced} basis of \(u^\perp\), that is $|v|\leq |w|$ and \(|v\cdot w |\leq \frac{1}{2} |v|^2\), so that in particular \(|av+bw|\sim|a|\cdot |v|+|b|\cdot|w|\)

Finally, let $x \in \mathbb{Z}^3 \cap \theta$ and let $a,b \in \mathbb{Z}$ be such that 
\[x=ap+bq+cu.\] 
We claim that knowing $a,b,v,w$
determines \(\theta\) up to \(O(1)\) possibilities. Indeed \(u=\pm v\wedge w\) and
\[\mathbb{Z}^3\cap\theta \subseteq ap+bq+\mathbb{Z} u ,\]
which in turn is contained in the line
\[
-a \frac{u\wedge v}{\|u\|^2}-b\frac{u\wedge w}{\|u\|^2}+\mathbb R u.\]
Any line intersects the shell \(S_{\lambda,\delta}\) in at most two segments, each of which intersects at most \(O(1)\) caps \(\theta\).

\medskip

\noindent \textit{Step 2. A quadratic inequality.} We retain all the assumptions and notation from Step 1.
We can split \(A=\rtA\rtA\) where the formal square root \(\rtA\) is also a symmetric real matrix. 
We begin with the identity
\begin{equation}
\label{eq:contructing_plane_ellipse}
|(\rtA u)\wedge(\rtA x)|^2+ 
|(\rtA u)\cdot (\rtA x)|^2
=
|\rtA u|^2\cdot |\rtA x|^2.
\end{equation}
By writing each term in a different way, we will show that the numbers \(a,b\) defined in step 1, must lie close to an ellipsoid.

At this point, we recall the following notation: if $M$ is a square matrix, then the adjoint matrix is defined as \(M^\ad=\det(M)M^{-1}\). The adjoint quadratic form of $Q$ is then naturally defined as \(Q^\ad(z) = z^T A^\ad z\).

Now observe that, by Lemma \ref{lem3by3},
\[
(\rtA u)\wedge(\rtA x)
=
\rtA^\ad (u\wedge x)
=
\rtA^\ad (av+bw).
\]
This will be our way to rewrite the first term in \eqref{eq:contructing_plane_ellipse}:
\[
|(\rtA u)\wedge(\rtA x)|^2
=|\rtA^\ad (av+bw)|^2
=Q^\ad(av+bw).
\]

For the next term, observe that $(\rtA u)\cdot (\rtA x)=u\cdot Ax $.
Also 
$u,2u,\dotsc, ku\in \theta-\theta$ for \(k\sim 2^s\).

Now $\theta-\theta$ is contained in a box of dimensions $\sim \delta \times \sqrt{\delta\lambda}\times \sqrt{\delta\lambda}$, with the short side in the direction of the unit normal $Ax_\theta/|Ax_\theta|$ to \(\theta\) defined at the start of section~\ref{sec:caps}. Hence,  for \(k\sim 2^s\), we see that $ku$ has a component of size \(O(\delta)\) in the $Ax_\theta/|Ax_\theta|$-direction, and a component of size \(O(\sqrt{\delta\lambda})\) orthogonal to $Ax_\theta/|Ax_\theta|$.

 Observe that $Ax/|Ax|$ is within \(O(\sqrt{\delta/\lambda})\) of the unit normal $Ax_\theta/|Ax_\theta|$. Therefore,  we have
\[
ku\cdot Ax/|Ax|
\lesssim
\delta+\sqrt{\delta/\lambda}\cdot \sqrt{\delta\lambda}=2\delta.
\]
Recalling that $k\sim 2^s$ and $|x| \sim \lambda$, we have
\[
(\rtA u)\cdot (\rtA x)=
u\cdot Ax \lesssim 
\lambda\delta 2^{-s}.\]

It remains to rewrite the third term in \eqref{eq:contructing_plane_ellipse}. We have  \(|\rtA x|^2 = Q(x)=(\lambda+O(\delta))^2\) and so
\[
|\rtA u|^2\cdot |\rtA x|^2
=Q(u)(\lambda+O(\delta))^2.
\]

Substituting back into \eqref{eq:contructing_plane_ellipse} and using that $2^s \lesssim \sqrt{\lambda \delta}$, we find that
\begin{equation}
    \label{eq:ellisoid_rank1}
Q^\ad(av+bw)=\lambda^2
Q(v\wedge w)
+O(\lambda\delta 2^{-s})^2,
\end{equation}
where for each solution $a,b,v,w$ there are $O(1)$ possible $\theta$.

\medskip

\noindent \textit{Step 3. Final count}
Note that in \eqref{eq:contructing_plane_ellipse} we have \((\lambda\delta 2^{-s})^2\lesssim\lambda^2
Q(v\wedge w)\). Hence the region defined by \eqref{eq:ellisoid_rank1} is of the form
\[
|q(a /\lambda|w|,b /\lambda|v|)-1|\leq \eta,
\]
where $q$ is a quadratic form and \(\eta <1\) with \(\eta \sim (\delta /2^{s}|v|\cdot |w|)^2\). Since \(v,w\) are roughly orthogonal, and \(\det(Q^\ad)\gtrsim 1\), we have $Q^\ad(av+bw)\sim |av|^2+|bw|^2$ for all $a,b\in\mathbb R$ and hence \(q(x,y)\sim x^2+y^2\) for all $x,y\in\mathbb R$. Hence  we can apply Lemma~\ref{lem:annular-count} to find that, for given \(v,w\), the number of possible $a,b$ is
\[
\lesssim
\frac{(\lambda\delta 2^{-s})^2}{
|w|\cdot|v|}
+
(\lambda^2|w|\cdot|v|)^{1/3}
\]
provided only that
\[
(\lambda|v|)^{1/2}\leq \lambda|w|\leq (\lambda|v|)^2,
\]
which is always true since \(|v|\leq |w|\) and \(|w|\leq |v|\cdot |w|\lesssim 2^{-s}\sqrt{\lambda\delta}\leq \lambda^{1/2}\).
Moreover this number of possible $a,b$ is
    \[
    \lesssim_{\epsilon}
    \frac{(\lambda\delta 2^{-s})^2}{
|w|\cdot|v|}
+(\lambda^2|w|\cdot|v|)^{131/416+\epsilon}
\]
provided that
\[
    (\lambda|v|)^{147/181+\epsilon}\leq 
    \lambda|w|
    \leq (\lambda|v|)^{181/147-\epsilon},
    \]
    which is to say $|w|
    \leq \lambda^{34/147-\epsilon}|v|^{181/147-\epsilon},$ or equivalently
    \begin{equation*}
    |w|\cdot |v|
    \leq \lambda^{34/147-\epsilon}|v|^{328/147-\epsilon}.
    \end{equation*}

Finally for $V\leq W$ the number of pairs $|v|\sim V, |w|\sim W$ is $\lesssim V^4 W^2$. So given $V,W$, the total number of \((\theta,L)\) such that $|v|\sim V, |w|\sim W $ is
  \begin{align}
      \nonumber
&\lesssim
(\lambda\delta 2^{-s})^2 V^3 W
+
\lambda^{2/3}W^{7/3}V^{13/3}
\\&
=(\lambda\delta 2^{-s})^2 (WV)V^2
+
\lambda^{2/3}(WV)^{7/3}V^{2},
  \label{eq:bad-v-w-bound}
  \end{align}
and moreover the number of pairs $(\theta, L)$ is the improved
  \begin{equation}\label{eq:good-v-w-bound}
\lesssim_\epsilon
(\lambda\delta 2^{-s})^2 V^3 W
+
(\lambda^2 WV)^{131/416+\epsilon} V^4W^2,
      \end{equation}
unless 
  \begin{equation}\label{eq:bad-v-w}
 \lambda^{34/147-\epsilon}V^{328/147-\epsilon}\leq VW.
 \end{equation}
From  the estimates $VW\sim |u| \lesssim\sqrt{\lambda \delta}2^{-s}$ and $V\leq \sqrt{VW},$ we find that the bound \eqref{eq:good-v-w-bound} is (up to $\epsilon$-loss)
\begin{equation}
    \label{eq:rank1caps-main-technical-bound}
\lesssim
(\lambda\delta)^3 2^{-4s}
+
\lambda^{131/208}(\lambda\delta 2^{-2s})^{(1/2)(3+131/416)}
,
\end{equation}
Meanwhile for \eqref{eq:bad-v-w-bound} we have the bound $VW\lesssim\sqrt{\lambda \delta}2^{-s}$. We also have the bound $V\leq \sqrt{VW}$, but from \eqref{eq:bad-v-w}, we have the superior alternative given by (up to $\epsilon$-loss)
\begin{equation*}
V\leq \la^{-\frac{34}{328}}(WV)^{\frac{147}{328}}
= \la^{-\frac{17}{164}}(WV)^{\frac{147}{328}}.
\end{equation*}
Putting these together, for \eqref{eq:bad-v-w-bound} we get a contribution
\[\lesssim
(\lambda\delta)^3 2^{-4s}
+
\lambda^{2/3 -\frac{17}{82}} (\lambda\delta 2^{-2s})^{(1/2)(7/3+147/164}).
\]
  We need to show that the last bound is smaller than the one before, which follows if  
    \[
\lambda^{2/3 - \frac{17}{82}} (\lambda\delta 2^{-2s})^{(1/2)(7/3+147/164})
\leq
\lambda^{131/208}(\lambda\delta 2^{-2s})^{(1/2)(3+131/416)}.
\]
This is
\[
1\lesssim
\lambda^{0.17\dots}
(\lambda\delta 2^{-2s})^{(1/2)(0.85\dots)}.
\]
As $2^s\lesssim \sqrt{\lambda\delta}$, we have $\la\delta 2^{-2s}\gtrsim 1$. The above estimate \eqref{eq:rank1caps-main-technical-bound} therefore always holds in our range of $s$. Observe that this is the first bound from the proposition,
\begin{align*}
\mathcal{N}^1_s &\lesssim
(\lambda\delta)^3 2^{-4s}
+
\lambda^{131/208}(\lambda\delta 2^{-2s})^{(1/2)(3+131/416)}
\\&=
(\lambda\delta)^3 2^{-4s}
+
    \lambda^{\frac{1903}{832}}\delta^{\frac{1379}{832}}2^{-\frac{1379}{416}s}.
    \end{align*}

It remains to check that \((\lambda\delta)^3 2^{-4s}+    \lambda^{\frac{1903}{832}}\delta^{\frac{1379}{832}}2^{-\frac{1379}{416}s}\lesssim
    \lambda^{\frac{7}{3}}\delta^{\frac{5}{3}}2^{-\frac{10}{3}s}\). This amounts to
    \[
    \lambda^{9-7}\delta^{9-5}2^{(10-12)s}\lesssim 1,
    \]
    and
    \[
    2^{\frac{23}{1248}}s
    \lesssim
    \delta^{\frac{23}{2496}}
    \lambda^{\frac{35}{2496}},
    \]
    which holds since $2^s\lesssim \sqrt{\lambda\delta}$.
\end{proof}

\section{The approach by cap counting and decoupling}
\label{sec:dec}

We now combine the estimates on the number of caps with many points obtained in the previous part with $\ell^2$ decoupling \cite{BourgainDemeter3} to bound $P_{\lambda,\delta}$.

Throughout this section, we suppose $p>4$ since the case $p\leq 4$ was already treated in \cite{Pezzi}.

Our main results here will be Lemma \ref{lem:PfewBound} and Lemma \ref{lem:PmanyCaps}. These results will also be used in \cref{sec:epsilonremoval2,sec:proofmaintheorem} to prove sharp estimates.  To give an indication of the strength of these lemmas we prove the following theorem now which will be subsumed by \cref{mainresult}.

\begin{thm}\label{thm:caps}
The conjecture holds with \(\epsilon\)-loss if \eqref{eq:expected_number_of_points} holds and
\begin{itemize}
    \item \(\delta > \lambda^{-\frac{1}{2}}\); or
    \item $4<p\leq 5$ and \(\delta >\max\{\la^{- \frac{6-p}{2}},\la^{\frac{27p-316}{104(p-2)}}\}\).
\end{itemize}
\end{thm}

The proof will proceed by splitting $P_{\lambda,\delta}^s$ into projectors on caps, and gathering caps with similar number of points. 
To be more specific, we choose a partition of unity $(\chi_\theta)$ adapted to the caps:
$$
\operatorname{Supp} \chi_\theta \subset \theta \qquad \mbox{and} \qquad \sum_\theta \chi_\theta = 1 \quad \mbox{on $S_{\lambda, \delta}$},
$$
and let
\begin{align*}
& \chi_s = \sum_{\theta \in \mathcal{C}_s} \chi_\theta,\\
& P_{\lambda,\delta}^s = P_{\lambda,\delta}\, \chi_s(D),
\end{align*}
where $\chi_s(D)$ is the Fourier multiplier with symbol $\chi_s(k)$.

We will distinguishing between caps containing few and many points. For us, \say{few} will mean either $O(1)$ points, or not many more points than average. In other words, we decompose
$$
P_{\lambda,\delta} = \underbrace{\sum_{2^s \lesssim \lambda^2 \delta + 1} P^s_{\lambda,\delta}}_{\displaystyle P_{\operatorname{few}}} + \underbrace{\sum_{2^s \gg \lambda^2 \delta + 1} P^s_{\lambda,\delta}}_{\displaystyle P_{\operatorname{many}}}.
$$

\subsection{Caps with few points}\label{sec:capsFew}

\begin{lem}\label{lem:PfewBound}
Assume that \eqref{eq:expected_number_of_points} holds. Then
\begin{enumerate}
    \item The operator $P_{\operatorname{few}}$ satisfies the conjecture with an $\epsilon$-loss provided that $\delta>\la^{\frac{p}{2}-3}$.
    \item The operator $P_{\operatorname{few}}$ satisfies the sharp version of the conjecture provided that $\delta>\la^{\frac{p}{2}-3+\kappa}$ and $\delta<\la^{2-\frac{p}{2}-\kappa}$ for some fixed $\kappa>0$.
\end{enumerate}

\end{lem}

The sharpness result will not be used in the proof of \cref{thm:caps}. However, it will be used in the proof of \cref{mainresult} as detailed in \cref{sec:proofmaintheorem}.
\begin{proof}
We start with the $L^2 \to L^4$ bound in \cite{GermainMyerson1},
\begin{equation}
\label{L4bound}
\| P_{\lambda,\delta}^{s}\|_{L^2 \to L^4} \lesssim_\epsilon \lambda^\epsilon 2^{s/4},
\end{equation}
which is proved via decoupling.
This implies
\begin{equation}
\label{Pfew}
\| P_{\operatorname{few}} \|_{L^2 \to L^4} \lesssim_\epsilon \lambda^\epsilon (\lambda^{1/4} \delta^{1/2} + 1).
\end{equation}
We also have the estimate
\[
\| P_{\operatorname{few}} \|_{L^2 \to L^\infty} \lesssim 
\sqrt{\#(\mathbb Z^3\cap S_{\lambda,\delta})},
\]
which is $O_{\epsilon}(\lambda^{1+\epsilon}\delta^{1/2})$
as we assumed \eqref{eq:expected_number_of_points} in the lemma. We thus have
$$
\| P_{\operatorname{few}} \|_{L^2 \to L^\infty} \lesssim_\epsilon  \lambda^{1+\epsilon} \delta^{1/2}.
$$
Interpolating between these two estimates gives, if $4 \leq p \leq \infty$,
$$
\| P_{\operatorname{few}} \|_{L^2 \to L^p} \lesssim_\epsilon \lambda^\epsilon \left[ \lambda^{1- \frac 3p} \delta^{1/2} + (\lambda \delta^{1/2})^{1 - \frac 4p} \right].
$$
Note that the first term on the right-hand side also appears in the conjecture, while so long as $p>4$, the last term does not.
Conjecture \ref{conjA} is thus verified for $P_{\operatorname{few}}$ with $\epsilon$-loss, provided
$$
(\lambda \delta^{1/2})^{1 - \frac 4p}
\leq
(\lambda\delta)^{\frac 12 - \frac 1p},
$$
in other words
$
\lambda^{\frac 12 -\frac 3p}
\leq
\delta^{1/p}
$, which matches the assumption in the lemma.

For the sharp version of the conjecture to hold we need
$$
\lambda^\epsilon
\left[
\lambda^{1-\frac 3 p}\delta^{1/2}+
(\lambda \delta^{1/2})^{1 - \frac 4p}
\right] 
 = I + II<
(\lambda\delta)^{\frac 12 - \frac 1p},
$$
Comparing $I$ to the right hand side gives the condition

\begin{equation*}
    \lambda^\epsilon\lambda^{1-\frac 3 p}\delta^{1/2}<(\lambda\delta)^{\frac 12 - \frac 1p}.
\end{equation*}
That is, we must have
\begin{equation*}
    \delta < \la^{2-\frac{p}{2}-\kappa}.
\end{equation*}
This condition  requires us to be in the interior of the geodesic focusing region, that is, to the right of but bounded away from the \say{critical curve} the appears in \cref{conjA}.

Comparing $II$ to the right hand side gives

\begin{equation*}
    (\lambda \delta^{1/2})^{1 - \frac 4p}<(\lambda\delta)^{\frac 12 - \frac 1p}.
\end{equation*}
This is true if and only if

\begin{equation*}
    \delta>\la^{\frac{p}{2}-3+\kappa}.
\end{equation*}

We thus find that $P_{\operatorname{few}}$ obeys the sharp sharp conjecture in the interior of the geodesic focusing region if, additionally, $\delta>\la^{\frac{p}{2}-3+\kappa}$. Proving sharp estimates for $P_{\operatorname{few}}$ in the point focusing region will be the subject of \cref{sec:epsilonremoval1}.
\end{proof}

\begin{rem}
To bound $P_{\operatorname{few}}$, the improved estimates on cap counting cannot help, and the limiting factor is $\ell^2$ decoupling itself.
The first term in \eqref{Pfew} represents the average number of lattice points in a cap. When $\la\delta^2$ is less than $1$, which happens when $\delta < \lambda^{-1/2}$, the average number of lattice points in a cap is less than 1. However, caps cannot have a fractional number of lattice points in them, so the bound $1$ in \eqref{Pfew} is used instead, which gives the restriction $\delta > \la^{\frac{p}{2}-3}$ stated above. The $L^\infty$ estimates we interpolate with are essentially sharp.
\end{rem}

\subsection{Caps with many points}\label{sec:capMany}

We have just proved the conjecture for $P_{\operatorname{few}}$ under some hypotheses.
For the full conjecture to hold we need bounds for both $P_{\operatorname{many}}$ and $P_{\operatorname{few}}$.

Recall that $P_{\operatorname{many}}$ consists of caps for which  $2^s \gg \lambda \delta^2 + 1$. If $\theta \in\mathcal C_s$, so that $N_\theta \sim 2^s$, we must have $N_\theta>1$, and hence \(\theta\) has rank $\geq 1$.

Denoting $\mathcal{N}_s$ for the number of caps with $\sim 2^s$ points, we have the trivial $L^2 \to L^\infty$ bound $\| P_{\lambda,\delta}^s \|_{L^2 \to L^\infty} \lesssim (\mathcal{N}_s 2^s)^{1/2}$. Interpolating it with \eqref{L4bound} gives
$$
\| P_{\lambda,\delta}^s \|_{L^2 \to L^p} \lesssim \lambda^\epsilon (\mathcal{N}_s)^{\frac 12 - \frac 2p} 2^{s \left(\frac{1}{2} - \frac 1p \right)}
$$
We want to prove \cref{conjA} for $P_{\operatorname{many}}$. The contribution from $\mathcal{N}^r_s$ will be satisfactory if, for some nonnegative $\nu\geq 0$, we have
\begin{align}
\mbox{either} & \qquad \mathcal{N}^r_s\lesssim
\lambda^{-\nu}
(2^{-s}\lambda\delta)^{\frac{p-2}{p-4}};
\label{eq:sufficient_geodesicfocussing} \\
\mbox{or} & \qquad \mathcal{N}^r_s \lesssim
\lambda^{-\nu}
(2^{-s}\lambda\delta)^{\frac{p-2}{p-4}}
\lambda \delta^{\frac{2}{p-4}}.
\label{eq:sufficient_pointfocussing}
\end{align}
This would give the conjecture with $\epsilon$-loss if $\nu=0$, or without loss if $\nu>0$.

By Propositions \ref{prop:rank_1_bound} and \ref{prop:rank2bound},
\begin{align}
\label{eq:expsums_rank1bound_improved}
    \mathcal{N}^1_s &\lesssim 
    (\lambda\delta)^3 2^{-4s}+
    \lambda^{\frac{{1903}}{832}}\delta^{\frac{1379}{832}}2^{-\frac{1379}{416}s},
    \\
\label{eq:expsums_rank1bound}
    \mathcal{N}^1_s &\lesssim (2^{-10s}\lambda^7\delta^5)^{1/3},
    \\
\label{eq:geomnos_rank2bound}
    \mathcal{N}^2_s &\lesssim (2^{-s}\lambda\delta)^{3},
    \\
\label{eq:geomnos_rank2bound_improved}
    \mathcal{N}^2_s &\lesssim_\epsilon{
    (2^{-s}\lambda\delta)^{2}+
    } \lambda^\epsilon\delta (2^{-s}\lambda\delta)^{4},  
    \\
\label{eq:geomnos_rank2bound_incidence}
    \mathcal{N}^2_s &\lesssim 
    (\lambda\delta)^3 2^{-5s/2}+\lambda^{\frac{{1903}}{832}}\delta^{\frac{1379}{832}}2^{-\frac{1795}{832}s}.
\end{align}

\begin{SMlongnote}{Note to self}
    We always have at most $O(\lambda\delta^{-1})$ caps in total, and hence the total number of lattice points in all caps with $N_\theta \lesssim 1+\lambda\delta^2$ must be $\lesssim \lambda\delta^{-1}+\lambda^2\delta$. Let's suppose that $\delta\ll \lambda^{-1/2}$, so we are looking there at caps containing exactly one point. Now consider caps on a slightly thicker shell $\delta'>\delta$. According to the bounds above, the total number of integer points contained in all rank 1 caps is
    $$
    \lesssim
    (\lambda\delta')^3
    + \lambda^{\frac{1903}{832}}\delta'^{\frac{1379}{832}}
    $$, while the total number of integer points contained in all rank 2 caps is
    $
    \lesssim
    \min\{(\lambda\delta')^3,(\lambda\delta')^2+\lambda^{4+\epsilon}\delta'^5\}
    $. Now any pair of points not counted in either of those, must be either in a rank 3 cap or separated by $>\sqrt{\delta'\lambda}$. And there are at most $\lambda\delta'^{-1}$ points in a $\sqrt{\delta'\lambda}$-separated set. Can we take $\lambda\delta'^{-1}<\lambda^2\delta$? That means $\delta' > (\lambda\delta)^{-1}$ which is a bit bigger than $\lambda^{-1/2}$. We could get a good bound if we could deal with the rank 3 big caps. A typical such cap is OK, but some of them might be unbalanced. The balanced ones give us exactly what we expect, so we would need to show there are no too many unbalanced ones. One might reasonably expect that we are OK if $\delta(\lambda\delta') > (\lambda\delta'^2)^{(3/2)/3}$ which would correspond to $ \delta > \lambda^{-1/2}$
\end{SMlongnote}

We will not make use of the last two inequalities as the limiting factor is rank 1 caps. We include them in any case since they could be of interest in future works.

\Cref{thm:caps} will follow at once from \cref{lem:PfewBound}, together with the following lemma.

\begin{lem}\label{lem:PmanyCaps}
Let $C>0$ and $\kappa \geq 0$ be parameters. Suppose that $2^s\geq C\lambda\delta^2$, $2^s\geq C$ and $\lambda\delta\geq C $.
We say the contribution of $\mathcal{N}^r_s$ is satisfactory if either \eqref{eq:sufficient_geodesicfocussing} or \eqref{eq:sufficient_pointfocussing} holds, for some $\nu\geq 0$, with an implicit constant depending on $C$. If $\kappa=0$, put $\nu=0$, and, if $\kappa>0$, assume that $0<\nu \ll_p \kappa$.
\begin{enumerate}
    \item 
    The contribution $\mathcal{N}^2_s$ from rank 2 caps is satisfactory if $p\leq 5$ and $C2^s\leq \lambda^{1-\kappa}\delta$, and also when $p>5$ and $\delta>\la^{-1/2+\kappa}$.

    \item The contribution $\mathcal{N}^1_s$  from rank  1 caps is satisfactory when  $4<p<\frac{4684}{963}=4\frac{832}{963}$ and $\delta>\la^{\frac{27p-316}{104(p-2)}+\kappa}$, when $4\frac{832}{963}=\frac{4684}{963}<p<{5}$ and $\delta>\lambda^{-\frac{5948-1071p}{3852-547p}+\kappa}$, and also when $p >{5}$ and $\delta>\lambda^{-1/2+\kappa}$.
\end{enumerate}
\end{lem}
The threshold of $\la^{\frac{-5948-1071p}{3852-547p}}$ is never the dominant term in \cref{thm:caps} when $p\leq 5$ as the condition $\delta\geq \lambda^{-3+p/2}$ from \cref{lem:PfewBound} is a stronger restriction.

\begin{proof}
The rank 2 caps prove easier, so we treat them first. Throughout the proof, all implicit constants can depend on $C$.

\medskip\noindent\textit{Rank 2 caps when $p\geq 5,\delta >\lambda^{-1/2+\kappa}$.}
Comparing \eqref{eq:sufficient_pointfocussing} and \eqref{eq:geomnos_rank2bound}, we need to show
\[
\lambda^{(p-4)\nu}
\lesssim
(2^{-s}\lambda\delta)^{(p-2)-3(p-4)}
\lambda^{p-4}\delta^{2}
\]
that is
\[
\lambda^{(p-4)\nu}\lesssim
2^{(2p-10)s} \lambda^{6-p} \delta^{12-2p}
\]
which is satisfied since $2^s \gtrsim \lambda \delta^2 \gtrsim \lambda^{2\kappa}$.

\medskip\noindent\textit{Rank 2 caps when $p<5$.} Comparing \eqref{eq:sufficient_geodesicfocussing} and \eqref{eq:geomnos_rank2bound}, we are done provided that
\[
\lambda^\nu\lesssim
(2^{-s}\lambda\delta)^{\frac{p-2}{p-4}-3}.
\]
Now $4<p<5$ so the exponent $\frac{10-2p}{p-4}$ is positive; since \(2^s\lesssim \lambda^{1-\kappa}\delta\), we are done.

\medskip\noindent\textit{Rank 1 caps when $p>4,\delta >\lambda^{-1/2+\kappa}$.}
We compare \eqref{eq:sufficient_pointfocussing} and \eqref{eq:expsums_rank1bound}, and find a satisfactory bound provided
\[
\lambda^{3(p-4)\nu}
(2^{-10s}\lambda^7\delta^5)^{p-4}
\lesssim
2^{-(3p-6)s}
\lambda^{6(p-3)}
\delta^{3p}
\]
which is
\[
\lambda^{3(p-4)\nu}
(\lambda\delta^2)^{p-10}
\lesssim
2^{(7p-34)s}.
\]
The exponent on $2^s$ is positive so we bound $2^s\gtrsim \lambda\delta^2$, and we need to show
\[
\lambda^{3(p-4)\nu}\lesssim
(\lambda\delta^2)^{6p-24}
\]
which is true as $p>4,\delta >\lambda^{-1/2+\kappa}$.

\medskip\noindent\textit{Rank 1 caps when \(4<p<\frac{4684}{963}, \delta>\lambda^{\frac{27p-316}{104(p-2)}+\kappa}.\)}
Note that $2^s\lesssim\sqrt{\lambda\delta}$, or else $\mathcal N^1_s$ vanishes.
Note also that 
we have $p>4$ and $\delta>\lambda^{\frac{27p-316}{104(p-2)}+\kappa}$ and thus we have $\delta>\lambda^{\kappa-1}$, which will be useful throughout the argument.
Combining \cref{eq:expsums_rank1bound_improved} with \eqref{eq:sufficient_geodesicfocussing}, it suffices to prove that
\begin{equation*}
\lambda^{(p-4)\nu}
    (\lambda\delta)^{3(p-4)} 2^{-4(p-4)s}
    \lesssim (2^{-s}\lambda\delta)^{p-2},
\end{equation*}
and that
\begin{equation}\label{eq:tricky-rank1-target}
\lambda^{832(p-4)\nu}
\lambda^{1903(p-4)}(\delta 2^{-2s})^{1379(p-4)}
    \lesssim (2^{-s}\lambda\delta)^{832(p-2)}.
\end{equation}
The first of these is
\begin{equation}\label{eq:rank1_mainterm}
\lambda^{(p-4)\nu}
2^{(14-3p)s}
\lesssim
(\lambda\delta)^{10-2p},
\end{equation}
{
which is true if \(14/3<p\leq 5\). Otherwise, if $4<p\leq 14/3$, since \(2^s\lesssim\sqrt{\lambda\delta}\), it holds if
\(
\lambda^{(p-4)\nu}
\lesssim
(\lambda\delta)^{3-p/2},\) which is true as $p\leq 6$.}
Thus it remains to prove \eqref{eq:tricky-rank1-target}, which is to say
\[
\lambda^{832(p-4)\nu}
2^{(9368-1926p)s}
\lesssim
\lambda^{5948-1071p}\delta^{3852-547p}.
\]
When $4<p<\frac{9368}{1926}=\frac{4684}{963}$, the exponent on $s$ is positive. Therefore, we can use the bound $2^s\lesssim\sqrt{\lambda\delta}$, so that we need to prove
\[
\lambda^{832(p-4)\nu}
\lesssim
\lambda^{1264-108p}
\delta^{416p-832}
=
\lambda^{4(316-27p)}\delta^{104(4p-8)}.
\]
Here the exponent of $\delta$ is $>0$ and so this becomes exactly our assumption \(\delta>\lambda^{\frac{27p-316}{104(p-2)}+\kappa}.\)

\medskip\noindent\textit{Rank 1 caps when \(4{\frac{832}{963}=\frac{4684}{963}\leq p\leq 5}.\)}
We proceed as in the previous case. As before, equation \eqref{eq:rank1_mainterm} holds since $14/3<p<5$. It remains to prove the last display above.
As the exponent on $2^s$ is now nonpositive, we will use the bound $2^s\gtrsim 1$, so that we need to prove
\[
\lambda^{832(p-4)\nu}
\lesssim
\lambda^{5948-1071p}\delta^{3852-547p},
\]
this is $\delta > \lambda^{-\frac{5948-1071p}{3852-547p}+\kappa}$.

We have now covered the full range for $p$.

\medskip\noindent\textit{Rank 1 caps when compared to the point focusing bound.} Finally, we record the comparison of \cref{eq:sufficient_pointfocussing} and \cref{eq:expsums_rank1bound_improved} for completion. This yields, using the first term,

\begin{equation*}
    2^{(14-3p)s}\lesssim(\la\delta)^{10-2p}\la^{p-4}\delta^2.
\end{equation*}
The range we are interested in is $p>5$ which means we use the bound $2^s\gtrsim 1$.

\begin{equation*}
    1\lesssim(\la\delta)^{10-2p}\la^{p-4}\delta^2 = (\la\delta^2)^{6-p}.
\end{equation*}

This only holds for $\delta\lesssim \la^{-\frac{1}{2}}$ for $p>6$, which is beyond the range we consider. We do not record the second term as it is superfluous.
\end{proof}

\begin{rem}
    \label{rem:obstruction}
Let us take a moment to review what we found in \cref{lem:PmanyCaps}.

The range $\delta > \la^{-\frac{1}{2}}$ is particularly amenable to our methods because an average cap contains more than 1 lattice point, and so decoupling enables us to prove almost sharp results for $P_{\operatorname{few}}$. Our counting argument is good enough to bound rank 1 and rank 2 caps in this regime. In dimensions higher than 3, the analogous range would be $\delta > \la^{-\frac{d-1}{d+1}}$; the cap-counting estimates from \cite{GermainMyerson1} would need to be improved in order to capture the whole of that range.

For rank 1 and 2 caps, the $2^s\sim 1$ term dominates the expression when $p>5$. 

To elaborate further, when $\delta \sim \lambda^{-1/2}$, then according to our bounds \eqref{eq:expsums_rank1bound_improved}-\eqref{eq:geomnos_rank2bound_incidence} there could be as many as $\lambda^{3/2}$ caps of rank 1 or 2  with $\sim 1$ points. For $p>5$ this exactly gives us the conjecture with no ''wiggle room'', and that lack of room prevents us from obtaining the conjecture by decoupling if $p>5$ and $\delta <\lambda^{-1/2-\epsilon}$.
This is suspected  to be the true order of magnitude for rank 1 caps, see \cref{rem:rank1_conjecture}, while the true number of rank 2 caps may be smaller.

For \cref{lem:PfewBound}, the limit of the method comes partly from needing a sharp bound for $\#\mathbb Z^3\cap S_{\lambda,\delta}$, and partly  from applying decoupling to estimate functions with Fourier support that does not include every cap in the decomposition. We are not taking advantage of the fact that many caps have no lattice points. Even with generous assumptions on the set of frequencies, that can be difficult. See \cite{Demeter} for some results and open questions in this direction, concerning decoupling type bounds beyond $p=4$. Exponential sum estimates and height splitting, the subjects of \cref{sec:exponential} and \cref{sec:epsilonremoval1} respectively, provide a way to break this barrier slightly for our problem.
\end{rem}

\section{The approach by exponential sums}
\label{sec:exponential}

We adapt here the approach laid out in \cite{DemeterGermain} to the case of dimension 3 and prove the following theorem.

\begin{thm}\label{expSumThm} Consider a general quadratic form $Q$. Then Conjecture \ref{conjA} is satisfied if
$$
\delta > \min \left( \la^{-\frac{p-4}{2p-2}}, \la^{-\frac{85p-358}{158p-128}} \right).
$$
\end{thm}
The first bound is not used in our main theorem, but we include it in the statement here as it represents what can be achieved with trivial estimates using this technique. There is an additional limitation on when we can use the second bound that will be detailed in the proof but, as it only forbids us from using the estimate other bounds are superior, it will play no role in our results. In our main theorem, the above result is only used when $\delta<\la^{-\frac{1}{2}+\kappa}$ for $\kappa>0$ by the independent results of \cref{sec:epsilonremoval1}.

\begin{proof} \noindent \textit{The mollified convolution kernel.} It is more convenient to work with a mollified version of the projector $P_{\lambda,\delta}$: define the Fourier multiplier
$$
P_{\lambda,\delta}^\sharp = \chi_{\lambda,\delta} (D) ,
$$
where the function $\chi_{\lambda,\delta}$ is a 
nonnegative smooth cutoff function adapted to the shell $\mathcal{A}_{\lambda,\delta}$. To be more explicit, consider a positive function $\chi$ whose Fourier transform is compactly supported and let
$$
\chi_{\lambda,\delta} = \delta^{-2} \chi (\delta^{-1} \cdot) * d\sigma_\lambda,
$$
where $d\sigma_{\lambda}$ is the superficial measure on the ellipsoid 
$$\mathcal{S}_\lambda = \{ x \in \mathbb{R}^3, \; Q(x) = \lambda^2 \}.$$
The Fourier transform (on $\mathbb{R}^3$) of $\dd \sigma_1$ is such that
$$
m (\xi) = \widehat{\dd \sigma_1}(\xi)
$$
satisfies
\begin{equation}
\label{expansionm}
m(\xi) \sim \mathfrak{Re}\;  e^{i \Theta(\xi)} \sum_{j=1}^\infty |\xi|^{-j} a_j(\xi) \quad \mbox{as $\xi \to \infty$},
\end{equation}
the phase $\Theta$ is smooth and homogeneous of degree one and $|\nabla^k a_j(\xi)| \lesssim_k |\xi|^{-k}$, see \cite{Stein}, VIII.B (the real part in the above formula accounts for the two points on the ellipse whose normal has a given direction, and we will omit it going forward). In particular, $m$ decays like $|\xi|^{-1}$. 

Since $\dd \sigma_\lambda = \lambda^{-1} \dd \sigma_1(\la^{-1}\cdot)$, we have that
\begin{equation}
\label{FT}
\widehat{\chi_{\lambda,\delta}}(\xi) = \lambda^2 \delta m(\lambda \xi) \widehat{\chi}(\delta \xi).
\end{equation}

\medskip

\noindent \textit{Dyadic decomposition.}
We now introduce a dyadic partition of unity, namely nonnegative functions $\varphi$ and $\psi$ in $\mathcal{C}_0^\infty$ such that
\begin{equation}
\label{dyadicpartitionofunity}
\varphi(y) + \sum_{M \in 2^\mathbb{N}} \psi \left( \frac{y}{M} \right) =1 \qquad \mbox{for all $y \in \mathbb{R}^3$} 
\end{equation}
and furthermore $\varphi$ is supported in a ball, and $\psi$ in an annulus.

Denoting $P_{\lambda,\delta}^\sharp(x)$ the kernel of $P_{\lambda,\delta}^\sharp$, we apply successively the Poisson summation formula and the partition of unity above to get
\begin{align*}
P_{\lambda,\delta}^\sharp(x) & = \sum_{k \in \mathbb{Z}^3} \chi_{\lambda,\delta}(k) e^{2\pi i k \cdot x} = \sum_{m \in \mathbb{Z}^3} \widehat{\chi_{\lambda,\delta}} (m - x)\\
& = \sum_{m \in \mathbb{Z}^3} \varphi(m - x) \widehat{\chi_{\lambda,\delta}} (m - x) + \sum_{m \in \mathbb{Z}^3} \sum_{M \in 2^\mathbb{N}} \psi \left( \frac{m-x}{M} \right) \widehat{\chi_{\lambda,\delta}} (m - x) \\
& = P_{\lambda,\delta}^{\sharp,0}(x) + \sum_{\substack{M \in 2^\mathbb{N}\\M\delta\lesssim 1}} P_{\lambda,\delta}^{\sharp,M}(x).
\end{align*}

\noindent \textit{Bounding the leading term $P_{\lambda,\delta}^{\sharp,0}$.} We consider first the operator on $\mathbb{R}^3$ with kernel $\varphi \widehat{\chi_{\lambda,\delta}}$. Modifying the cutoff function, this kernel can be written as $\delta \widehat{\chi_{\lambda,1}}$, stemming from the symbol $\delta \chi_{\lambda,1}$. By the Stein-Tomas theorem (see \cite{GermainMyerson1} for the version we need here), the corresponding operator has norm $\lesssim \lambda^{2 - \frac 6p} \delta$ from $L^{p'}$ to $L^p$ if $p \geq 4$. By locality, the same bound holds for $P_{\lambda,\delta}^{\sharp,0}$.

\medskip

\noindent \textit{Bounding the noise $P_{\lambda,\delta}^{\sharp,M}$.} 
The plan is now to interpolate between $L^2 \to L^2$ and $L^1 \to L^\infty$ bounds.

To obtain an $L^2 \to L^2$ bound, we deduce from the definition of the kernel $\Phi_{\lambda,\delta}^{\sharp,M}$ and Poisson summation that $P_{\lambda,\delta}^{\sharp,M}$ is the Fourier multiplier on the torus with symbol $M^3 \widehat{\psi}(M\cdot) * \chi_{\lambda,\delta}$. Therefore,
\begin{equation}
\label{L2toL2}
\| P_{\lambda,\delta}^{\sharp,M} \|_{L^2 \to L^2} \lesssim \| M^3 \widehat{\psi}(M\cdot) * \chi_{\lambda,\delta} \|_{L^\infty} \lesssim M \delta \qquad \text{if }M \lesssim \delta^{-1}.
\end{equation}
To obtain the $L^1 \to L^\infty$ bound, we rely on exponential sum estimates. By \eqref{expansionm}, and assuming that \(M\lesssim \delta^{-1}\), we need to bound
$$
S_{\lambda,\delta}^M(x) = \lambda \delta \sum_{\substack{n \in \mathbb{Z}^3}} \psi\left( \frac {n-x} M \right) a_1(\lambda(n-x)) \frac{e^{i \lambda \Theta(n-x)}}{|n - x |}
$$
(we only kept the leading order term in \eqref{expansionm}, since higher order terms are easier to control). A first obvious bound is
\begin{equation}\label{L1LinftyExpSumBound}
| S_{\lambda,\delta}^M(x) | \lesssim \lambda \delta M^{2}.
\end{equation}
We can also use results from \cite{Guo}. By arguments in that paper, we have
\begin{equation}\label{importedBoundExpSums}
    |S_{\lambda,\delta}^M(x)|\lesssim \delta \la^{\frac{103}{94}}M^{2-\frac{30}{94}} \qquad \mbox{if $M > \lambda^{\frac 27}$}.
\end{equation}
See \cref{importedSumsLemma} below for the details of this bound.

Interpolating between \eqref{L2toL2} and \eqref{L1LinftyExpSumBound} and summing over $M\lesssim \delta^{-1}$ gives the following bound for any $\delta> \lambda^{-1}$
\begin{equation*}
\sum_{1 \leq M \lesssim \delta^{-1}} \| P_{\lambda,\delta}^{\sharp,M} \|_{L^{p'} \to L^p} \lesssim \sum_{1 \leq M \lesssim \delta^{-1}} (M\delta)^{\frac{2}{p}} (\lambda \delta M^2)^{1-\frac 2p} \sim \lambda^{1 - \frac 2p} \delta^{\frac 2p -1}
\end{equation*}

If $\delta < \lambda^{-\frac 27}$, we can split the sum into two pieces for which we interpolate between \eqref{L2toL2} and \eqref{L1LinftyExpSumBound} or \eqref{importedBoundExpSums} depending on the value of $M$. This gives
\begin{align*}
& \sum_{1 \leq M \lesssim \delta^{-1}} \| P_{\lambda,\delta}^{\sharp,M} \|_{L^{p'} \to L^p} = \Bigl(\sum_{1 \leq M \lesssim \la^{\frac{2}{7}}}  + \sum_{\la^{\frac{2}{7}}\lesssim M \lesssim \delta^{-1}}\Bigr) \| P_{\lambda,\delta}^{\sharp,M} \|_{L^{p'} \to L^p} \\
& \qquad \lesssim \sum_{1 \leq M \lesssim \la^{\frac{2}{7}}} (M\delta)^{\frac{2}{p}} (\lambda \delta M^2)^{1-\frac 2p} + \sum_{\la^{\frac{2}{7}}\lesssim M \lesssim \delta^{-1}} (M\delta)^{\frac{2}{p}} (\delta \la^{\frac{103}{94}}M^{2-\frac{30}{94}} )^{1-\frac 2p} \\
& \qquad \sim \lambda^{\frac{11}{7} - \frac{18}{7p}} \delta + \lambda^{\frac{103}{94} \left(1-\frac 2p \right)} \delta^{-\frac {32}{47} \left(1-\frac 2p \right)} \sim \lambda^{\frac{103}{94} \left(1-\frac 2p \right)} \delta^{-\frac {32}{47} \left(1-\frac 2p \right)}.
\end{align*}

Overall, we found that
$$
\sum_{1 \leq M \lesssim \delta^{-1}} \| P_{\lambda,\delta}^{\sharp,M} \|_{L^{p'} \to L^p} \lesssim 
\begin{cases}
\lambda^{1 - \frac 2p} \delta^{\frac 2p -1} & \mbox{if $\delta > \lambda^{-1}$} \\
\lambda^{\frac{103}{94} \left(1-\frac 2p \right)} \delta^{-\frac {32}{47} \left(1-\frac 2p \right)}  & \mbox{if $\lambda^{-1} < \delta < \lambda^{-\frac{2}{7}}$}.
\end{cases}
$$
This is $\lesssim \la^{2-\frac{6}{p}}\delta$ if
\begin{equation*}
\delta > \min \left( \la^{-\frac{p-4}{2p-2}}, \la^{-\frac{85p-358}{158p-128}} \right),
\end{equation*}
which is the desired result.
\end{proof}
\noindent
We now quickly detail the bound \eqref{importedBoundExpSums}. First, we import the key result of \cite{Guo}. We have not altered the notation from the original paper in the following statement. That is done in the proof below.

\begin{lem}[Theorem 2.6, \cite{Guo}, case $d=3$]\label{GuoBoundLemma}
Let $q\in\{1,2\}$. Define
    \begin{equation*}
        S(T,M;G,F) = \sum_{m\in \Z{d}}G(\frac{m}{M})e(TF(\frac{m}{M})).
    \end{equation*}
    Here, $G$ is a smooth bump function supported in the unit ball and $F$ is a smooth function on a convex domain $\Omega$ that contains the support of $G$. Let $\delta>0$ be a parameter. Assume $M>\max(1,\delta^{-1})$, $dist(\supp{G},\Omega)\gtrsim \delta$, and that for all $x\in \Omega$ and $\nu\in \N^3$ with $|\nu|\leq 11$ we have
    \begin{align*}
        &|(D^\nu G)(x)|\lesssim \delta^{-|\nu|}, \\
        &|(D^\nu F)(x)|\lesssim 1.
    \end{align*}
    For $1\leq i,j\leq 3$, let
    \begin{equation*}
        a_{i,j}^{(k)}(x)=\frac{\partial^{k+2}F}{\partial x_1 \partial x_i \partial x_j \partial x_{3}^{k-1}}(x).
    \end{equation*}
    We further assume that for all $x\in \Omega$ and $k\in \{q,q+1\}$ we have
    \begin{equation*}
        |\det(a_{i,j}^{(k)}(x))_{1\leq i,j\leq 3}|\gtrsim 1,
    \end{equation*}
    and
    \begin{equation*}
        \begin{cases}
            |a^{(q)}_{i,i}(x)|\sim 1&\quad 1\leq i \leq 2,\\
            |a^{(q)}_{2,1}(x)|\lesssim 1,&\\
            |a^{(q)}_{3,1}(x)|\sim 1,&\\
            |a^{(q)}_{3,j}(x)|\lesssim \delta&\quad j=2,3.\\
        \end{cases}
    \end{equation*}

    If $\delta$ is smaller than an absolute constant and $T$ is restricted to 
    \begin{equation*}
        M^{q+\frac{2}{Q}-\frac{1}{3}-\frac{4}{9}}\lesssim T \lesssim M^{q+\frac{2}{Q}+2},
    \end{equation*}
    then on setting $Q=2^q$, we have
    \begin{equation*}
        |S(T,\delta M;G,F)|\lesssim_\delta T^{\frac{9}{18(Q-1)+6Q+4Q}}M^{3-\frac{9(q+2)+3}{18(Q-1)+6Q+4Q}}.
    \end{equation*}
\end{lem}
\noindent
The $\delta$ in the above lemma does not depend on $\la$. With this, we may complete the proof of \cref{expSumThm}.

\begin{lem}\label{importedSumsLemma}
The bound \eqref{importedBoundExpSums} holds when $M\gtrsim \la^{\frac{2}{7}}$.
\end{lem}

\begin{proof}
   Recall that we are trying to estimate
    \begin{equation*}
        S_{\lambda,\delta}^M(x) = \la\delta\sum_{n\in\mathbb{Z}^3}\psi(\frac{n-x}{M})a_{1}(\la(n-x))\frac{e^{i\la \Theta(n-x)}}{|n-x|}.
    \end{equation*}
    To apply \cref{GuoBoundLemma}, we will need to manipulate this expression. We write
    \begin{equation*}
        S_{\lambda,\delta}^M(x) = \la\delta M^{-1}\sum_{n\in\Z{3}}G(n/M)e^{(\la M)F(n/M)},
    \end{equation*}
    where we have
    \begin{equation*}
        G(n/M)\coloneq \psi(\frac{n-x}{M})a_1(\la(n-x))\frac{M}{|n-x|},
    \end{equation*}
    and
    \begin{equation*}
        F(n/M)\coloneq \Theta(\frac{n-x}{M}),
    \end{equation*}
    using the homogeneity of $\Theta$. Note that $|n-x|$ is localized to values $\sim M$ and $G$ is uniformly bounded. The parameter $M$ plays the same role as in the statement of \cref{GuoBoundLemma}, while $T = \la M$ in our notation. The estimate in \cref{GuoBoundLemma} applies for sums over $c M$ for some small absolute constant $c$, but we can apply this result at $M$ at the cost of an absolute constant in our final bounds. This gives, after applying \cref{GuoBoundLemma} with $q=2$,
    \begin{equation*}
        |S_{\lambda,\delta}^M(x)|\lesssim \la\delta M^{-1}\Bigl((\la M)^{\frac{9}{94}}M^{3-\frac{39}{94}}\Bigr) = \delta \la^{\frac{103}{94}}M^{2-\frac{30}{94}}.
    \end{equation*}
    The dominant term is when $M\sim \delta^{-1}$ which gives
    \begin{equation*}
        |S_{\lambda,\delta}^{\delta^{-1}}(x)|\lesssim\la^{\frac{103}{94}} \delta^{-\frac{64}{94}},
    \end{equation*}
    and this is the desired result. Note that \cref{GuoBoundLemma} only applies when
    \begin{equation*}
        M^{2+\frac{1}{2}-\frac{1}{3}-\frac{4}{9}} = M^{\frac{31}{18}}\lesssim (\la M) \lesssim M^{2+\frac{1}{2}+2} = M^{\frac{9}{2}}.
    \end{equation*}
    As $M\lesssim \delta^{-1}\lesssim \la$ the lower bound is trivially satisfied. We therefore must also have either $\la \lesssim M^{\frac{7}{2}}$, or
    \begin{equation*}
        M\gtrsim \la^{\frac{2}{7}}.
    \end{equation*}
    The conditions on $F$ and $G$ are easily verified, as $x/M$ is small, while the determinant conditions are verified in this context in Section 3 of \cite{Guo}, and so we are done.
\end{proof}
    
\section{The approach by the circle method}\label{sec:circle}
In this section, we apply the circle method, which was first employed by Bourgain \cite{Bourgain1} to bound the $L^p$ norms of eigenfunctions of the Laplacian on the square torus. This method gives nontrivial results, which will nevertheless be superseded in our final theorem (Theorem \ref{mainresult}) by a combination of cap counting, $\ell^2$-decoupling, exponential sum estimates, and height splitting. It is interesting to measure the power of this method against the other techniques we have employed. In dimensions $\geq 5$, this approach can be used to give sharp eigenfunction estimates, see \cite{PezziSharpEFncn}. However, it is difficult to reach the $\delta\sim \la^{-1}$ case in dimension $3$ for large $p$ with any approach.

\begin{thm}
Conjecture~\ref{conjA} is satisfied on the square torus if 
$$
\delta > \lambda^{-\frac 12 + \frac 3p + \kappa} \qquad \mbox{with $\kappa>0$.}
$$
\end{thm}
\begin{proof} 
\noindent \textit{Decomposition of the projection operator.} Just like in the previous section, it is convenient to mollify the projection operator. We do it in a different way by defining $P^\flat_{\lambda,\delta}$ in terms of the Schr\"odinger group as
$$
P_{\lambda,\delta}^\flat = \chi \left( \frac{ \lambda^2-|D|^2}{\lambda \delta} \right) = \lambda \delta \int_{-\infty}^\infty \widehat{\chi}(\lambda \delta t) e^{2\pi i(\lambda^2-|D|^2) t} \dd t.
$$
Here, $\chi$ is a smooth cutoff functions which is equal to $1$ in a neighborhood of $0$. As a consequence, it is equivalent to prove bounds for $P_{\lambda,\delta}$ or $P_{\lambda,\delta}^\flat$. We now follow the ideas in Bourgain~\cite{Bourgain1} and Bourgain-Demeter~\cite{BourgainDemeter1}. The kernel of $P_{\lambda,\delta}^\flat$ is
\begin{align*}
& K_{\lambda,\delta}^\flat(x) = \lambda \delta \int \widehat{\chi}(\lambda \delta t) \prod_{j=1}^3 G(t,x_j) e^{2\pi i\lambda^2 t}\dd t \\
& \mbox{with} \qquad 
 G(t,x) = \sum_{k \in \mathbb{Z}} \gamma \left( \frac{k}{\lambda} \right) e^{2\pi i (kx - |k|^2 t)},
\end{align*}
where $\gamma$ is a function in $\mathcal{C}_0^\infty$ equal to $1$ on the unit ball.

If $a$ and $q$ are relatively prime integers $(a,q) =1$ such that $|q| \lesssim \lambda$ and $|t-\frac{a}{q}| \lesssim \frac{1}{\lambda q}$, then the function $G$ enjoys the bound
$$
|G(t,x)| \lesssim \lambda^\epsilon q^{-\frac 12} \min \left( \lambda, \left|t-\frac{a}{q}\right|^{-\frac 12} \right).
$$
(see the proof in \cite{BourgainDemeter1}). As a consequence, we get for $|t| \lesssim \frac{1}{\lambda}$ that $|G(t,x)| \lesssim \lambda^\epsilon \min\left( \lambda , |t|^{-\frac 12} \right)$. The $\lambda^\epsilon$ can be removed for $t$ small:
after stationary phase and Poisson summation,
$$
|G(t,x)| \lesssim \min\left( \lambda , |t|^{-\frac 12} \right) \qquad \mbox{if $|t| \lesssim \frac 1 \lambda$}.
$$

We now let $\eta_{Q,s}$ be cutoff functions localizing to values of $t$ such that $|t - \frac{a}{q}| \sim \frac{1}{\lambda 2^s}$ where $Q \in 2^\mathbb{N}$, $q \sim Q \ll \lambda$, $a \neq 0$, $(a,q) = 1$ and $Q < 2^s < \lambda$; it is understood that for $2^s \sim \lambda$, the localization is to the region $|t - \frac{a}{q}| \lesssim \frac{1}{\lambda^2}$. The cutoff function $\eta_{0}$ localizes to the region $|t| \lesssim \frac{1}{\lambda}$, where $1 < 2^s < \lambda$.

The $\eta_{Q,s}$ are chosen such that
$$
\eta_Q = \sum_{Q < 2^s < \lambda} \eta_{Q,s}(t) = 1 \qquad \mbox{if $|t-\frac{a}{q}| \lesssim \frac{1}{\lambda Q}$ with $q \sim Q$, $(a,q) = 1$}.
$$

Since $Q \ll \lambda$, the supports of the $\eta_{Q}$ and $\eta_{0}$ are disjoint, and we let $\rho$ be such that
$$
\eta_0 + \sum_{Q \ll \lambda} \eta_{Q} + \rho = 1.
$$
We can split accordingly the kernel $K_{\lambda,\delta}^\flat$
$$
K_{\lambda,\delta}^\flat = \int \left[\sum_{Q,s} \eta_{Q,s} + \eta_{0} + \rho \right] \dots \dd t = \sum_{Q,s} K_{\lambda,\delta}^{Q,s} + K_{\lambda,\delta}^{0} + K_{\lambda,\delta}^\rho
$$
and the operators
$$
P_{\lambda,\delta}^\flat = \sum P_{\lambda,\delta}^{Q,S} + P_{\lambda,\delta}^{0} + P_{\lambda,\delta}^\rho.
$$
We will now bound successively all the summands above.
\vspace{.1in}\newline
\noindent \textit{The arc around zero.} Choosing $\eta_0(t) = \eta(\lambda t)$, for a cutoff function $\eta$, the operator $P^0_{\lambda, \delta}$ is given by the formula
$$
P^0_{\lambda,\delta} = \lambda \delta \int \eta (\lambda t) e^{2\pi i( \lambda^2 - |D|^2) t} \, dt = \delta \widehat{\eta} \left( \frac{ \lambda^2 - |D|^2}{\lambda} \right).
$$
By Sogge's theorem and rapid decay of $\widehat{\eta}$, it follows that
\begin{equation}
\label{arcaroundzero}
\| P_{\lambda,\delta}^0 \|_{L^{p'} \to L^p}  \lesssim \lambda^{2 - \frac 6p} \delta \qquad \mbox{if $p>4$}.
\end{equation}

\vspace{.1in}

\noindent \textit{The major arcs.} On the support of $\eta_{Q,s}$ there holds $|G| \lesssim \lambda^\epsilon Q^{-\frac{1}{2}} (\lambda 2^s)^{\frac 12}$, so that (for a rapidly decaying function $\chi_0$ and $\kappa>0$)
\begin{align*}
 \| K_{\lambda,\delta}^{Q,s} \|_{L^\infty} & \lesssim \lambda^{1 + \epsilon} \delta \sum_{\substack{q \sim Q \\ (a,q) =1 \\ a\neq 0}} \chi_0 \left(\lambda \delta \frac a q \right) (\lambda 2^s)^{-1} (Q^{-\frac{1}{2}} \lambda^{\frac 12} 2^{\frac s2})^{3} \\
& \lesssim 
\left\{ \begin{array}{ll}
\lambda^{\frac{1}{2}+\epsilon} Q^{\frac 12} 2^{\frac s2} & \mbox{if $Q > \lambda^{1 - \kappa} \delta$} \\
\lambda^{-N} & \mbox{if $Q < \lambda^{1 - \kappa} \delta$} \\
\end{array} \right.\\
\| \widehat{K_{\lambda,\delta}^{Q,s}} \|_{L^\infty} & \lesssim \lambda^{1 + \epsilon} \delta \sum_{\substack{q \sim Q \\ (a,q) =1 \\ a\neq 0}} \chi_0 \left(\lambda \delta \frac a q \right) (\lambda 2^s)^{-1} \lesssim 
\left\{ \begin{array}{ll}
\lambda^{\epsilon-1} Q^2 2^{-s} & \mbox{if $Q > \lambda^{1 - \kappa} \delta$} \\
\lambda^{-N} & \mbox{if $Q < \lambda^{1 - \kappa} \delta$} \\
\end{array} \right.
\end{align*}
This implies that, if $p>2$,
$$
\| P^{Q,s}_{\lambda,\delta} \|_{L^{p'} \to L^p} \lesssim \left\{ \begin{array}{ll}
\lambda^{\frac{1}{2} - \frac 3{p}} Q^{\frac 12 + \frac 3{p}} 2^{s(\frac 12 - \frac 3{p})} & \mbox{if $Q > \lambda^{1 - \kappa} \delta$} \\
\lambda^{-N} & \mbox{if $Q < \lambda^{1 - \kappa} \delta$} \\
\end{array} \right.
$$

Summing over $Q$ and $s$ gives for $p>6$ and $\delta > \lambda^{\kappa -1}$
\begin{equation}
\label{boundQ}
\begin{split}
\sum_{1 < Q < 2^s < \lambda } \| P^{Q,s}_{\lambda,\delta} \|_{L^{p'} \to L^p} & \lesssim \la^\epsilon\sum_{ \substack{1 < Q < 2^s < \lambda \\ Q > \lambda^{1-\kappa} \delta }} \lambda^{\frac{1}{2} - \frac 3{p}} Q^{\frac 12 + \frac 3{p}} 2^{s(\frac 12 - \frac 3{p})} + O(\lambda^{-N}) \\
& \sim\la^\epsilon \lambda^{\frac{1}{2} - \frac 3{p}} \cdot \lambda^{\frac{1}{2} + \frac 3{p}} \cdot \lambda^{\frac{1}{2} - \frac 3 {p}} 
\sim \lambda^{\frac{3}{2}-\frac{3}{p}+\epsilon} .
\end{split}
\end{equation}

\medskip

\noindent \textit{The minor arcs.} By the Dirichlet approximation theorem, on the support of $\rho$, there exists $q \sim \lambda$ such that $|t-\frac a q| \lesssim \frac 1 {\lambda^2}$, which gives the bounds $|G| \lesssim \lambda^{\frac 12 + \epsilon}$ and therefore
\begin{align*}
& \| K^{\rho}_{\lambda,\delta} \|_{L^\infty} \lesssim \lambda^{\frac 32 + \epsilon} \\
& \| \widehat{K^\rho_{\lambda,\delta}} \|_{L^\infty} \lesssim 1
\end{align*}
so that, for $p>2$,
\begin{equation}
\label{boundrho}
\| P^\rho_{\lambda,\delta} \|_{L^{p'} \to L^p} \lesssim \lambda^{\frac 32 - \frac 3p + \epsilon}.
\end{equation}
Putting all the pieces together, we find that
$$
\| P_{\lambda,\delta}^\flat \|_{L^{p'} \to L^p} \lesssim \lambda^{2- \frac 6p} \delta + \lambda^{\frac 32 - \frac 3p + \epsilon}, 
$$
which gives the conjecture for $\delta > \lambda^{-\frac 12 + \frac 3p + \kappa}$.
\end{proof}

\section{Removing the $\eps$: handling $P_{\operatorname{few}}$}
\label{sec:epsilonremoval1}

In Section \ref{sec:dec}, we split the projection operator in the following way.
\begin{equation*}
    P_{\la,\delta} = P_{\operatorname{few}} +P_{\operatorname{many}}.
\end{equation*}

Different estimates were used for $P_{\operatorname{few}}$ and $P_{\operatorname{many}}$, leading to the proof of Conjecture \ref{conjA} with $\epsilon$-loss in some range of $\lambda,\delta,p$. Recall that the term $P_{\operatorname{few}}$ is expected to be dominant in the point-focusing range, while the term $P_{\operatorname{many}}$ should dominate in the geodesic focusing regime.

We also used exponential sum estimates in \cref{sec:exponential} to prove sharp estimates for the full projector. In this section, we use a height splitting argument and exponential sum estimates to prove the following improved bound for $P_{\operatorname{few}}$.

\begin{thm}\label{PfewThm}
\SMnote{Can we check if \textbf{in the proof} it is used that \eqref{eq:expected_number_of_points} holds? It looks like that condition was added to the proposition by accident, so I removed it. But maybe it should be there for some reason.}
Assume that 
$p\geq 4$ and
$$
\delta> \min\bigPara{\la^{-\frac{85p-152}{158p-256}+\kappa}, \la^{-\frac{9p-224}{444-158p}+\kappa}},
$$ 
with $\kappa>0$, we have
\begin{equation}\label{PfewThmEqn}
\lpnorm{P_{\operatorname{few}} f}{p}{\T{3}}\lesssim_\kappa \Bigl( \la^{1-\frac{3}{p}}\delta^{1/2}+(\la\delta)^{\frac{1}{2}-\frac{1}{p}}\Bigr )\lpnorm{f}{2}{\T{3}}. 
\end{equation}
\end{thm}
Note that 
$$
\la^{-\frac{85p-152}{158p-256}} < \lambda^{-\frac 12} \qquad \mbox{for }p>4.
$$
In Section \ref{dec-height}, we will prove a weaker result than the previously stated theorem where $\delta$ is restricted to be $> \la^{-\frac{1}{2}+\kappa}$ for arbitrary $\kappa>0$. This uses simpler estimates and allows for cleaner exposition. In Section \ref{dec-exp}, we import exponential sum estimates from \cref{sec:exponential} to get the full result.

\subsection{Combining decoupling and height splitting}

\label{dec-height}

\begin{DPlongnote}{On why (1.3) is not used.}
    The proof does not use the lattice point bound (1.3) explicitly. The $L^1\rightarrow L^\infty$ estimate used is (5.5)  which follows from trivial estimates on the Fourier transform of the surface measure on the ellipsoid. The height decomp using this estimate does nothing past $\la^{-\frac{1}{2}}$ as the cutoff we need, $(\la\delta^{-1})^{1/2}$, is larger than the conjectured $L^\infty$ estimate implied by (1.3).

    The large height argument is true whenever (5.5) holds unconditionally. However, the small height augment breaks whenever the window is small enough because the required $L^\infty$ bound is larger than would follow from $(1.3)$, which we expect to be true. If one has additional cancellation in the exponential sum and is not using the triangle inequality like is used in (5.5), then one can take a smaller height as the cutoff which is done in section 7.2.

    If (1.3) does not hold for some torus and some $\delta$, then the conjecture itself would have to updated anyways. 

    Recall this argument holds on more general manifolds than the torus. It is the case that the improvements on the exponential sums of Guo hold for the same $\delta$ where (1.3) holds which is why the condition was stated there. But (1.3) itself is not sufficient; we cannot prove sharp estimates for the square torus beyond $\la^{1/2}$ without non-trivial cancellation, for example.
\end{DPlongnote}

\begin{prop} \label{PfewProp}
\SMnote{Can we check if \textbf{in the proof} it is used that \eqref{eq:expected_number_of_points} holds? It looks like that condition was added to the proposition by accident, so I removed it. But maybe it should be there for some reason.}
Let $P_{\operatorname{few}}$ be as before. Then for $\delta> \la^{-\frac{1}{2}+\kappa}$, we have

    \begin{equation}\label{PfewPropEqn}
        \lpnorm{P_{\operatorname{few}} f}{p}{\T{3}}\lesssim_\kappa \la^{1-\frac{3}{p}}\delta^{1/2}\lpnorm{f
        }{2}{\T{3}}. 
    \end{equation}
\end{prop}

Recall that \cref{lem:PfewBound} verifies the sharp conjecture for $P_{\operatorname{few}}$ when $\la^{\frac p2-3+\kappa}<\delta < \la^{2-\frac{p}{2}-\kappa}.$ The proposition above handles $\lambda^{-\frac 12+\kappa}<\delta <1$. Since
    \begin{align*}
          [\la^{\frac p2-3+\kappa},\la^{2-\frac{p}{2}-\kappa}]
     \cup
     [\lambda^{-\frac 12 +\kappa}, 1]
     &=  [\la^{\frac p2-3+\kappa},1]
     \supseteq
       [\la^{\frac p2-3+3\kappa},1]
     &\text{if }p\leq 5-4\kappa,
     \\
          [\lambda^{-\frac 12 +\kappa}, 1]
&\supseteq
         [\la^{\frac p2-3+3\kappa},1]
           &\text{if }p\geq 5-4\kappa,
    \end{align*}
we can write $\kappa'=3\kappa$ to obtain

\begin{cor}\label{cor:Pfewbound}
Assume that \eqref{eq:expected_number_of_points} holds and that $\delta> \la^{\frac p2-3+\kappa}$ for any fixed $\kappa>0$. Then
   $P_{\operatorname{few}}$ satisfies the sharp version of \cref{conjA}.
\end{cor}

\begin{proof}[Proof of \cref{PfewProp}]

\noindent  \textit{Step 1: decomposition into caps and heights.}
Recall that, if $\lambda \delta^2 > 1$,
$$
P_{\operatorname{few}} = \sum_{2^s \lesssim \lambda \delta^2} P^s_{\lambda,\delta}.
$$

Let $P^{\sharp}_{\lambda,\delta}$ be the mollified version of our projection operator previously defined in the proof of \cref{expSumThm}. In order to prove that $P^{\sharp}_{\lambda,\delta} P_{\operatorname{few}}$ has an operator norm $\lesssim \lambda^{1-\frac 3p} \delta^{\frac 12}$ from $L^2$ to $L^p$, it suffices to show, by an application of orthogonality, that
$$
\| P^{\sharp}_{\lambda,\delta} f \|_{L^p} \lesssim \lambda^{1-\frac 3p} \delta^{\frac 12}
$$
for $f$ such that
$$
f = P_{\operatorname{few}} F \quad \mbox{and} \quad \| F \|_{L^2} = 1,
$$
which we assume henceforth.

We shall split the torus into large and small heights. Let
\begin{equation} \label{defA-A+}
\begin{split}
& A_- = \{x\in \T{3} : | P^{\sharp}_{\lambda,\delta} f(x)| < C_0 (\la\delta^{-1})^{\frac{1}{2}}\}, \\
& A_+= \{x\in \T{3} : | P^{\sharp}_{\lambda,\delta} f(x)| \geq C_0 (\la\delta^{-1})^{\frac{1}{2}}\}.
\end{split}
\end{equation}
The reason for the choice of $(\la\delta^{-1})^{\frac{1}{2}}$ as the defining height will become apparent later in the proof. It should be noted it is the square root of the trivial bound \cref{L1LinftyExpSumBound} when $M\sim \delta^{-1}$. Here, $C_0$ is a positive constant we shall select later.

\medskip

\noindent \textit{Step 2: small height.} On the one hand, the $L^\infty$ norm of $P^{\sharp}_{\lambda,\delta} f$ on $A_-$ is bounded by the very definition of $A_-$
\begin{equation}\label{P0smallheightsbound}
        \lpnorm{P^{\sharp}_{\lambda,\delta} f}{\infty}{A_-}\lesssim (\la\delta^{-1})^{\frac{1}{2}}.
    \end{equation}
On the other hand, its $L^4$ norm on $A_-$ can be estimated by \eqref{Pfew} to give
$$ 
\lpnorm{P^{\sharp}_{\lambda,\delta}  f}{4}{A_-} \leq \lpnorm{P^{\sharp}_{\lambda,\delta}f}{4}{\T{3}}
= \lpnorm{ P^{\sharp}_{\lambda,\delta} P_{\operatorname{few}} F}{4}{\T{3}} \lesssim_\epsilon \la^{\epsilon} \la^{\frac 14 }\delta^{\frac 12}.
$$

Interpolating between these two norms yields
\begin{equation}\label{P0holderbound}
\begin{split}
\lpnorm{P^{\sharp}_{\lambda,\delta} f}{p}{A_-}&  \leq \lpnorm{P^{\sharp}_{\lambda,\delta} f}{\infty}{A_-}^{\frac{p-4}{p}}
\| P^{\sharp}_{\lambda,\delta} f \|_{L^4(A_-)}^{\frac{4}{p}} \\
& \lesssim_\epsilon \lambda^{\frac 4p \epsilon} (\la\delta^{-1})^{\frac{p-4}{2p}}\la^{1/p}\delta^{2/p}\lesssim_\kappa \la^{1-\frac{3}{p}}\delta^{1/2},\end{split}
\end{equation}
where the last inequality is a consequence of $\delta>\lambda^{-\frac{1}{2}+\kappa}$.

\medskip

\noindent \textit{Step 3: large height.} By duality and self-adjointness of $ P^{\sharp}_{\lambda,\delta}$, there exists  $g\in L^{p'}(A_+)$ with $L^{p'}$ norm $1$ such that
$$
\| P^{\sharp}_{\lambda,\delta} f \|_{L^p(A_+)} = \int \mathbf{1}_{A_+} \cdot P^{\sharp}_{\lambda,\delta} f \cdot g \dd x \leq \| P^{\sharp}_{\lambda,\delta} \mathbf{1}_{A_+} g \|_{L^2}.
$$
Therefore,
\begin{align*}
\| P^{\sharp}_{\lambda,\delta} f \|_{L^p(A_+)}^2 & \leq \int P^{\sharp}_{\lambda,\delta} \mathbf{1}_{A_+} g \cdot \overline{P^{\sharp}_{\lambda,\delta} \mathbf{1}_{A_+} g} \dd x 
= \int (P^{\sharp}_{\lambda,\delta})^2 \mathbf{1}_{A_+} g \cdot  \overline{\mathbf{1}_{A_+} g} \dd x,
\end{align*}
where we again used that the projection operator is self-adjoint. At this point, we split $(P^{\sharp}_{\lambda,\delta})^2$ into a \say{local} part $L_{\lambda}$ and a \say{global} part $G_{\lambda,\delta}$:
$$
(P^{\sharp}_{\lambda,\delta})^2 = L_\lambda + G_{\lambda,\delta}, \qquad L_\lambda = \delta (P^{\sharp}_{\lambda,1})^2, \qquad G_{\lambda,\delta} = (P^{\sharp}_{\lambda,\delta})^2 - \delta (P^{\sharp}_{\lambda,1})^2.
$$
From the above, we see that we need to estimate
$$
\int L_\la (\mathbf{1}_{A_+}\cdot g(x))\overline{(g\cdot \mathbf{1}_{A_+})(x)} dx + \int G_\la (\mathbf{1}_{A_+}\cdot g(x))\overline{(g\cdot \mathbf{1}_{A_+})(x)} dx = I + II.
$$

The local piece is easily estimated. Indeed, H\"older's inequality and Sogge's universal estimates lead to the bound 
\begin{align*}
|I| \leq \| L_\lambda (g\cdot \mathbf{1}_{A_+}) \|_{L^p} \| g\cdot \mathbf{1}_{A_+}\|_{L^{p'}} \leq \| L_{\lambda} \|_{L^{p'} \to L^{p}} \| g\cdot \mathbf{1}_{A_+}\|_{L^{p'}}^2 \lesssim \delta\lambda^{2-\frac{6}{p}}.
\end{align*}

In order to bound $II$, we need to estimate the $L^\infty$ norm of the convolution kernel of $G_{\lambda,\delta}$. By \eqref{FT} and Poisson summation, it equals
$$
G_{\lambda,\delta}(x) = \sum_{n \in \mathbb{Z}^3} \lambda^2 \delta \left[ \widehat{\chi}(\delta(x+n))) - \widehat{\chi}(x+n) \right] m(\lambda(x+n)).
$$
Observe that
$$
\left| \lambda^2 \delta \left[ \widehat{\chi}(\delta x)) - \widehat{\chi}(x) \right] m(\lambda x) \right| \lesssim \lambda \delta \langle x \rangle^{-1}
$$
(which can be seen by distinguishing the cases $|x|< \frac 1 \lambda$, $\frac{1}{\lambda} < |x| < 1$ and $|x| >1$) and that it decays rapidly for $|x| \gg \delta^{-1}$.
It follows that

\begin{equation}\label{section5LatticeCountBound}
    \lpn{G_{\la,\delta}}{1\rightarrow \infty}\leq C_1 \la\delta^{-1}.
\end{equation}

For some constant $C_1$. This allows us to bound $II$ using H\"older's inequality, the definition of $g$, and the definition of $A_+$. 
\begin{align*}
    |II|&\leq C_1 \la\delta^{-1}\lpn{\mathbf{1}_{A_{+}}\cdot g}{1}^2 \leq C_1(\la\delta^{-1}) (\lpn{\mathbf{1}_{A_+}}{p}\lpn{g}{p'})^2\\
    &\leq C_1C_0^{-2} (\la\delta^{-1})\Bigl((\la\delta^{-1})^{-1/2}\lpnorm{P^\sharp_{\lambda,\delta} f}{p}{A_{+}}\Bigr)^2\leq  \frac{1}{4}\lpnorm{ P^\sharp_{\lambda,\delta} f}{p}{A_{+}}^2;
\end{align*}
here, we chose $C_0$ so that $C_1 C_0^{-2}\leq \frac{1}{4}$.

Combining with our bound for $I$ gives 

\begin{equation*}
    \lpnorm{P^\sharp_{\lambda,\delta} f}{p}{A_+}^2\leq C\delta \la^{2-\frac{6}{p}} + \frac{1}{4}\lpnorm{P^\sharp_{\lambda,\delta} f}{p}{A_+}^2,
\end{equation*}
which finally yields the desired estimate. Note that $\lpnorm{P_{\la,\delta}f}{\infty}{\T{3}}\lesssim \la\delta^{1/2}$, so that the preceding argument is only non-trivial when $\delta> \la^{-\frac{1}{2}+\kappa}$. 
\end{proof}

\subsection{Improvement using exponential sum estimates}
\label{dec-exp}

Our idea now is to get an improvement of Proposition \ref{PfewProp} with the help of exponential sum estimates. These estimates will allow us to use a sharper threshold for height splitting.

\begin{proof}[Proof of Theorem \ref{PfewThm}] We shall follow the proof of Proposition \ref{PfewProp}, only modifying the threshold $Z$ defining $A_-$ and $A_+$. Let
\begin{equation*}
\begin{split}
& A_- = \{x\in \T{3} : | P^{\sharp}_{\lambda,\delta} f(x)| < C_0 Z^{\frac{1}{2}}\} \\
& A_+= \{x\in \T{3} : | P^{\sharp}_{\lambda,\delta} f(x)| \geq C_0 Z^{\frac{1}{2}}\}.
\end{split}
\end{equation*}

\medskip

\noindent \textit{Finding $Z$.} Defining $L_{\lambda,\delta}$ and $G_{\lambda,\delta}$ as above, $Z$ will be chosen to match the bound on $\|G_{\lambda,\delta}\|_\infty$ which we now derive. Recall that the kernel of $G_{\lambda,\delta}$ is given by
$$
G_{\lambda,\delta}(x) = \sum_{n \in \mathbb{Z}^3} \lambda^2 \delta \left[ \widehat{\chi}(\delta(x+n))) - \widehat{\chi}(x+n) \right] m(\lambda(x+n)).
$$
We now make use of the dyadic partition of unity \eqref{dyadicpartitionofunity}. For $|x+n| \lesssim 1$, we observe as above that $|\lambda^2 \delta \left[ \widehat{\chi}(\delta(x+n))) - \widehat{\chi}(x+n) \right] m(\lambda(x+n))| \lesssim \lambda \delta$. For $|x+n| \gtrsim 1$, we replace $m$ by the leading order term in its asymptotic expansion \eqref{expansionm} (since higher order terms are easier to treat). This gives
\begin{align*}
G_{\lambda,\delta}(x) & = \la\delta \sum_{1\lesssim M \lesssim \delta^{-1}}\sum_{n \in \mathbb{Z}^3} \left[ \psi(\frac{n+x}{M})a_1(\la(n+x))\frac{e^{i\Theta(|n-x|)}}{|n+x|} + O(\lambda^{-1} M^{-2}) \right] \\
& \qquad  + O(\lambda \delta).
\end{align*}
The bound used in the previous section corresponds to \eqref{L1LinftyExpSumBound}. Here, we will use \eqref{importedBoundExpSums}. Doing so yields
\begin{equation}\label{eq:heightImprovementLemma}
        \lpn{G_{\la,\delta}}{L^\infty}\lesssim \la^{\frac{103}{94}}\delta^{-\frac{64}{94}}
\end{equation}
and leads to the choice
$$
Z = \la^{\frac{103}{94}} \delta^{-\frac{64}{94}}.
$$

\medskip
\noindent \textit{Large height} The argument is identical to the proof of Proposition \ref{PfewProp} but with our new bound for $\lpn{G_{\la,\delta}}{L^\infty}$.

\medskip
\noindent
\textit{Small height} We will restrict to the case $\delta\lesssim \la^{-1/2}$ as larger windows have already been handled. In this regime, the best $L^2\rightarrow L^4$ norm we can prove for $P_{\operatorname{few}}$ is $O(\lambda^\epsilon)$. Interpolating with $L^\infty$ gives
\begin{equation}\label{P0holderbound}
\begin{split}
\lpnorm{P^{\sharp}_{\lambda,\delta} f}{p}{A_-}&  \leq \lpnorm{P^{\sharp}_{\lambda,\delta} f}{\infty}{A_-}^{\frac{p-4}{p}}
\| P^{\sharp}_{\lambda,\delta} P_{\operatorname{few}} f \|_{L^4(A_-)}^{\frac{4}{p}} \\
& \lesssim_\epsilon \lambda^{\frac 4p \epsilon} Z^{\frac{p-4}{2p}} \lesssim \la^\epsilon(\la^{\frac{103}{94}}\delta^{-\frac{64}{94}})^{\frac{p-4}{2p}} \end{split},
\end{equation}
which agrees with the conjecture when the right hand side is $\lesssim \la^{1-\frac{3}{p}}\delta^{1/2}$ or $\lesssim(\la\delta)^{\frac{1}{2}-\frac{1}{p}}$. Comparing against the first and second terms gives 
\begin{equation*}
\delta> \la^{-\frac{85p-152}{158p-256}+\kappa}
\qquad \mbox{and} \qquad
 \delta> \la^{-\frac{9p-224}{444-158p}+\kappa}
\end{equation*}
respectively (note the $\kappa$ exponent is absorbing the $\epsilon$-loss from the application of decoupling).
\end{proof}

\section{Removing the $\eps$: handling $P_{\operatorname{many}}$}
\label{sec:epsilonremoval2}

We have previously decomposed $P_{\lambda,\delta}$ into $P_{\operatorname{few}}$ and $P_{\operatorname{many}}$. In this section, we focus on $P_{\operatorname{many}}$ and prove the following.

\begin{thm}\label{PmanyThm}
Assume that \eqref{eq:expected_number_of_points} holds and let $\kappa>0$,  \begin{equation*}
        \delta >\max\bigPara{\la^{-\frac{316-27p}{104(p-2)}+\kappa}, \la^{-\frac{5948-1071p}{3852-547p}+\kappa}}
    \end{equation*}
    when $p\leq 5$, and $\delta>\la^{-\frac{1}{2}+\kappa}$ when $p>5$. Then,

    \begin{equation}\label{PmanyEqn}
    \| P_{\operatorname{many}}f \|_{L^p} \lesssim _\kappa (\la^{1-\frac{3}{p}}\delta^{1/2} + (\la\delta)^{\frac{1}{2}-\frac{1}{p}})\| f \|_{L^2}.
    \end{equation}
\end{thm}

The proof will occupy the rest of this section and will be built up to in a number of lemmas and propositions. To begin, recall $P^s_{\la,\delta}$ is the projection onto the number of caps with $\sim 2^s$ lattice points where $2^s \gtrsim \la\delta^2+1$. We are going to decompose $P_{\operatorname{many}}$ further into
\begin{equation*}
P_{\operatorname{many}} = \underbrace{\sum_{\lambda^2 \delta + 1 \ll 2^s \leq (\lambda \delta)^{1-\alpha}} P^s_{\la,\delta}}_{\displaystyle P_{\operatorname{med}}} + \underbrace{\sum_{(\lambda \delta)^{1-\alpha} \leq 2^s \lesssim \lambda \delta} P^s_{\la,\delta}}_{\displaystyle P_{\operatorname{max}}},
\end{equation*}
where $\alpha>0$ is a parameter which will be chosen such that $\alpha \ll 1$. The operator $P_{\operatorname{med}}$ gathers caps containing a number of points much larger than average but much less than maximal, while the operator $P_{\operatorname{max}}$ corresponds to caps containing almost the maximal number of points.

The proof of Theorem \ref{PmanyThm} will proceed by bounding separately $P_{\textup{med}}$ and $P_{\textup{max}}$. To bound $P_{\textup{med}}$  satisfactorily requires \eqref{eq:expected_number_of_points}; the bound for  $P_{\textup{max}}$ is independent of that hypothesis.

\Cref{lem:PmanyCaps} gives that $P_{\textup{med}}$ satisfies estimates with a power saving compared to the conjecture away from the boundary in $(p,\delta)$ defined by \cref{thm:caps}. Hence, for any $\kappa>0$, provided that
\begin{equation*}
        \delta >\max\bigPara{\la^{-\frac{316-27p}{104(p-2)}+\kappa}, \la^{-\frac{5948-1071p}{3852-547p}+\kappa}},
    \end{equation*}
    when $p\leq 5$, or that $\delta>\la^{-\frac{1}{2}+\kappa}$ when $p>5$, we have
    \begin{equation*}
\| P_{\textup{med}}f \|_{L^p} \lesssim_\kappa (\la^{1-\frac{3}{p}}\delta^{1/2} + (\la\delta)^{\frac{1}{2}-\frac{1}{p}}) \| f \|_{L^2}.
    \end{equation*}
The proof of Theorem \ref{PmanyThm} now reduces to the proof of the boundedness of $P_{\operatorname{max}}$. Below we will prove the following result: 

\begin{prop} \label{propPS}
If $p \geq 4$ and $\delta \gtrsim \lambda^{-1+\kappa}$, with $\kappa>0$,
\begin{equation*}
\| P_{\operatorname{max}} f \|_{L^p} \lesssim (\la^{1-\frac{3}{p}}\delta^{1/2}+(\la\delta)^{\frac{1}{2}-\frac{1}{p}})\| f \|_{L^2}.
\end{equation*}
\end{prop}
\Cref{PmanyThm} will follow at once upon verifying the proposition.

The proof of \cref{propPS} involves two main ideas. First, we prove that larger caps than have been considered up to this point enjoy the desired estimates. Second, we employ a bilinear organization, where diagonal terms are handled using the aforementioned sharp estimate for large caps and off-diagonal terms are handled by transversality. The only bound we need for caps is the ``easy" first part of \cref{prop:rank2bound}, which suffices essentially because that bound is optimal for large $s$.

We define a new partition of unity of the ellipsoid $\mathcal{S}_\lambda$ into bigger caps, of size $\sim \lambda \delta$. To be more precise, we let $\{\nu\}$ be a maximal $\delta$ separated set on the ellipsoid $\mathcal{S}_1 = \{ Q(x) = 1 \}$. Let $\{ \beta^\nu \}$ be a partition of unity on this ellipsoid such that $\beta^\nu$ is compactly supported in a $4\delta$ neighborhood of $\nu$. We extend $\beta^\nu$ to a zero-homogeneous function on $\mathbb{R}^3$ away from the origin. Finally, we define the Fourier multiplier $P^\nu_{\lambda,\delta}$ by
\begin{equation}
P^\nu_{\lambda,\delta} = \chi_{\lambda,\delta}^\nu (D), \qquad  \chi_{\lambda,\delta}^\nu  = \beta^\nu \chi_{\lambda,\delta}
\end{equation}
(recall that the symbol $\chi_{\lambda,\delta}$ associated to the mollified projector $P^\sharp_{\la,\delta}$ was introduced in Section \ref{sec:exponential}).

We can now decompose $P_{\operatorname{max}} f$ into functions supported on caps with diameter $\sim \lambda \delta$ 
\begin{equation*}
        P_{\operatorname{max}}f = \sum_{\nu}P^\nu_{\lambda,\delta}P_{\operatorname{max}}f = \sum_{\nu}P^\nu_{\operatorname{max}}f
    \end{equation*}
and then make the problem bilinear by writing    
\begin{equation*}
(P_{\operatorname{max}}f )^2 = \sum_{\nu,\nu'} P^\nu_{\operatorname{max}}fP^{\nu'}_{\operatorname{max}}f = 
\sum_{|\nu - \nu'| \lesssim  \delta} P^\nu_{\operatorname{max}}fP^{\nu'}_{\operatorname{max}}f + \sum_{|\nu - \nu'| \gg  \delta} P^\nu_{\operatorname{max}}fP^{\nu'}_{\operatorname{max}}f.
\end{equation*}
We will treat separately the \say{close} interactions, for which $|\nu - \nu'| \lesssim \delta$ and the \say{far} interactions for which $|\nu - \nu'| \gtrsim \delta$. Note that for a fixed $\nu$, there are only a constant number of $\nu'$ such that $|\nu-\nu'|\lesssim \delta$.

The optimal bound for large caps is given by the following proposition.

\begin{prop}\label{KNEstThrm}
    Let $\delta>\la^{-1+\kappa}$. The operator $P^\nu_{\lambda,\delta}$ obeys the following bound
    \begin{equation}\label{KNEstEqn}
\| P^\nu_{\lambda,\delta}f \|_{L^p} \lesssim_\kappa (\la^{1-\frac{3}{p}}\delta^{1/2}+(\la\delta)^{\frac{1}{2}-\frac{1}{p}})\| f \|_{L^2}.
    \end{equation}
\end{prop}

We also need an estimate involving bilinear interactions between transverse caps which is contained in the following lemma. 

\begin{lem}[Transverse interactions] \label{fourmi} Let $P^\nu_{\operatorname{max}} = P^\nu_{\lambda,\delta} P_{\operatorname{max}}$. If $p \geq 4$ and $|\nu - \nu'| \gtrsim \delta$,
$$
\| (P^\nu_{\operatorname{max}}f P^{\nu'}_{\operatorname{max}}f)^{1/2}\|_{L^p}
\lesssim (\lambda \delta)^{\frac 12 - \frac 3{2p} + 3\alpha}  \| P^{\nu}_{\operatorname{max}} f \|_{L^2}^{1/2} \| P^{\nu'}_{\operatorname{max}} f\|_{L^2}^{1/2}.
$$
\end{lem}

We postpone the proofs of these two results until after the proof of \cref{propPS}. With these results in hand, we now bound $P_{\operatorname{max}}$.

\begin{proof}[Proof of \cref{propPS}] We treat close and far interactions separately. 

\medskip\noindent\textit{Close interactions} For close interactions we use \cref{KNEstThrm}, which states that the $P^\nu_{\lambda,\delta}$ enjoys the same bound as conjectured in \cref{conjA} for $P_{\lambda,\delta}$. Using this result in conjunction with the H\"older and Cauchy-Schwarz inequalities,
\begin{align*}
\left\| \left( \sum_{(|\nu - \nu'| \lesssim \delta}P^\nu_{\operatorname{max}} f P^{\nu'}_{\operatorname{max}}f \right)^{1/2} \right\|_{L^p} &\lesssim \left( \sum_{\nu}\| P^\nu_{\operatorname{max}} f \|_{L^p}^2 \right)^{1/2}
\\ &\lesssim (\la^{1-\frac{3}{p}}\delta^{\frac 12}+(\la\delta)^{\frac{1}{2}-\frac{1}{p}}) \left( \sum_{\nu}\| P^\nu_{\operatorname{max}} f \|_{L^2}^2 \right)^{1/2}\\
&\lesssim (\la^{1-\frac{3}{p}}\delta^{\frac 12}+(\la\delta)^{\frac{1}{2}-\frac{1}{p}})\| f \|_{L^2}.
\end{align*}

\medskip
\noindent
\textit{Far interactions.} To treat far interactions, the key result is Lemma \ref{fourmi}, which gives a gain in bilinear interactions between far caps, namely
$$
\| (P^\nu_{\operatorname{max}}f P^{\nu'}_{\operatorname{max}}f)^{1/2}\|_{L^p}
\lesssim (\lambda \delta)^{\frac 12 - \frac 3{2p} + 3\alpha}  \| P^{\nu}_{\operatorname{max}} f \|_{L^2}^{1/2} \| P^{\nu'}_{\operatorname{max}} f\|_{L^2}^{1/2}.
$$
Applying this inequality together with Minkowski's inequality,
\begin{align*}
& \left\| \left( \sum_{|\nu -\nu'| \gtrsim  \delta}  P^\nu_{\operatorname{max}} f P^{\nu'}_{\operatorname{max}}f \right)^{1/2} \right\|_{L^p} \leq \left( \sum_{|\nu -\nu'| \gtrsim  \delta} \left\| (P^\nu_{\operatorname{max}} f P^{\nu'}_{\operatorname{max}}f)^{1/2} \right\|_{L^p}^2 \right)^{1/2} \\
& \qquad \qquad\qquad \qquad \qquad \lesssim ( \delta)^{\frac 12 - \frac 3{2p} + 3\alpha} \left( \sum_{\nu,\nu'} \| P^\nu_{\operatorname{max}} f \|_{L^2} \| P^{\nu'}_{\operatorname{max}} f \|_{L^2} \right)^{1/2}.
\end{align*}
At this point, we can use that the number of $\nu$ such that $P^\nu_{\operatorname{max}} f \neq 0$ is less than the number of caps $\theta$ for which $\Lambda_\theta$ has dimension 2 and which contains a number of points $> (\lambda \delta)^{1-\alpha}$. By \cref{prop:rank2bound}, this number is
\begin{equation}
\label{boundNmax}
\mathcal{N}_{\operatorname{max}} = \sum_{2^s > (\lambda \delta)^{1-\alpha}} \mathcal{N}^2_s \lesssim \sum_{2^s > (\lambda \delta)^{1-\alpha}} (2^{-s} \lambda \delta)^3 = (\lambda \delta)^{3\alpha}.
\end{equation}
Using this bound and the Cauchy-Schwarz inequality and going back to the estimate on far interactions, we find that
\begin{align*}
\left\| \left( \sum_{|\nu -\nu'| \gtrsim  \delta}  P^\nu_{\operatorname{max}} f P^{\nu'}_{\operatorname{max}}f \right)^{1/2} \right\|_{L^p} & \lesssim (\lambda \delta)^{\frac 12 - \frac 3{2p} + 6\alpha} \left( \sum_\nu \| P^\nu_{\operatorname{max}} f \|_{L^2}^2 \right)^{1/2} \\
& \lesssim (\lambda \delta)^{\frac 12 - \frac 3{2p} + 6\alpha} \| f \|_{L^2}.
\end{align*}
This gives the desired result provided $\frac 12 - \frac 3{2p} + 6\alpha < \frac 1 2 - \frac 1p$, which can be achieved by choosing $\alpha$ sufficiently small.
\end{proof}

We now prove \cref{KNEstThrm}, which states that projectors on larger caps $P^\nu_{\lambda,\delta}f$ satisfy the same bound as expected by Conjecture \ref{conjA} for $P_{\lambda,\delta}$. This explains the choice of the diameter $\sim \lambda \delta$ of the cap associated to $\nu$: it is the biggest diameter which allows for a simple proof of the above proposition. We shall prove an oscillatory integral lemma first which contains the bulk of the proof. 
\begin{lem} 
\label{coccinelle}
Assume that $\delta > \lambda^{\kappa -1}$, with $\kappa>0$ and that the cap associated to $\nu$ has, at its center, a unit normal $n$ to the ellipsoid defined by $Q$. Then, if $|\xi| \lesssim \delta^{-1}$,
$$
\left| \widehat{\chi^\nu_{\lambda,\delta}}(\xi) \right| \lesssim_N 
\begin{cases}
\lambda \delta \langle \xi \rangle^{-1} & \mbox{if $|n \cdot \xi| \lesssim 1$} \\
\lambda \delta^2 \langle P_n^\perp \xi \rangle^{-N} & \mbox{if $|n \cdot \xi| \gg 1$, for $N>0$,} 
\end{cases}
$$
where we set $P_n^\perp \xi = \xi - (\xi \cdot n) n$.
\end{lem}

This lemma is ultimately an oscillatory integral bound, but some care is needed to organize the proof in the simplest form possible.

\begin{proof} \textit{Step 1: reduction to the Fourier transform of a measure on a cap} We will abuse notation by using $\nu$ to refer to the cap associated to the direction $\nu$. Up to rotating the ellipsoid, we can assume that $n$ is aligned with the $x_1$ axis; for greater simplicity in the notations, we furthermore assume that the center of $\nu$ has coordinates $x_2=x_3=0$. 

The function $\chi_{\lambda,\delta}(x)$ localizes smoothly to a $\delta$ neighborhood of the ellipsoid $\mathcal{S}_\lambda$, while the function $\beta^\nu(x)$ localizes smoothly to a cone with opening $\sim \delta$ around $x_2 = x_3 = 0$. Therefore, $\beta^\nu(x) \chi_{\lambda,\delta}(x)$ can be written as an average of the superficial measures $\dd\sigma_\rho$ (defined in Section \ref{sec:exponential}) against an appropriate function $\varphi_{\rho}$:
$$
\beta^\nu(x) \chi_{\lambda,\delta}(x) = \int \varphi_\rho (x_2,x_3) \dd\sigma_\rho(x) \dd \rho.
$$

Here, the function $\varphi_\rho (x_2,x_3)$ is localized on the set where $|(x_2,x_3)| \lesssim \rho \delta$ and satisfies
$$
|\nabla^k \varphi_\rho (x_2,x_3)| \lesssim (\rho \delta)^{-k} \chi_0 \left( \frac{\rho-\lambda}{\delta}\right),
$$
where $\chi_0$ is a rapidly decaying cutoff function. To prove the lemma, it suffices to bound the Fourier transform of $\varphi_\rho (x_2,x_3) \dd\sigma_\rho(x)$ and prove that
$$
\widehat{\varphi_\rho \dd \sigma_\rho}(\xi) \lesssim \chi_0 \left( \frac{\rho-\lambda}{\delta}\right) \begin{cases}
\lambda \langle \xi \rangle^{-1} & \mbox{if $|(\xi_2,\xi_3)| \lesssim 1$} \\
\lambda \delta \langle (\xi_2,\xi_3)\rangle^{-N} & \mbox{if $|(\xi_2,\xi_3)| \gg 1$} 
\end{cases}
$$

The cases $\rho \ll \lambda$ or $\rho \gg \lambda$ are easily treated because of the rapid decay of $\chi_0$, there remains the case $\rho \sim \lambda$. To simplify notations, we set $\rho = \lambda$ and $\varphi = \varphi_\lambda$, and we then aim at proving
$$
| \widehat{\varphi \dd \sigma_\lambda}(\xi)| \lesssim \begin{cases}
\lambda \langle \xi \rangle^{-1} & \mbox{if $|(\xi_2,\xi_3)| \lesssim 1$} \\
\lambda \delta \langle (\xi_2,\xi_3)\rangle^{-N} & \mbox{if $|(\xi_2,\xi_3)| \gg 1$} 
\end{cases}
$$
where $\varphi(x_2,x_3)$ is supported on $|(x_2,x_3)| \lesssim \lambda \delta$ and satisfies for any $N>0$
$$
| \nabla^k \varphi(x_2,x_3) | \lesssim (\lambda \delta)^{-k}, \qquad | \nabla^\ell \widehat{\varphi}(\xi_2,\xi_3) | \lesssim_N (\lambda \delta)^{2+\ell} \langle \lambda \delta (\xi_2,\xi_3) \rangle^{-N}.
$$

\medskip

\noindent
\textit{Step 2: Further reductions on the oscillatory integral.} We can now write
$$
\widehat{\varphi \dd \sigma_\lambda}(\xi)  = \lambda^2 m (\lambda \xi ) * \widehat{\varphi}(\xi).
$$
The case $|\xi| \lesssim 1$ is easily dealt with, therefore, we focus on $|\xi| \gg 1$ and replace $m$ by its leading order in the asymptotic expansion \eqref{expansionm}, since higher order terms are easier to estimate. This gives
$$
\lambda \int \frac{a_1(\lambda(\xi_1, \xi_2 - \xi_2',\xi_3 - \xi_3'))}{|(\xi_1, \xi_2 - \xi_2',\xi_3 - \xi_3')|} e^{i\lambda \Theta(\xi_1, \xi_2 - \xi_2',\xi_3 - \xi_3')} \widehat{\varphi}(\xi_2',\xi_3') \dd\xi_2' \dd\xi_3'
$$
By rapid decay of $\widehat{\varphi}$ on a scale $(\lambda \delta)^{-1} < \lambda^{-\kappa}$, it is harmless to add in the above integral a compactly supported cutoff function $\chi_1(\lambda^{\kappa/2}(\xi_2',\xi_3'))$. Furthermore, since $\frac{a_1(\lambda(\xi_1, \xi_2 - \xi_2',\xi_3 - \xi_3'))}{|(\xi_1, \xi_2 - \xi_2',\xi_3 - \xi_3')|}$ varies very slowly on the region where $|(\xi_2',\xi_3')| \lesssim 1$, it is inconsequential to replace this factor by $|\xi|^{-1}$. This results in the expression
$$
I = \frac{\lambda}{|\xi|} \int e^{i\lambda \Theta(\xi_1, \xi_2 - \xi_2',\xi_3 - \xi_3')} \chi_1(\lambda^{\kappa/2}(\xi_2',\xi_3')) \widehat{\varphi}(\xi_2',\xi_3') \dd\xi_2' \dd\xi_3'
$$
If $|(\xi_2,\xi_3)| \lesssim 1$, this integral is estimated immediately by

$$
|I| \lesssim \frac{\lambda}{|\xi|} \int | \widehat{\varphi}(\xi_2',\xi_3') | \dd \xi_2' \dd \xi_3'\lesssim \frac{\lambda}{|\xi|}.
$$
As $|\hat{\varphi}(\xi)|\lesssim (\la\delta)^2$ and it is essentially supported where $|(\xi_2',\xi_3')|\lesssim (\la\delta)^{-1}$.

\medskip

\noindent
\textit{Step 3: The regime of strong oscillation.} 
If $|(\xi_2,\xi_3)| \gg 1$, we view $I$ as an oscillatory integral with large parameter $\lambda$ and phase 
$$
\Phi(\xi,\xi') = \Theta(\xi_1,\xi_2-\xi_2',\xi_3-\xi_3').
$$
Assuming $|\xi_3'| < |\xi_2'| \ll 1$ and $1 \ll |\xi_2| \ll |\xi|$, and using homogeneity of $\Theta$ and strict convexity of its level sets, we see that
$$
\left| \partial_{\xi_2'}^k \left(\frac{1}{\partial_{\xi_2'} \Phi(\xi,\xi')} \right)\right| \lesssim \frac{|\xi|}{|\xi_2|^{1+k}}.
$$
Since $\xi_2$ and $\xi_3$ play symmetrical roles, we can assume without loss of generality that $|\xi_2| \gtrsim |\xi_3|$.
Integrating by parts once in $I$ results in the estimate
\begin{align*}
|I| & = \left| \frac{\lambda}{|\xi|} \int e^{i\lambda \Phi(\xi,\xi')} \partial_{\xi_2'} \left[ \frac{\chi_1(\lambda^{\kappa/2}(\xi_2',\xi_3'))  \widehat{\varphi}(\xi_2',\xi_3') }{\lambda \partial_{\xi_2'} \Phi(\xi,\xi')} \right] \dd\xi_2' \dd\xi_3' \right| \\
& \lesssim \frac{\lambda}{|\xi|} \int \left[ \left| \partial_{\xi_2'} \left( \frac{1}{\lambda \partial_{\xi_2'} \Phi} \right) \widehat{\varphi}(\xi_2',\xi_3') \right| + \left| \frac{1}{\lambda \partial_{\xi_2'} \Phi} \partial_{\xi_2'} \widehat{\varphi}(\xi_2',\xi_3') \right| \right] \dd \xi_2' \dd \xi_3' \\
& \qquad \qquad + \{ \mbox{easier term} \} \\
& \lesssim \frac{\lambda}{|\xi|} \frac{|\xi|}{\lambda |\xi_2|^2} + \frac{\lambda}{|\xi|} \cdot \frac{|\xi|}{\lambda |\xi_2|} \cdot \lambda \delta \lesssim \frac{\lambda \delta}{|\xi_2|}.
\end{align*}
We now iterate this procedure; it turns out that the worst always occurs when derivatives hit $\widehat{\varphi}$, which gives the estimate
\begin{align*}
|I| & = \left| \frac{\lambda}{|\xi|} \int e^{i\lambda \Phi(\xi,\xi')} \left( \partial_{\xi_2'} \frac{1}{\lambda \partial_{\xi_2'} \Phi} \right)^N \left[ \chi_1(\lambda^{\kappa/2}(\xi_2',\xi_3')) \widehat{\varphi}(\xi_2',\xi_3') \right] \dd\xi_2' \dd\xi_3' \right| \\
& \lesssim \left| \frac{\lambda}{|\xi|} \int e^{i\lambda \Phi(\xi,\xi')} \left(  \frac{1}{\lambda \partial_{\xi_2'} \Phi} \right)^N \chi_1(\lambda^{\kappa/2}(\xi_2',\xi_3')) (\partial_{\xi_2'})^N \widehat{\varphi}(\xi_2',\xi_3')\dd\xi_2' \dd\xi_3' \right|  \\
& \qquad + \{ \mbox{easier terms} \} \\
& \lesssim \frac{\lambda}{|\xi|} \left( \frac{|\xi|}{\lambda |\xi_2|} \right)^N (\lambda \delta)^N \lesssim \frac{\lambda \delta}{|\xi_2|^N},
\end{align*}
where we used $|\xi| \lesssim \delta^{-1}$ in the last inequality.
\end{proof}

\begin{proof}[Proof of \cref{KNEstThrm}] \textit{Decomposition of the kernel.} Denoting $P^\nu_{\lambda,\delta}(x)$ for the kernel of $P^\nu_{\lambda,\delta}$, we proceed as in Section \ref{sec:exponential} and write it as
\begin{align*}
P_{\lambda,\delta}^\nu(x) & = \sum_{k \in \mathbb{Z}^3}\chi_{\lambda,\delta}^\nu(k) e^{2\pi i k \cdot x} = \sum_{m \in \mathbb{Z}^3} \widehat{\chi_{\lambda,\delta}^\nu} (m - x)\\
& = \sum_{m \in \mathbb{Z}^3} \varphi(m - x) \widehat{ \chi_{\lambda,\delta}^\nu } (m - x) + \sum_{m \in \mathbb{Z}^3} \sum_{M \in 2^\mathbb{N}} \psi \left( \frac{m-x}{M} \right) \widehat{  \chi_{\lambda,\delta}^\nu} (m - x) \\
& = P_{\lambda,\delta}^{\nu,0}(x) + \sum_{\substack{M \in 2^\mathbb{N}\\M\delta\lesssim 1}} P_{\lambda,\delta}^{\nu,M}(x).
\end{align*}

\medskip
\noindent
\textit{Bound for the leading term $P^{\nu,0}_{\lambda,\delta}$.} By locality and the Stein-Tomas theorem, the operator $P_{\lambda,\delta}^\nu$ has operator norm $\lesssim \lambda^{2 - \frac 6 p} \delta$ from $L^{p'}$ to $L^p$ if $p \geq 4$.

\medskip
\noindent
\textit{Bound for the noise $\sum_{M \geq 1} P^{\nu,M}_{\lambda,\delta}$.} The $L^{p'} \to L^p$ operator norm of $P^\nu_{\lambda,\delta}$ will be estimated by interpolating between $L^2 \to L^2$ and $L^1 \to L^\infty$. The former is not hard: just like in Section \ref{sec:exponential},
$$
\| P^{\nu,M}_{\lambda,\delta} \|_{L^2 \to L^2} = \| M^3 \widehat{\psi}(M \cdot) * \chi_{\lambda,\delta}^\nu \|_{L^\infty} \lesssim M \delta \qquad \mbox{if $M \lesssim \delta^{-1}$}.
$$

We now claim that
$$
\| P^{\nu,M}_{\lambda,\delta} \|_{L^1 \to L^\infty} = \left\| P_{\lambda,\delta}^{\nu,M}(x) \right\|_{L^\infty} \lesssim \lambda \delta.
$$
Indeed, denoting $n$ for the unit normal to $\nu$ and applying Lemma \ref{coccinelle} below,
\begin{align*}
|P_{\lambda,\delta}^{\nu,M}(x)| & \lesssim \sum_m \left| \psi \left( \frac{m-x}{M} \right) \widehat{  \chi_{\lambda,\delta}^\nu} (m - x) \right| \\
& \lesssim \sum_{\substack{|m-x| \sim M \\ |(m_1-x_1,m_2-x_2)| \lesssim 1}} \lambda \delta \langle m-x \rangle^{-1} \langle P_n^\perp (m-x) \rangle^{-N} \lesssim \lambda \delta
\end{align*}
(recall that $P_n^\perp x = x - (x\cdot n) n$).
Interpolating between these two bounds gives
$$
\| P^{\nu,M}_{\lambda,\delta} \|_{L^{p'} \to L^p} \lesssim (M \delta)^{\frac 2p} (\lambda \delta)^{1 - \frac 2p},
$$
so that
$$
\sum_{1 \leq M \lesssim \delta^{-1}} \| P^{\nu,M}_{\lambda,\delta} \|_{L^{p'} \to L^p} \lesssim (\lambda \delta)^{1 - \frac 2p},
$$
which concludes the proof after the usual $TT^*$ argument.
\end{proof}

We have therefore reduced matters to proving \cref{fourmi}. Before giving the proof, a few comments might be helpful. The basic idea is to obtain a gain at $p=4$ and interpolate with $p=\infty$ to prove the result above. The choice of $p=4$ is not related to $L^4$ being a critical space for decoupling in dimension $3$; rather, it is chosen as an even integer which allows a connection to a lattice points counting problems, where transversality can be used very effectively. We also make use of the fact that there are very few caps in $P_{\operatorname{max}}$ so that any application of the triangle inequality only incurs a marginal loss.

\begin{proof}[Proof of \cref{fourmi}] We shall prove the result by proving $L^4$ and $L^\infty$ estimates and interpolating.

\medskip
\noindent\textit{The $L^4$ bound.} Let $\theta,\theta'$ be caps of diameter $(\la\delta)^{1/2}$ that contribute to the operator $P_{\operatorname{max}}$, that is, they have near the maximum number of lattice points. We shall write $\theta\subset \nu$ to mean that $\theta$ is contained in the cap associated to $\nu$. By the triangle inequality, we may write
\begin{equation*}
\| (P^{\nu}_{\operatorname{max}} f P^{\nu'}_{\operatorname{max}} f)^{1/2}\|_{L^4}\lesssim (\lambda \delta)^{3\alpha}\sup_{\substack{\theta \subset \nu \\ \theta' \subset \nu'}} \| (P_{\theta} f P_{\theta'}f)^{1/2}\|_{L^4},
\end{equation*}
where we used the bound \eqref{boundNmax} on $\mathcal{N}_{\operatorname{max}}$.

Because $\theta,\theta'$ have near the maximum number of lattice points, the lattice points must be contained on $2$ dimensional affine spaces. After applications of Parseval's identity and the Cauchy-Schwarz inequality, one gets the following bound
\begin{equation*}
 \| (P_{\theta} f P_{\theta'}f)^{1/2} \|_{L^4} \leq M^{1/4}\| P_{\theta}f \|_{L^2}^{1/2} \| P_{\theta'}f \|_{L^2}^{1/2},
\end{equation*}
where
\begin{equation*}
M = \sup_{k \in \mathbb{Z}^3} \# \{(a,a')\in [ \theta \cap \mathbb{Z}^3] \times [ \theta' \cap \mathbb{Z}^3] ,\; a + a' = k\}.
\end{equation*}

As $\nu,\nu'$ are transverse by at least $\delta$, the normal vectors of the affine spaces that contain the lattice points in $\theta,\theta'$ must also be transverse. Indeed, by \eqref{eq:rank2_geom_of_nos},
$$
\left| \frac{A x_\theta}{|A x_\theta|} \times \frac{w}{|w|} \right| \lesssim \frac{\delta^{\frac{1}{2}-\alpha}}{\lambda^{\frac{1}{2}-\alpha}},
$$
where $\frac{A x_\theta}{|A x_\theta|}$ is the unit normal to the cap and $\frac{w}{|w|}$ is the unit normal to $\Lambda_\theta$. By choosing $\alpha$ sufficiently small, we have $\lambda^{-\frac{1}{2}+\alpha}\delta^{\frac{1}{2}-\alpha} \ll \delta$, hence transversality holds as caps of size $\la\delta$ have $\delta$ separated normal vectors. 

Let us denote the normal vectors of the affine spaces that contain lattice points as $n_\theta$ and $n_\theta'$. The set of $(a,a')$ in the above right-hand side is contained in the set of $(x,y) \in \mathbb{R}^3 \times \mathbb{R}^3$ such that
$$
x \cdot n_\theta =0, \qquad y \cdot n_{\theta'} = 0, \qquad x+y = k.
$$
These three equations define an affine subspace of dimension $1$. Therefore, the number of $(a,a') \in [ \theta \cap \mathbb{Z}^3] \times [ \theta' \cap \mathbb{Z}^3]$ such that $a + a'=k$ can be bounded by the diameter of $\theta$, in other words
\begin{equation*}
    M\lesssim (\la\delta)^{1/2}.
\end{equation*}
Combining the above estimates gives
\begin{equation*}
\| (P^{\nu}_{\operatorname{max}}f P^{\nu'}_{\operatorname{max}} f)^{1/2}\|_{L^4} \lesssim (\la\delta)^{\frac{1}{8} + 3 \alpha} \| P^{\nu}_{\operatorname{max}} f \|_{L^2}^{1/2} \| P^{\nu'}_{\operatorname{max}} f\|_{L^2}^{1/2}.
\end{equation*}

\medskip \noindent \textit{The $L^\infty$ bound} The Fourier support of $P_\nu P_{\operatorname{max}} f$ comprises at most $\mathcal{N}_{\operatorname{max}}$ caps of diameter $(\lambda \delta)^{1/2}$. Therefore, it contains at most $(\lambda \delta)^{1 + 3 \alpha}$ lattice points. By the Cauchy-Schwarz inequality, we therefore have
\begin{align*}
\| (P^{\nu}_{\operatorname{max}} f P^{\nu'}_{\operatorname{max}} f)^{1/2} \|_{L^\infty} & \lesssim \| P^{\nu}_{\operatorname{max}} f \|_{L^\infty}^{1/2} \| P^{\nu'}_{\operatorname{max}} f \|_{L^\infty}^{1/2} \\
& \lesssim (\lambda \delta)^{\frac{1}{2} + \frac{3\alpha}{2}} \| P^{\nu}_{\operatorname{max}} f \|_{L^2}^{1/2} \| P^{\nu'}_{\operatorname{max}} f\|_{L^2}^{1/2}.
\end{align*}
Interpolating with the $L^4$ bound gives the desired estimate and proves \cref{fourmi}, which completes the proof of \cref{PmanyThm}.
\end{proof}

\section{Proof of the main theorems} \label{sec:proofmaintheorem}
In this section we prove \cref{mainresult} and \cref{generaltorithm} by combining the results of the previous sections.

\begin{proof} 
    \noindent\textit{Decoupling and the result of \cite{Pezzi}}. The cited work proves the sharp version of the conjecture when 
    \begin{equation}
        \label{eq:pezzi}
    \delta>\la^{-1+\kappa},
    \qquad p\leq 4.
    \end{equation}
    We assume $p>4$ for the next few steps in the proof. Additionally, we assume the bound \eqref{eq:expected_number_of_points}.

\medskip
\noindent\textit{Decomposition}.
We have decomposed our operator into
$$
P_{\lambda,\delta} = \underbrace{\sum_{2^s \lesssim \lambda^2 \delta + 1} P^s_{\lambda,\delta}}_{\displaystyle P_{\operatorname{few}}} + \underbrace{\sum_{2^s \gg \lambda^2 \delta + 1} P^s_{\lambda,\delta}}_{\displaystyle P_{\operatorname{many}}}.
$$
Bounds will be proven on the entire operator $P_{\la,\delta}$ or on  $P_{\operatorname{few}}$ and $P_{\operatorname{many}}$. By the triangle inequality, the conjecture is proven on the largest range where the relevant bounds hold for $P_{\operatorname{few}}$ and $P_{\operatorname{many}}$.

\medskip
    \noindent\textit{Contribution from \cref{expSumThm}}. Using exponential sum estimates we obtain a bound that will be useful when $p$ is large and $\delta$ is small. The sharp version of the conjecture is proven when
    \begin{equation}\label{introProofExpSumResults}
        \delta >\la^{-\frac{85p-358}{158p-128}}\quad
        \text{or}\quad \delta>\lambda^{-\frac{85}{158}},\, p=\infty.
    \end{equation}

    \medskip
    \noindent\textit{Contribution from \cref{cor:Pfewbound} and \cref{PfewThm}}. The case $\delta>\la^{\frac p2-3+\kappa}$ is \cref{cor:Pfewbound}. For the rest, we know from \cref{PfewThm} that $P_{\operatorname{few}}$ obeys the sharp conjecture if 
    \[
    \delta>\min\bigPara{\la^{\frac{p}{2}-3+\kappa},\la^{-\frac{85p-152}{158p-256}+\kappa}, \la^{-\frac{9p-224}{444-158p}+\kappa}}.
    \]
    
    This can also be written
    \begin{equation}\label{introProofPfewBound}
    \delta>
    \begin{cases}
    \lambda^{\frac p2-3+\kappa}
    &4<p<\frac{389}{79}=4.92\dots,
    \\
    \la^{-\frac{9p-224}{444-158p}+\kappa}
    &\frac{389}{79}<p<\frac{529+ \sqrt{70017}}{158}=5.02\dots,
    \\
\la^{-\frac{85p-152}{158p-256}+\kappa}&
    p>\frac{529+ \sqrt{70017}}{158}.
    \end{cases}
    \end{equation}
    Note that the bound $\lambda^{\frac p2-3+\kappa}$ only dominates when $p<4.92...$
    
    \medskip
    \noindent\textit{Contribution from \cref{PmanyThm}}. Unwinding the bounds from \cref{PmanyThm} gives that $P_{\operatorname{many}}$ obeys the sharp version of the conjecture in the following range:
    \begin{equation}\label{introProofPmanyBound}
        \delta >
        \begin{cases}
            \la^{-\frac{316-27p}{104(p-2)}+\kappa}
            &4\leq p \leq \frac{4684}{963}=4.86\dots,
            \\
            \la^{-\frac{5948-1071p}{3852-547p}+\kappa}
            &\frac{4684}{963} \leq p \leq 5,
            \\
            \la^{-\frac{1}{2}+\kappa}
            &p\geq 5.
        \end{cases}
    \end{equation}

    \medskip
    \noindent\textit{Putting together the full theorem}
    When $p$ is small we have \eqref{eq:pezzi}.
    
    For intermediate $p$ we need both of the bounds \eqref{introProofPfewBound} and \eqref{introProofPmanyBound} to hold, in order to have results for both $P_{\operatorname{few}}$ and $P_{\operatorname{many}}$. As a result the terms $\lambda^{-\frac{85p-152}{158p-256}}$ and $\la^{-\frac{5948-1071p}{3852-547p}}$ will not appear in the final theorem as they are never the limiting factor.
    
    When $p$ is large we will instead use \cref{introProofExpSumResults} for the full projector. Putting these together gives the sharp conjecture when
    \begin{equation}
        \delta > 
        \begin{cases}
           \la^{-1+\kappa} & 2\leq p \leq 4,\\
           \max(
           \la^{-\frac{316-27p}{104(p-2)}+\kappa},
           \la^{\frac{p}{2}-3+\kappa}) & 4\leq p < \frac{389}{79},\\
           \la^{-\frac{9p-224}{444-158p}+\kappa} &\frac{389}{79}<p<5,\\
           \min(\la^{-\frac{1}{2}+\kappa},
           \la^{-\frac{85p-358}{158p-128}})&5<p\leq \infty,
        \end{cases}
    \end{equation}
    which is exactly the range in the theorem.

    There is a discontinuity at $p=5$ because of the behavior of caps with $\sim 1$ points, as discussed in \cref{rem:obstruction}.
    If we take $\kappa=0$ we get the result with an $\epsilon$-loss. This proves the main theorem.
    \end{proof}

\section{Extension to other dimensions} \label{sec:otherd}
The results of \cref{sec:epsilonremoval1} and \cref{sec:epsilonremoval2} extend to all $d\geq 2$ with small adjustment. In this section, we discuss how to do this and the consequences for the higher dimensional version of the conjecture considered in this work. 

We define $P_{\la,\delta}$ in arbitrary dimension $d$ in the obvious way. The caps we consider now have size $(\la\delta)^{1/2}\times (\la\delta)^{1/2}... \times (\la\delta)^{1/2}\times  \delta$. We similarly break up the projector into

\begin{equation*}
    P_{\la,\delta} = P_{\operatorname{few}}  + P_{\operatorname{med}}  + P_{\operatorname{max}}.
\end{equation*}

Here, $P_{\operatorname{few}}$ projects onto caps with $\lesssim(\la\delta)^{\frac{d-1}{2}}\delta+1$ lattice points, $P_{\operatorname{med}}$ projects onto caps that contain between $\sim(\la\delta)^{\frac{d-1}{2}}\delta+1$ and $\sim\la^{-\alpha}(\la\delta)^{\frac{d-1}{2}}$ lattice points, and $P_{\operatorname{max}}$ projects onto caps that contain $\gtrsim\la^{-\alpha}(\la\delta)^{\frac{d-1}{2}}$ lattice points. The parameter $\alpha$ is $\ll 1$.

Adapting the methods of \cref{sec:epsilonremoval1}, one can easily prove the following result.

\begin{prop}\label{heightSplittingGenDimProp}
    Let $d\geq 2$ and $\T{d}$ a general torus. Let $\kappa>0$ and $\delta > \la^{-\frac{d-1}{d+1}+\kappa}$. Then 

    \begin{equation*}\label{heightSplittingGenDimEqn}
        \lpnorm{P_{\operatorname{few}}f}{p}{\T{d}}\lesssim \la^{\frac{d-1}{2}-\frac{n}{p}}\delta^{\frac{1}{2}}\lpnorm{f}{2}{\T{d}}.
    \end{equation*}
\end{prop}

Note that when $d=3$, we recover the result of \cref{PfewProp}. The proof is the same as the $d=3$ case after making adjustments to the numerology based on the dimension, as we only use decoupling, which holds at $\frac{2(d+1)}{d-1}$ in any dimension, trivial exponential sum estimates, and interpolation. One could use the results of \cite{Guo} when $d\geq 4$ to get stronger results like we have done here for $d=3$, but we do not pursue this direction.

Extending \cref{propPS} for $P_{\operatorname{max}}$ requires more work. We must always select $\alpha$ small enough so that $P_{\operatorname{max}}$ contains only caps whose rank is $d-1$, although the proof forces such a small $\alpha$ regardless. In arbitrary $d$ we have the following generalization.

\begin{prop} \label{propPSgenD}
Let $p \geq 2$ and $\delta \gtrsim \lambda^{-1+\kappa}$, with $\kappa>0$. If $\alpha$  is selected to be small enough independently of $\la$, we have
\begin{equation*}
\| P_{\operatorname{max}} f \|_{L^p} \lesssim (\la^{\frac{d-1}{2}-\frac{3}{p}}\delta^{1/2}+(\la\delta)^{\frac{d-1}{2}(\frac{1}{2}-\frac{1}{p})})\| f \|_{L^2}.
\end{equation*}
\end{prop}

This is essentially Proposition 5.2 in \cite{Pezzi}. That result is stated for $\frac{2(d+1)}{d-1}$, although the proof, with minor adjustments, gives the result for all $p$.

To prove a generalization of \cref{PmanyThm} one also requires a generalization of \cref{fourmi}, the $L^4$ estimate for far interacting caps. The equivalent result in \cite{Pezzi} is stated and proved in any dimension. Here, in $d=3$, we considered the intersection of transverse planes in the proof of \cref{fourmi} and found them to be lines, which was enough of a gain to run the argument. In general, one considers the intersection of $d-1$ dimensional hyperplanes and finds them to be $d-2$ dimensional affine spaces, points when $d=2$, lines when $d=3$, planes when $d=4$, and so on, which gives enough of a gain for the argument to close via interpolation in any dimension.

Recall $P^\nu_{\operatorname{max}} = P^\nu_{\la,\delta} P_{\operatorname{max}}$, where $P^\nu_{\la,\delta}$ is the projection onto a cap of size $\la\delta\times ... \times \la\delta \times \delta$. The collection of $\nu$ is a $\delta$ separated set. We have the following.

\begin{prop}\label{fourmiGenD}
    Let $d\geq 2$. If $|\nu-\nu'|\gtrsim \delta$, then 

    \begin{equation*}
        \lpnorm{(P^\nu_{\operatorname{max}}f P^{\nu'}_{\operatorname{max}}f)^{1/2}}{4}{\T{d}}\lesssim (\la\delta)^{\frac{d-2}{8}+d\alpha}\bigPara{\lpnorm{P^\nu_{\operatorname{max}}f}{2}{\T{d}} \|P^{\nu'}_{\operatorname{max}}f\|_{L^2(\T{d})}}^{1/2}.
    \end{equation*} 
\end{prop}
This, combined with \cref{heightSplittingGenDimProp} and \cref{propPSgenD}, implies \cref{d2EpsilonRemovalIntro} which improves Theorem 1.3 in \cite{DemeterGermain}. The proof follows the scheme outlined in \cref{sec:epsilonremoval1} and \cref{sec:epsilonremoval2}. Similarly, we may improve Theorem 6.1 in \cite{GermainMyerson1} to remove $\epsilon$-losses away from the \say{boundary} of the results in that theorem to get \cref{genDEpsRemovalIntro} here. The proof of this uses \cref{propPSgenD} and \cref{fourmiGenD}, along with the bounds for $P_{med}$ found in \cite{GermainMyerson1}. 

\appendix

\section{Additive combinatorics}
\label{sec:additive}

Viewing the problem of $L^p$ boundedness of spectral projectors on the Fourier side, it is well known that it becomes a problem in additive combinatorics if $p$ is an even integer.

In this appendix, we point out a simple reformulation of this additive combinatorics question, but do not attempt to exploit it any further. The statements here can be generalized to any dimension.

\subsection{Additive energy}
Let $\mathcal{A}_{\lambda,\delta}$ be the integer points in the upper part of the corona of inner radius $\lambda-\delta$ and outer radius $\lambda+\delta$.
$$
\mathcal{A}_{\lambda,\delta}
= \{ n \in \mathbb{Z}^3, \, n_1 > |n_2| + \dots + |n_d|, \,\lambda - \delta < |n| < \lambda + \delta \}.
$$

Let $A\subset \Z{3}$ and $|A|$ be its cardinality. For $p$ and even integer, we define the additive energy of $A$ to be
$$E_{\frac{p}{2}}(A) = |\{(a_1,...,a_p)\in A^{p}: a_1+ ... +a_{\frac{p}{2}}=a_{\frac{p}{2}+1}+...+a_{p}\}|.$$
The additive energy measures the additive structure of $A$. We immediately have the bound $|A|^{\frac{p}{2}}\leq E(A)\leq |A|^{p-1}$. We note that some authors define the additive energy with a normalization involving $|A|$.
A conjecture about the additive structure of $\mathcal{A}_{\lambda,\delta}$ was stated when $d=2$ in \cite{DemeterGermain}. We now state a version for $d=3$.

\begin{conj}\label{additiveConj}
    For $\mathcal{A}_{\lambda,\delta}$, we have

    $$E_{\frac{p}{2}}(\mathcal{A}_{\lambda,\delta})\lesssim \la^p\delta^{\frac{p}{2}}+\la^{2p-3}\delta^{p}.$$
\end{conj}

The first term on the right hand side comes from considering the universal lower bounds on additive energies in the context of \cref{eq:expected_number_of_points}. The second comes from the fact all lattice points in $\mathcal{A}_{\lambda,\delta}$ are contained in a ball of radius $\la$, so the sum of $\frac{p}{2}$ elements of $\mathcal{A}_{\lambda,\delta}$ are spread out in a ball of radius $\lesssim_p\la$ which contains $\lesssim_p \la^3$ lattice points.

\subsection{A counting problem implying Conjecture \ref{conjA}}

If $p=2r$ is an even integer, the conjecture
$$
\| P_{\lambda,\delta} f \|_{L^p} \lesssim \left[ (\lambda \delta)^{ \frac{1}{2} - \frac{1}{p}} + \lambda^{1 - \frac{3}{p}} \delta^{\frac 12} \right] \| f \|_{L^2}
$$
can be written in terms of the Fourier coefficients of $f$ as
\begin{equation}
\label{conjfouriercoeff}
\sum_{k \in \mathbb{Z}^3} \left| \sum_{ \substack{n_1,\dots,n_r \in \mathcal{A}_{\lambda,\delta} \\ n_1 + \dots + n_r =k  }} a_{n_1} \dots a_{n_r} \right|^2 \lesssim \left[ (\lambda \delta)^{ \frac{p}{2} - 1} + \lambda^{p - 3} \delta^{\frac p2} \right] \| a_n \|_{\ell^2}^{p}.
\end{equation}
Note that
$$
\left[ (\lambda \delta)^{ \frac{p}{2} - 1} + \lambda^{p - 3} \delta^{\frac p2} \right] = 
\begin{cases}
\lambda \delta & \mbox{if $p=4$} \\
\lambda^{p - 3} \delta^{\frac p2} & \mbox{if $p=2r$, $r \geq 3$}.
\end{cases}
$$

Denoting
$$
Z_{\lambda,\delta} = \max_{k \in \mathbb{Z}^3} \# \left\{ (n_1,\dots,n_r) \in \mathcal{A}_{\lambda,\delta}, \; n_1 + \dots + n_r = k \right\}
$$
and applying the Cauchy-Schwarz inequality to the left-hand side of \eqref{conjfouriercoeff}, we obtain
\begin{equation}
\sum_{k \in \mathbb{Z}^3} \left| \sum_{ \substack{n_1,\dots,n_r \in \mathcal{A}_{\lambda,\delta} \\ n_1 + \dots + n_r =k  }} a_{n_1} \dots a_{n_r} \right|^2 \leq Z_{\lambda,\delta} \sum_{k} \sum_{ \substack{n_1,\dots,n_r \in \mathcal{A}_{\lambda,\delta} \\ n_1 + \dots + n_r =k  }} |a_{n_1} \dots a_{n_r}|^2 \leq Z_{\lambda,\delta} \| a_n \|_{\ell^2}^p.
\end{equation}

Therefore, the conjecture would be a consequence of the following inequality
$$
Z_{\lambda,\delta} \lesssim \begin{cases}
\lambda \delta & \mbox{if $p=4$} \\
\lambda^{p - 3} \delta^{\frac p2} & \mbox{if $p=2r$, $r \geq 3$}.
\end{cases}
$$
This bound aligns with \cref{additiveConj} when $p=2r$ and $p\geq 3$. The bound at $p=4$ comes from an additive subset of $\mathcal{A}_{\lambda,\delta}$. Namely, considering a cap with $\sim \la\delta$ lattice points gives the lower bound $\la\delta$ for $Z_{\la,\delta}$ as these lattice points line on a hyperplane. This is in alignment with \cref{conjA} as the dominant term when $p=4$ corresponds to a function whose fequency is supported in a single cap. 
\bibliographystyle{plain}
\bibliography{references}

@article{PezziSharpEFncn,
      title={Sharp Eigenfunction Bounds on the Torus for large $p$}, 
      author={Daniel Pezzi},
      JOURNAL = {arXiv preprint arXiv:2603.10927},
      Year={2026}, 
}

@article{Germain,
  title={{$L^2$} to {$L^p$} bounds for spectral projectors on thin intervals in {R}iemannian manifolds},
  author={Germain, Pierre},
  journal={arXiv preprint arXiv:2306.16981},
  year={2023}
}

@article {Demeter,
    AUTHOR = {Demeter, Ciprian},
     TITLE = {Beyond canonical decoupling},
   JOURNAL = {J. Fourier Anal. Appl.},
  FJOURNAL = {The Journal of Fourier Analysis and Applications},
    VOLUME = {32},
      YEAR = {2026},
    NUMBER = {1},
     PAGES = {Paper No. 22, 25},
      ISSN = {1069-5869,1531-5851},
   MRCLASS = {42B10},
  MRNUMBER = {5024266},
       DOI = {10.1007/s00041-025-10219-x},
       URL = {https://doi.org/10.1007/s00041-025-10219-x},
}

@article{Bourgain97,
  title={Analysis results and problems related to lattice points on surfaces},
  author={Bourgain, Jean},
  journal={Contemporary Mathematics},
  volume={208},
  pages={85--110},
  year={1997},
  publisher={Providence, RI: American Mathematical Society}
}

@article{BSSY,
  title={On {$L^p$} resolvent estimates and the density of eigenvalues for compact {R}iemannian manifolds},
  author={Bourgain, Jean and Shao, Peng and Sogge, Christopher D and Yao, Xiaohua},
  journal={Communications in Mathematical Physics},
  volume={333},
  pages={1483--1527},
  year={2015},
  publisher={Springer}
}

@book{GrahamKolesnik, place={Cambridge}, series={London Mathematical Society Lecture Note Series}, title={Van der Corput’s Method of Exponential Sums}, publisher={Cambridge University Press}, author={Graham, S. W. and Kolesnik, Grigori}, year={1991}, collection={London Mathematical Society Lecture Note Series}}

@article{BlairHuangSogge,
  title={Improved spectral projection estimates},
  author={Blair, Matthew D and Huang, Xiaoqi and Sogge, Christopher D},
  journal={J. Eur. Math. Soc. (2024), published online first},
  year={2024}
}

@article{Bourgain1,
  title={Eigenfunction bounds for the {L}aplacian on the n-torus},
  author={Bourgain, Jean},
  journal={International Mathematics Research Notices},
  volume={1993},
  number={3},
  pages={61--66},
  year={1993},
  publisher={Hindawi Publishing Corporation}
}

@article{BourgainDemeter1,
  title={Improved estimates for the discrete {F}ourier restriction to the higher dimensional sphere},
  author={Bourgain, Jean and Demeter, Ciprian},
  journal={Illinois Journal of Mathematics},
  volume={57},
  number={1},
  pages={213--227},
  year={2013},
  publisher={Duke University Press}
}

@article{BourgainDemeter3,
  title={The proof of the {$L^2$} decoupling conjecture},
  author={Bourgain, Jean and Demeter, Ciprian},
  journal={Annals of mathematics},
  pages={351--389},
  year={2015},
  publisher={JSTOR}
}

@article{ChamizoIwaniec,
  title={On the sphere problem},
  author={Chamizo, Fernando and Iwaniec, Henryk},
  journal={Revista matem{\'a}tica iberoamericana},
  volume={11},
  number={2},
  pages={417--429},
  year={1995}
}

@article {DemeterGermain,
    AUTHOR = {Demeter, Ciprian and Germain, Pierre},
     TITLE = {{$L^2$} to {$L^p$} bounds for spectral projectors on the
              {E}uclidean two-dimensional torus},
   JOURNAL = {Proc. Edinb. Math. Soc. (2)},
  FJOURNAL = {Proceedings of the Edinburgh Mathematical Society. Series II},
    VOLUME = {67},
      YEAR = {2024},
    NUMBER = {2},
     PAGES = {431--459},
      ISSN = {0013-0915,1464-3839},
   MRCLASS = {42B15 (11L07 11P21)},
  MRNUMBER = {4746486},
       DOI = {10.1017/S0013091524000099},
       URL = {https://doi.org/10.1017/S0013091524000099},
}

@inproceedings{GermainMyerson1,
  title={Bounds for spectral projectors on tori},
  author={Germain, Pierre and Myerson, Simon L Rydin},
  booktitle={Forum of Mathematics, Sigma},
  volume={10},
  year={2022},
  organization={Cambridge University Press}
}

@article{Huxley2,
  title={Exponential sums and lattice points III},
  author={Huxley, Martin N},
  journal={Proceedings of the London Mathematical Society},
  volume={87},
  number={3},
  pages={591--609},
  year={2003},
  publisher={Oxford University Press}
}

@article{Guo,
title={On lattice points in large convex bodies},
author={Guo, Jingwei},
journal={Acta Arithmetica},
volume={151},
pages={83--108},
year={2012}
}

@article{Hickman,
  title={Uniform {$L^p$} Resolvent Estimates on the Torus},
  author={Hickman, Jonathan},
  journal={Mathematics Research Reports},
  volume={1},
  pages={31--45},
  year={2020}
}

@article{mudgal,
author = {Mudgal, Akshat},
title = {Additive energies on spheres},
journal = {Journal of the London Mathematical Society},
volume = {106},
number = {4},
pages = {2927-2958},
doi = {https://doi.org/10.1112/jlms.12652},
url = {https://londmathsoc.onlinelibrary.wiley.com/doi/abs/10.1112/jlms.12652},
eprint = {https://londmathsoc.onlinelibrary.wiley.com/doi/pdf/10.1112/jlms.12652},
year = {2022}
}

@article{Muller,
  title={Lattice points in large convex bodies},
  author={M{\"u}ller, Wolfgang},
  journal={Monatshefte f{\"u}r Mathematik},
  volume={128},
  number={4},
  pages={315--330},
  year={1999},
  publisher={Springer}
}

@article {Pezzi,
    AUTHOR = {Pezzi, Daniel},
     TITLE = {Sharp spectral projection estimates for the torus at {$p =
              \frac{2(n+1)}{n-1}$}},
   JOURNAL = {J. Geom. Anal.},
  FJOURNAL = {Journal of Geometric Analysis},
    VOLUME = {35},
      YEAR = {2025},
    NUMBER = {1},
     PAGES = {Paper No. 10, 30},
      ISSN = {1050-6926,1559-002X},
   MRCLASS = {42A45 (11H06 11L07 11P21 42B15)},
  MRNUMBER = {4818473},
       DOI = {10.1007/s12220-024-01839-8},
       URL = {https://doi.org/10.1007/s12220-024-01839-8},
}

@book{Sogge,
  title={{F}ourier integrals in classical analysis},
  author={Sogge, Christopher D},
  volume={210},
  year={2017},
  publisher={Cambridge University Press},
}

@book{Stein,
  title={Harmonic analysis: real-variable methods, orthogonality, and oscillatory integrals},
  author={Stein, Elias M},
  volume={3},
  year={1993},
  publisher={Princeton University Press}
}

@article{Tomas,
  title={A restriction theorem for the {F}ourier transform},
  author={Tomas, P.},
  journal={Bull. Amer. Math. Soc},
  volume={81},
  year={1975}
}

\end{document}